\documentclass{amsart}
\newcommand{\timestamp}{December 26, 2001}
\usepackage{amsthm}
\usepackage{epic,eepic}
\usepackage[dvips]{graphicx}

\makeatletter
\def\authornote#1{\relax\ignorespaces}

\newcommand{\bmath}[1]{\mbox{\mathversion{bold}$#1$}}

\newcommand{\C}{\bmath{C}}
\newcommand{\Z}{\bmath{Z}}
\newcommand{\R}{\bmath{R}}

\newcommand{\CP}{\bmath{C\!P}}
\newcommand{\SL}{\operatorname{SL}}
\renewcommand{\sl}{\operatorname{\mathfrak{sl}}}
\newcommand{\SU}{\operatorname{SU}}
\newcommand{\PSL}{\operatorname{PSL}}
\newcommand{\PSU}{\operatorname{PSU}}
\newcommand{\cmcone}{\mbox{\rm CMC}-$1$}
\newcommand{\TA}{\operatorname{TA}}
\newcommand{\trans}{{}^t_{}\!}
\newcommand{\gO}{\mathbf{O}}
\newcommand{\gI}{\mathbf{I}}
\newcommand{\gGamma}{\bmath{\Gamma}}
\newcommand{\ord}{\operatorname{ord}}
\newcommand{\id}{\operatorname{id}}
\newcommand{\diag}{\operatorname{diag}}
\newcommand{\Hyp}{\mathcal{H}}
\newcommand{\metone}{\operatorname{Met}_1}
\newcommand{\typeP}{\operatorname{P}}
\newcommand{\typeN}{\operatorname{N}}
\newcommand{\Fl}{\mbox{$\mathcal F\!l$}}
   \newtheorem{theorem}{Theorem}[section]
   \newtheorem*{theorem*}{Theorem}
   \newtheorem{proposition}[theorem]{Proposition}
   \newtheorem{corollary}[theorem]{Corollary}
   \newtheorem{lemma}[theorem]{Lemma}
 \theoremstyle{definition}
   \newtheorem{definition}[theorem]{Definition}
   \newtheorem{example}[theorem]{Example}
 \theoremstyle{remark}
   \newtheorem{remark}[theorem]{Remark}
   \newtheorem*{remark*}{Remark}
\numberwithin{equation}{section}
\makeatother
\title[CMC-1 surfaces]{
   Period problems\\
   for mean curvature one surfaces in \boldmath$H^3$\\
   {\scriptsize
     (with application to surfaces of low total curvature)
   }
}
\author{Wayne Rossman}
\author{Masaaki Umehara}
\author{Kotaro Yamada}
\date{\timestamp}
\address[Rossman]{%
   Department of Mathematics, Faculty of Science,
   Kobe University,
   Rokko, Kobe 657-8501, Japan%
}
\email{wayne@math.kobe-u.ac.jp}
\address[Umehara]{%
   Department of Mathematics, Graduate School of Science,
   Osaka University,
   Osaka 560-0043, Japan%
}
\email{umehara@math.sci.osaka-u.ac.jp}
\address[Yamada]{%
   Faculty of Mathematics,
   Kyushu University 36, 6-10-1
   Hakozaki, Higashi-ku, Fukuoka 812-8581, Japan%
}
\email{kotaro@math.kyushu-u.ac.jp}
\begin{document}
\maketitle
\section{Introduction}
There is a wide body of knowledge about minimal surfaces in Euclidean 
$3$-space $\R^3$, and there is a canonical local isometric
correspondence (sometimes called the Lawson correspondence) between
minimal surfaces in $\R^3$ and \cmcone{} (constant mean curvature one)
surfaces in hyperbolic $3$-space $H^3$ (the complete simply-connected
$3$-manifold of constant sectional curvature $-1$).  
This has naturally led to the recent interest in and development of 
\cmcone{} surfaces in $H^3$ in the last decade.  
There are now many known examples, and it is a natural next step 
to classify all such surfaces with low total absolute curvature.  

By this canonical local isometric correspondence, 
minimal immersions in $\R^3$ are locally equivalent to \cmcone{}
immersions in $H^3$.  
But there are interesting differences between these two types of
immersions on the global level.  
There are period problems on non-simply-connected domains of the
immersions, 
which might be solved for one type of immersion but not the other.  
Solvability of the period problems is usually more likely in the $H^3$
case, leading to a wider variety of surfaces there.  
For example, a genus $1$ surface with finite total curvature and two 
embedded ends cannot exist as a minimal surface in $\R^3$, but it 
does exist as a \cmcone{} surface in $H^3$ \cite{rs}.  
And a genus $0$ surface with finite total curvature and two embedded ends 
exists as a minimal surface in $\R^3$ only if it is a surface of 
revolution, but it may exist as a \cmcone{} surface in $H^3$ without being 
a surface of revolution (see Example~\ref{exa:catenoid}).  
So there are many more possibilities for \cmcone{} surfaces in $H^3$
than there are for minimal surfaces in $\R^3$.  
This means that it is more difficult to classify \cmcone{} surfaces with
low total curvature in $H^3$.  

To find complete \cmcone{} surfaces in $H^3$ with low total curvature, 
we must first determine the meromorphic data in the Bryant
representation of the surfaces that can admit low total curvature, 
and then we must analyse when the parameters in the data can be adjusted
to close the period problems.  
Generally, finding the data is the easier step, and solving the period
problems is the more difficult step.  
As the period problems are generally the crux of the problem, 
we have chosen the title of this paper to reflect this.  

The total absolute curvature of a minimal surface in $\R^3$ is 
equal to the area of the image (counted with multiplicity) of the Gauss map 
of the surface, and complete minimal surfaces in $\R^3$ with total 
curvature at most $8\pi$ have been classified 
(see Lopez~\cite{Lopez} and also Table~\ref{tab:minimal}).
Furthermore, as the Gauss map of a complete conformally parametrized 
minimal surface is meromorphic, and has a well-defined limit at each end 
when the surface has finite total curvature, 
the area of the Gauss image must be an integer multiple of $4\pi$. 

However, unlike minimal surfaces in $\R^3$, when searching for \cmcone{}
surfaces in $H^3$ with low total absolute curvature, we have a choice of
two different Gauss maps: 
the {\em hyperbolic Gauss map\/} $G$ and the {\em secondary Gauss map\/}
$g$.  
So there are two ways to pose the question in $H^3$, with two very
different answers.
One way is to consider the true total absolute curvature, 
which is the area of the image of $g$, but since $g$ might not be
single-valued on the surface, the total curvature might not be an
integer multiple of $4\pi$, 
and this allows for many more possibilities.  
Furthermore, the Osserman inequality does not hold for the true total 
absolute curvature.  
The weaker Cohn-Vossen inequality is the best general lower bound for
true absolute total curvature (with equality never holding \cite{uy1}).  
So the true total absolute curvature is difficult to analyse, 
but it is important because of its clear geometric meaning.  

The second way is to study the area of the image of $G$, 
which we call the {\it dual\/} total absolute curvature, as it is the
true total curvature of the dual \cmcone{} surface (which we define in 
Section~\ref{sec:prelim}) in $H^3$.  
This way has the advantage that $G$ is single-valued on the surface, 
and so the dual total absolute curvature is always an integer multiple
of $4\pi$, like the case of minimal surfaces in $\R^3$.  
Furthermore, the dual total curvature satisfies not only the Cohn-Vossen
inequality, but also the Osserman inequality \cite{uy5,Yu2} 
(see also \eqref{eq:osserman} in Section~\ref{sec:prelim}).  
So the dual total curvature shares more properties with the total
curvature of minimal surfaces in $\R^3$, motivating our interest in it.  

We shall refer to the true total absolute curvature of a \cmcone{}
immersion $f\colon{}M\to H^3$ of a Riemann surface as $\TA(f)$,
and the dual total absolute curvature as $\TA(f^{\#})$.

We review the classification results for surfaces with $\TA(f)\leq 4\pi$
or $\TA(f^{\#})\leq 4\pi$ in Section~\ref{sec:four-pi}, which are 
results from \cite{ruy4} and \cite{ruy3}. 
An inequality for $\TA(f)$ stronger than the Cohn-Vossen inequality
\cite{ruy4} (for surfaces of genus zero with odd number of ends) 
is also introduced.
In Section~\ref{sec:prelim}, we review basic notions and terminology.
We introduce some important examples of \cmcone{} surfaces in
Section~\ref{sec:Added}.
Section~\ref{sec:dualcurvatureatmost8pi} is devoted to describing the
results in \cite{ruy3}, a partial classification of \cmcone{} surfaces
with $\TA(f^{\#})\leq 8\pi$.
In the final section \ref{sec:higher}, we introduce new results on
partial classification of \cmcone{} surfaces with $\TA(f)\leq 8\pi$.
Since the proofs of these results are more technical and delicate than
those of the results on $\TA(f^{\#})$, we include them in
Appendix~\ref{app:detailed}.
\section{The cases $\TA(f)$ or $\TA(f^\#) \leq 4\pi$, and a natural extension}
\label{sec:four-pi}

In \cite{ruy4} the following theorem was proven: 
\begin{theorem}\label{thm:four-pi}
 Let $f\colon{}M\to H^3$ be a complete \cmcone{} immersion
 of total absolute curvature $\TA(f) \leq 4\pi$.  Then $f$ is either 
 \begin{itemize}
 \item a horosphere {\rm(}Example~{\rm\ref{exa:horosphere})},
 \item an Enneper cousin {\rm(}Example~{\rm\ref{exa:enneper})},
 \item an embedded catenoid cousin 
       {\rm (}$0<l<1$, $\delta=1$ and $b=0$ in
       Example~{\rm\ref{exa:catenoid})}, 
 \item a finite $\delta$-fold covering of an embedded catenoid cousin
       {\rm (}$\delta\geq 2$, $0<l \leq 1/\delta$ and $b=0$ in
       Example~{\rm\ref{exa:catenoid})},
       or
 \item a warped catenoid cousin with injective secondary Gauss map
       {\rm (}$l=1$, $\delta\in\Z^+$ and $b>0$ in 
              Example~{\rm\ref{exa:catenoid})}.
 \end{itemize}
\end{theorem}

The horosphere is the only flat (and consequently totally umbilic)
\cmcone{} surface in $H^3$.  
The catenoid cousins are the only \cmcone{} surfaces of revolution
\cite{Bryant}.  
The Enneper cousins are isometric to minimal Enneper surfaces
\cite{Bryant}.
The warped catenoid cousins \cite{uy1,ruy3} are less well known
and are described more precisely in Section~\ref{sec:Added}, as well as
the other above three examples.  

Although this theorem is simply stated, for the reasons given in the 
introduction the proof is more delicate than it would be if the
condition $\TA(f)\le 4\pi$ is replaced with $\TA(f^\#)\le 4\pi$, or if
minimal surfaces in $\R^3$ with $\TA\le 4\pi$ are considered.  
\cmcone{} surfaces $f$ with $\TA(f^\#)\le 4\pi$ are classified in
Theorem~\ref{thm:four-pi-dual} below.  
It is well-known that the only complete minimal surfaces in $\R^3$ with
$\TA \leq 4\pi$ are the plane, the Enneper surface, and the catenoid
(see Table~\ref{tab:minimal}).  

We extend the above result in Section~\ref{sec:higher} to find an
inclusive list of possibilities for \cmcone{} surfaces with $\TA(f)\le
8\pi$, 
and we consider which possibilities we can classify or find examples
for, see Table~\ref{tab:ta-8pi}.  
(Minimal surfaces in $\R^3$ with $\TA\le 8\pi$ are classified by Lopez
\cite{Lopez}. See Table~\ref{tab:minimal}.)

For a complete \cmcone{} immersion $f$ in $H^3$, equality in the
Cohn-Vossen inequality never holds (\cite[Theorem~4.3]{uy1}).  
In particular, if $f$ is of genus $0$ with $n$ ends, then 
\begin{equation}\label{eq:cohn-vossen}
    \TA(f)>2\pi (n-2)\;.
\end{equation}
When $n=2$, the catenoid cousins show that \eqref{eq:cohn-vossen} is
sharp.  
However, we see from the above theorem that 
\[
   \TA(f)>4\pi \quad\text{for}\quad  n=3 \; , 
\]
which is stronger than the Cohn-Vossen inequality \eqref{eq:cohn-vossen}.
The following theorem, which extends the above theorem and is 
proven in \cite{ruy4}, gives a sharper inequality than the 
Cohn-Vossen inequality when $n$ is any odd integer:  
\begin{theorem}\label{thm:odd}
  Let $f\colon{}\C\cup\{\infty\}\setminus\{p_1,\dots,p_{2m+1}\}\to H^3$
  be a complete conformal genus $0$ \cmcone{} immersion with $2m+1$ ends, 
  $m \in \Z^+$.  Then 
\[
  \TA(f)\ge 4\pi m\;.
\]
\end{theorem}

\begin{remark*}
 When $m=1$, we know that
 the lower bound $4\pi$ in the theorem is sharp 
 (see Example~\ref{exa:trinoid}).
 However, we do not know if it is sharp for general $m$.  
 For genus $0$ \cmcone{} surfaces with an even number $n \geq 4$ of ends, 
 it is still an open question whether 
 there exists any stronger lower bounds than that of the 
 Cohn-Vossen inequality.  
 It should be remarked that in Section~\ref{sec:Added} we have 
 numerical examples with $n=4$ whose total absolute 
 curvature tends to $4\pi$.
\end{remark*}

For the case of $\TA(f^\#)$, the following theorem was proven in
\cite{ruy3}:
\begin{theorem}\label{thm:four-pi-dual}
  A complete \cmcone{} immersion $f$ with 
  $\TA(f^{\#})\le 4\pi$ is congruent to one of the following{\rm :}
\begin{enumerate}
 \item a horosphere {\rm(}Example~{\rm\ref{exa:horosphere})}, 
 \item an Enneper cousin dual {\rm(}Example~{\rm\ref{exa:enneper})},  
 \item a catenoid cousin {\rm(}$\delta=1$, $l\neq 1$ and $b=0$ in
       Example~{\rm\ref{exa:catenoid})}, or 
 \item a warped catenoid cousin with embedded ends and injective
       hyperbolic Gauss map
       {\rm(}$\delta=1$, $l\in\Z$, $l\geq 2$ and $b>0$ in
       Example~{\rm\ref{exa:catenoid})}. 
\end{enumerate}
\end{theorem}

\section{Preliminaries}
\label{sec:prelim}

Before we can state any results for the cases of higher $\TA(f)$ and
higher $\TA(f^\#)$, we must give some preliminaries here.  

Let $f\colon{}M\to H^3$ be a conformal \cmcone{} immersion of a Riemann 
surface $M$ into $H^3$.  
Let $ds^2$, $dA$ and $K$ denote the induced metric, induced area element
and Gaussian curvature, respectively.  
Then $K \leq 0$ and $d\sigma^2:=(-K)\,ds^2$ is a conformal pseudometric
of constant curvature $1$ on $M$.  
We call this pseudometric's developing map $g\colon{}\widetilde
M(:=\text{the universal cover of $M$})\to \CP^1=\C\cup\{\infty\}$ 
the {\em secondary Gauss map\/} of $f$.  
Namely, $g$ is a conformal map so that its pull-back of the 
Fubini-Study metric of $\CP^1$ equals $d\sigma^2$:
\begin{equation}\label{eq:pseudo}
    d\sigma^2 = (-K)\,ds^2 = \frac{4\,dg\,d\bar g}{(1+g\bar g)^2}\;.
\end{equation}
Such a map $g$ is determined by $d\sigma^2$ uniquely up to the change
\begin{equation}\label{eq:g-change}
   g\mapsto a\star g :=\frac{a_{11}g + a_{12}}{a_{21}g +a_{22}}\;,
    \qquad
   a=\begin{pmatrix}
        a_{11} & a_{12} \\
        a_{21} & a_{22}
     \end{pmatrix}\in\SU(2)\;.
\end{equation}
Since $d\sigma^2$ is invariant under the deck transformation group
$\pi_1(M)$, there is a representation
\begin{equation}\label{eq:g-repr}
    \rho_g\colon{}\pi_1(M)\longrightarrow  \PSU(2)
    \quad\text{such that}\quad
    g\circ \tau^{-1}  = \rho_g(\tau)\star g \quad \bigl(\tau\in\pi_1(M)\bigr)\;,
\end{equation}
where $\PSU(2)=\SU(2)/\{\pm\id\}$.
The metric $d\sigma^2$ is called {\em reducible\/} if the image 
of $\rho_g$ can be diagonalized simultaneously, and is called 
{\em irreducible\/} otherwise.  
In the case $d\sigma^2$ is reducible,
we call it is {\em $\Hyp^3$-reducible\/} if the image of $\rho_g$ is the
identity, and is called 
{\em $\Hyp^1$-reducible\/} otherwise.
We call a \cmcone{} immersion $f\colon{}M\to H^3$ $\Hyp^1$-reducible
(resp.\ $\Hyp^3$-reducible) if the corresponding pseudometric
$d\sigma^2$ is $\Hyp^1$-reducible (resp.\ $\Hyp^3$-reducible).
For details on reducibility, see Section~\ref{sec:red}.

In addition to $g$, two other holomorphic invariants $G$ and $Q$ are closely 
related to geometric properties of \cmcone{} surfaces.  
The {\em hyperbolic Gauss map\/} $G\colon{}M\to \CP^1$ is holomorphic
and is defined geometrically by identifying the ideal boundary of $H^3$
with $\CP^1$: 
$G(p)$ is the asymptotic class of the normal geodesic of $f(M)$ starting
at $f(p)$ and oriented in the mean curvature vector's direction.  
The {\em Hopf differential} $Q$ is a holomorphic symmetric
$2$-differential on $M$ such that $-Q$ is the $(2,0)$-part of the
complexified second fundamental form.
The Gauss equation implies 
\begin{equation}\label{eq:metrics-rel}
  ds^2 \cdot d\sigma^2 = 4 \, Q \cdot\overline Q\;,
\end{equation}
where $\cdot$ means the symmetric product.  
Moreover, these invariants are related by 
\begin{equation}\label{eq:schwarz}
   S(g)-S(G) = 2 Q\;,
\end{equation}
where $S(\cdot)$ denotes the Schwarzian derivative: 
\[
   S(h):=
    \left[
        \left(\frac{h''}{h'}\right)'-
        \frac{1}{2}\left(\frac{h''}{h'}\right)^2
    \right]\,dz^2\qquad \left ('=\frac{d}{dz}\right)
\]
with respect to a local complex coordinate $z$ on $M$.

In terms of $g$ and $Q$, the induced metric $ds^2$ and complexification
of the second fundamental form $I\!I$ are 
\[
    ds^2 = (1+|g|^2)^2 \left| \frac{Q}{dg} \right|^2 \; , \qquad
    I\!I= -Q-\overline{Q}+ds^2 \; . 
\]  

Since $K\leq 0$, we can define the {\em total absolute curvature} as 
\[
    \TA(f):=\int_M (-K)\,dA\in [0,+\infty] \;\; .
\]
Then $\TA(f)$ is the area of the image of $M$ in $\CP^1$ of the
secondary Gauss map $g$.  
$\TA(f)$ is generally not an integer multiple of $4\pi$ --- 
for catenoid cousins \cite[Example~2]{Bryant} and their $\delta$-fold 
covers, $\TA(f)$ admits {\em any\/} positive real number.  

For each conformal \cmcone{} immersion $f\colon{}M\to H^3$, there is a 
holomorphic null immersion $F\colon{}\widetilde M\to\SL(2,\C)$, 
the {\em lift\/} of $f$, satisfying the differential equation 
\begin{equation}\label{eq:bryant}
   dF = F \begin{pmatrix} 
             g & -g^2 \\ 1 & -g\hphantom{^2}
          \end{pmatrix}\omega \;\; ,\qquad \omega=\frac{Q}{dg}
\end{equation}
so that $f=FF^{*}$, where $F^{*}=\trans\overline{F}$ \cite{Bryant,uy1}.  
Here we consider 
\[
    H^3=\SL(2,\C)/\SU(2)=\{aa^{*}\,|\,a\in\SL(2,\C)\}\;.
\]
We call a pair $(g,\omega)$ the {\em Weierstrass data\/} of $f$.
The lift $F$ is said to be {\em null\/} because $\det F^{-1}dF$,
the pull-back of the Killing form of $\SL(2,\C)$ by $F$, vanishes
identically on $\widetilde{M}$.
Conversely, for a holomorphic null immersion 
$F\colon{}\widetilde M\to\SL(2,\C)$,
$f:=FF^*$ is a conformal \cmcone{} immersion of $\widetilde M$ into
$H^3$.
If $F=(F_{ij})$, equation \eqref{eq:bryant} implies 
\begin{equation}\label{eq:g-F}
   g=-\frac{dF_{12}}{dF_{11}} =-\frac{dF_{22}}{dF_{21}} \; , 
\end{equation}
and it is shown in \cite{Bryant} that 
\begin{equation}\label{eq:G-F}
   G = \frac{dF_{11}}{dF_{21}} = \frac{dF_{12}}{dF_{22}} \; . 
\end{equation}
The inverse matrix $F^{-1}$ is also a holomorphic null immersion,
and produces a new \cmcone{} immersion 
$f^{\#}=F^{-1}(F^{-1}{})^{*}\colon{}\widetilde M\to H^3$, 
called the {\em dual\/} of $f$ \cite{uy5}.  
The induced metric $ds^2{}^{\#}$ and the Hopf differential $Q^{\#}$ of 
$f^{\#}$ are 
\begin{equation}\label{eq:fund-forms}
    ds^2{}^{\#}= (1+|G|^2)^2 \left|\frac{Q}{dG}\right|^2\;,
    \qquad
    Q^{\#}= - Q \; . 
\end{equation}
So $ds^2{}^{\#}$ and $Q^{\#}$ are well-defined on $M$ itself, 
even though $f^{\#}$ might be defined only on $\widetilde M$.
This duality between $f$ and $f^{\#}$ interchanges the roles of the 
hyperbolic Gauss map $G$ and secondary Gauss map $g$.  
In particular, one has 
\begin{equation} \label{eq:Fsharp}
  dF\, F^{-1}= -(F^{-1})^{-1}d(F^{-1})=
    \begin{pmatrix} G & -G^2 \\ 1 & -G\hphantom{^2} \end{pmatrix}
     \frac{Q}{dG}\;.
\end{equation} 
Hence $dFF^{-1}$ is single-valued on $M$, whereas $F^{-1}dF$ generally
is not.

Since  $ds^2{}^{\#}$ is single-valued on $M$,
we can define the {\em dual total absolute curvature\/} 
\[
    \TA(f^{\#}) := \int_M (-K^{\#})\,dA^{\#},
\]
where $K^{\#}$ ($\leq 0$) and $dA^{\#}$ are the Gaussian curvature and
area element of $ds^2{}^{\#}$, respectively.  As 
\begin{equation}\label{eq:pseudo-sharp}
   d\sigma^2{}^{\#}:=(-K^{\#})ds^2{}^{\#} = 
        \frac{4\,dG\,d\overline{G}}{(1+|G|^2)^2}
\end{equation}
is a pseudo-metric of constant curvature $1$ with developing map $G$,
$\TA(f^{\#})$ is the area of the image of $G$ on $\CP^1=S^2$.
The following assertion is important for us:  
\begin{lemma}[\cite{uy5,Yu2}] \label{lem:complete} 
  The Riemannian metric $ds^2{}^{\#}$ is complete
  $($resp.\ non\-degen\-erate$)$
  if and only if $ds^2$ is complete $($resp.~nondegenerate$)$.  
\end{lemma}

We now assume that the induced metric $ds^2$ 
(and consequently $ds^2{}^\#$) on $M$ is complete and that either
$\TA(f)<\infty$ or $\TA(f^\#)<\infty$, 
hence there exists a compact Riemann surface 
$\overline M_\gamma$ of genus $\gamma$ and a finite set of points 
$\{p_1,\dots,p_n\}\subset \overline M_\gamma$ ($n\geq 1$) so that 
$M$ is biholomorphic to $\overline M_\gamma\setminus\{p_1,\dots,p_n\}$ 
(see Theorem 9.1 of \cite{O}).  
We call the points $p_j$ the {\em ends\/} of $f$.  

Unlike the Gauss map for minimal surface with $\TA < \infty$ in $\R^3$,
the hyperbolic Gauss map $G$ of the surface might not extend to a
meromorphic function on $\overline M_\gamma$, as the Enneper cousin 
(Example~\ref{exa:enneper}) shows.  
However, the Hopf differential $Q$ does extend 
to a meromorphic differential on $\overline M_\gamma$ \cite{Bryant}.  
We say an end $p_j$ $(j=1,\dots,n)$ of a \cmcone{} immersion is 
{\em regular\/} if $G$ is meromorphic at $p_j$.  
When $\TA(f) < \infty$, an end $p_j$ is regular precisely when the order
of $Q$ at $p_j$ is at least $-2$, and otherwise $G$ has an essential
singularity at $p_j$ \cite{uy1}. 
Moreover, the pseudometric $d\sigma^2$ as in \eqref{eq:pseudo} has a
{\em conical singularity\/} at each end $p_j$ \cite{Bryant}.
For a definition of conical singularity, see Section~\ref{sec:red} 
(see also \cite{uy3,uy7}).

Thus the orders of $Q$ at the ends $p_j$ are important for understanding the 
geometry of the surface, so we now introduce a notation that reflects this.  
We say a \cmcone{} surface is of {\em type $\gGamma(d_1,\dots,d_n)$} if it is 
given as a conformal immersion 
$f\colon{}\overline M_\gamma\setminus\{p_1,\dots,p_n\}\to H^3$, 
where $\ord_{p_j} Q = d_j$ for $j=1,\dots,n$ (for example, if $Q=z^{-2}dz^2$ 
at $p_1 = 0$, then $d_1=-2$).  We use $\gGamma$ because it is the capitalized 
form of $\gamma$, the genus of $\overline M_\gamma$.  
For instance, the class $\gI(-4)$  means the class of surfaces of genus
$1$ with $1$ end so that $Q$ has a pole of order $4$ at the end, and 
the class $\gO(-2,-3)$ is the class of surfaces of genus $0$ with 
two ends so that $Q$ has a pole of order $2$ at one end and a pole of 
order $3$ at the other.

\subsection*{Analogue of the Osserman inequality.} 

For a \cmcone{} surface of genus $\gamma$ with $n$ ends, 
the second and third authors showed that the equality of the Cohn-Vossen
inequality for the total absolute curvature never holds \cite{uy1}:
\begin{equation}\label{eq:cohn-vossen-general}
   \frac{1}{2\pi}\TA(f) > -\chi(M) = 2\gamma-2+n\;.
\end{equation}
The catenoid cousins (Example~\ref{exa:catenoid}) show that this 
inequality is the best possible.

On the other hand,  the dual total absolute curvature satisfies an
Osserman-type inequality \cite{uy5}: 
\begin{equation}\label{eq:osserman}
   \frac{1}{2\pi}\TA(f^{\#})\geq -\chi(M)+n = 2(\gamma+n-1) \; . 
\end{equation}   
Moreover, equality holds exactly when all the ends are embedded: 
This follows by noting that equality is equivalent to all ends being
regular and embedded (\cite{uy5}), 
and that any embedded end must be regular (proved recently by Collin,
Hauswirth and Rosenberg \cite{chr} and Yu \cite{Yu3}).  

\subsection*{Effects of transforming the lift $\bmath{F}$.} 
Here we consider the change $\hat F=aFb^{-1}$ of the lift $F$,
where $a,b\in \SL(2,\C)$.  Then $\hat F$ is also a holomorphic null
immersion, and the hyperbolic Gauss map $\hat G$, the secondary Gauss
map $\hat g$ and the Hopf differential $\hat Q$ of $f=\hat F\hat F^*$ 
are given by (see \cite{uy3})
\begin{equation}\label{eq:changes}
  \hat G=a \star G = \frac{a_{11}G+a_{12}}{a_{21}G+a_{22}}\;,\quad
  \hat g=b\star g = \frac{b_{11}g+b_{12}}{b_{21}g+b_{22}}\;,\quad
  \text{and}\quad
  \hat Q=Q\;,
\end{equation}
where $a=(a_{ij})$ and $b=(b_{ij})$.
In particular, the change $\hat F=aF$ moves the surface by a rigid
motion of $H^3$, and does not change $g$ and $Q$.  

\subsection*{\boldmath$\SU(2)$-monodromy conditions.}  

Here we recall from \cite{ruy1} the construction of \cmcone{}
surfaces with given hyperbolic Gauss map $G$ and Hopf differential $Q$.

Let $\overline{M}_{\gamma}$ be a compact Riemann surface
and $M:=\overline{M}_{\gamma} \setminus\{p_1,\dots,p_n\}$.
Let $G$ and $Q$ be a meromorphic function and meromorphic
$2$-differential on $\overline{M}_{\gamma}$.
We assume the pair $(G,Q)$ satisfies the following two compatibility
conditions:
\begin{align}
  \label{eq:cond1}
      &\text{For all $q \in M$, 
        $\ord_q Q$ is equal to the branching order of $G$, and}\\
  \label{eq:cond2} 
      &\text{for each end $p_j$, (branching order of $G$)$-d_j\geq 2$.}
\end{align}
The first condition implies that the metric $ds^2{}^{\#}$ as in
\eqref{eq:fund-forms} is non-degenerate at $q\in M$.
The second condition implies that the metric ${ds^2}^\#$ is complete at
$p_j\in \overline{M}_{\gamma}$ ($j=1,\dots,n$).
Our goal is to get a \cmcone{} immersion $f\colon{}M\to H^3$ with
hyperbolic Gauss map $G$ and Hopf differential $Q$.
If such an immersion exists, the induced metric $ds^2$ of $f$ is
non-degenerate and complete, by Lemma~\ref{lem:complete}.

Since a pair $(G,Q)$ satisfies \eqref{eq:cond1} and \eqref{eq:cond2}, 
the differential equation \eqref{eq:Fsharp}  may have singularities at
$\{p_1,\dots,p_n\}$, but is regular on $M$.  
Then there exists a solution $F\colon{}\widetilde M\to \SL(2,\C)$,
where $\widetilde M$ is the universal cover of $M$.
Since the solution $F$ of \eqref{eq:Fsharp} is unique up to the change
$F\mapsto Fa$ ($a\in\SL(2,\C)$), there exists a representation
$\rho_F\colon{}\pi_1(M)\to \SL(2,\C)$ such that
\begin{equation}\label{eq:F-repr}
     F\circ\tau = F\rho_F(\tau)  \qquad 
       \bigl(\tau\in\pi_1(M)\bigr)\;.
\end{equation}
Here we consider an element $\tau$ of the fundamental group $\pi_1(M)$
as a deck transformation on $\widetilde M$.
Thus:
\begin{proposition}\label{prop:sutwo}
 If there exists a solution $F\colon{}\widetilde M\to\SL(2,\C)$ 
 of \eqref{eq:Fsharp} for $(G,Q)$ satisfying \eqref{eq:cond1} and
 \eqref{eq:cond2}, 
 then $f:=FF^*$ is a complete conformal \cmcone{} immersion into $H^3$ 
 which is well-defined on $M$ if $\rho_F(\tau) \in \SU(2)$ for all 
 $\tau\in\pi_1(M)$.
 Moreover, the hyperbolic Gauss map and the Hopf differential of $f$
 are $G$ and $Q$, respectively.  
\end{proposition}
\section{Important Examples with $\TA(f)$ or $\TA(f^{\#})\leq 8\pi$}
\label{sec:Added}
In this section, we shall introduce several important \cmcone{}
surfaces with $\TA(f)$ $\leq 8\pi$ or $\TA(f^{\#})\leq 8\pi$.
\begin{figure}
\begin{center}
\begin{tabular}{c@{\hspace{3em}}c@{\hspace{3em}}c}
       \includegraphics[width=0.9in]{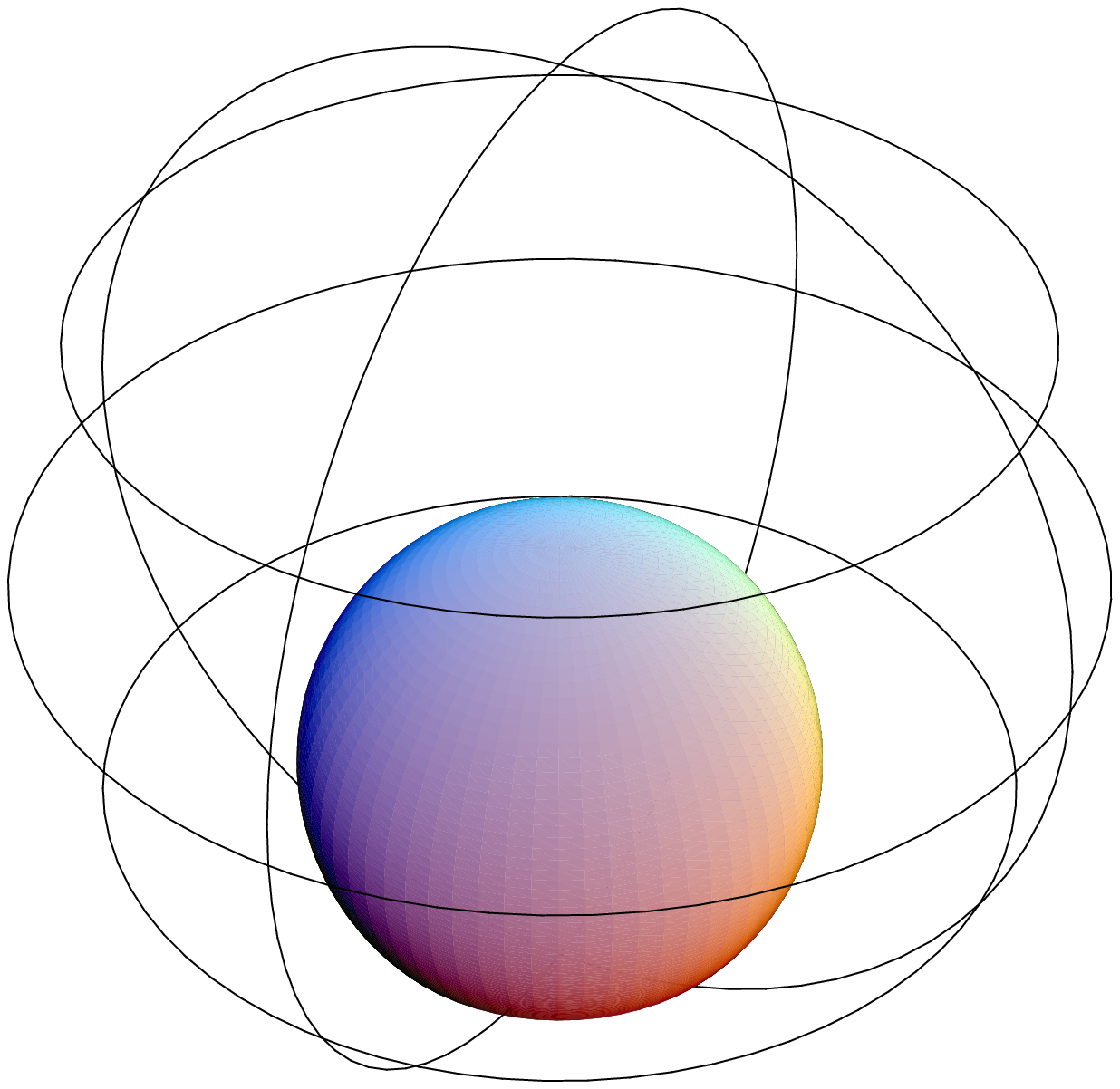}& 
       \includegraphics[width=0.9in]{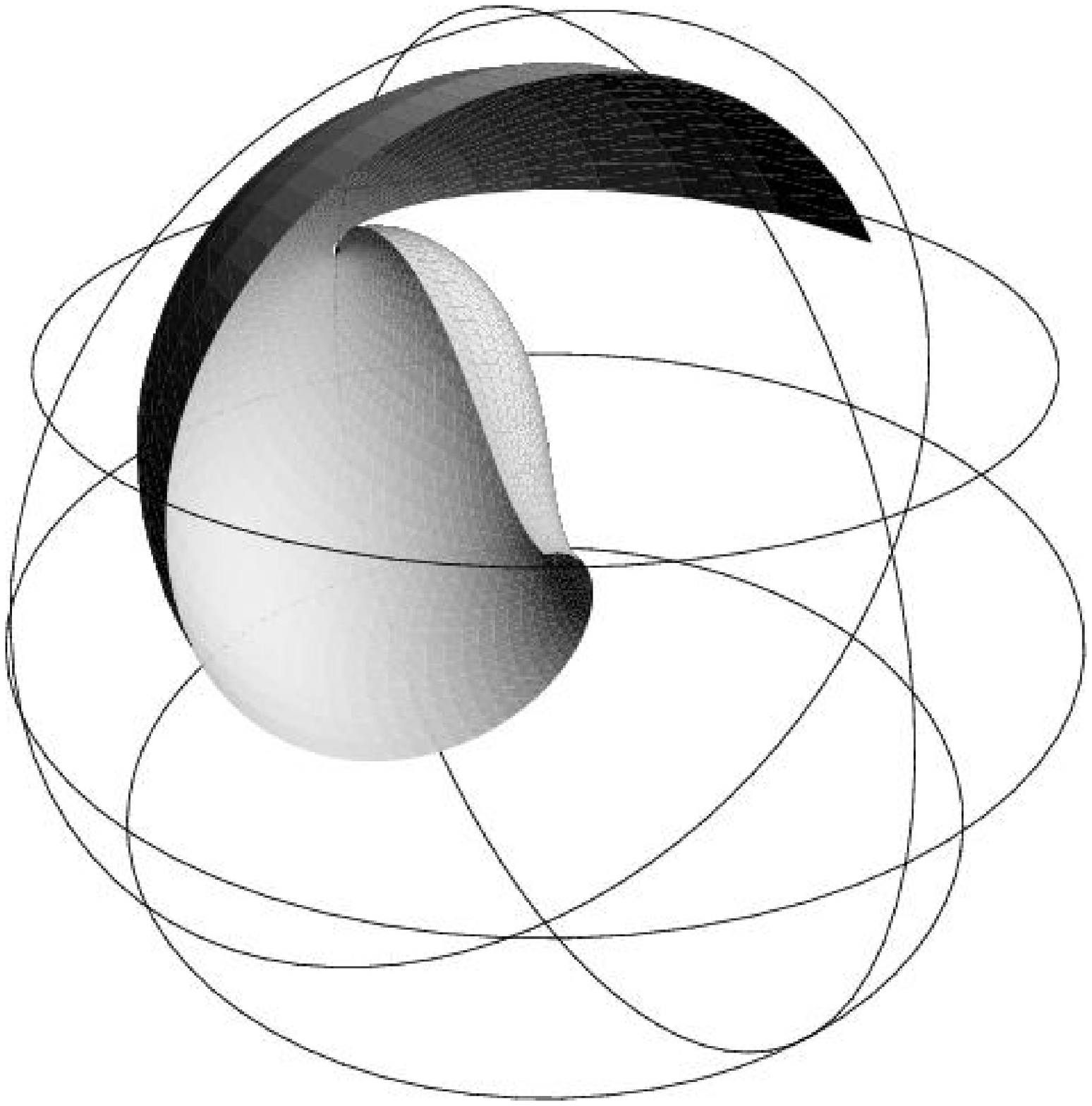}&
       \includegraphics[width=0.9in]{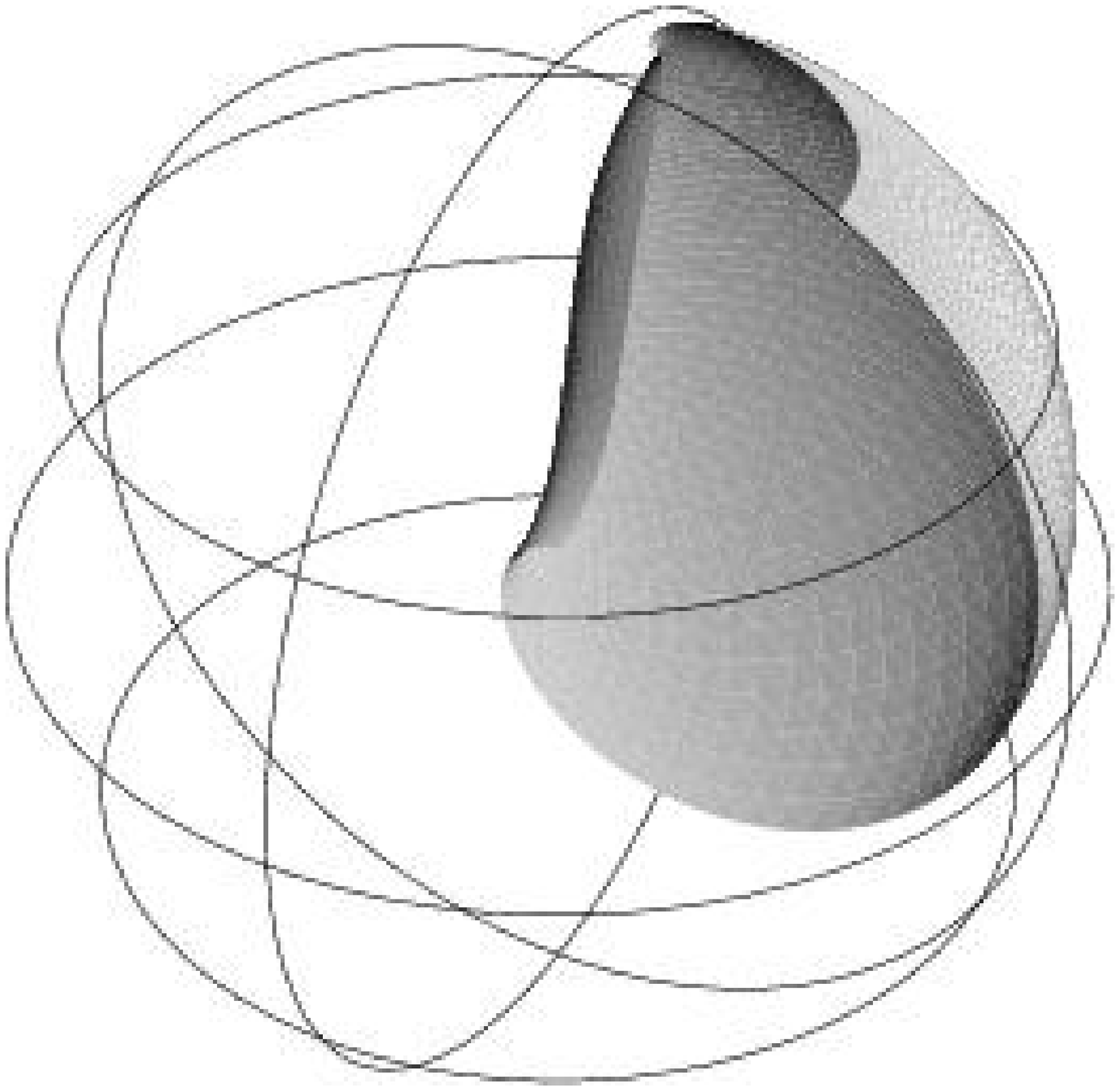} 
\end{tabular}
\end{center}
\caption{
 A horosphere, and fundamental pieces (one-fourth of the surfaces 
 with the ends cut away) of an Enneper cousin and the dual of an Enneper
 cousin.
 Figures are shown in the Poincar\'e model of $H^3$.
}
\label{fig:enneper}
\end{figure}
\begin{example}[Horosphere]\label{exa:horosphere}
 A horosphere (Figure~\ref{fig:enneper}) is the only surface of
 type $\gO(0)$,
 with Weierstrass data given by
\[
    g=0,\qquad \omega=a\, dz\qquad (a\in \C\setminus\{0\})\;.
\]
 The holomorphic lift $F\colon{}\C\to \SL(2,\C)$ of the surface
 with initial condition $F(0)=\id$ is given by
\[
 F=\begin{pmatrix}
    1  & 0 \\
    az & 1 
   \end{pmatrix}\;.
\]
 In particular the hyperbolic Gauss map is a constant function, 
 as well as the secondary Gauss map $g=0$.
 This surface is flat and totally umbilic. 
 In particular, the total curvature and the dual total 
 curvature of the surface are both equal to zero.
 Any flat or totally umbilic \cmcone{} surfaces are
 parts of this surface. Planes in $\R^3$ are 
 the corresponding minimal surfaces with the same
 Weierstrass data $(g,\omega)=(0,a\,dz)$.
\end{example}

\begin{example}[Enneper cousin and dual of Enneper cousin]
\label{exa:enneper}

 The Enneper cousin is  given in \cite{Bryant}
 (Figure~\ref{fig:enneper}), with the same Weierstrass data as the
 Enneper surface in $\R^3$:
\[
      g=z,\qquad \omega=a\, dz\qquad \bigl(a\in \C\setminus\{0\}\bigr)\;.
\]
 The holomorphic lift $F\colon{}\C\to \SL(2,\C)$ of the surface
 with initial condition $F(0)=\id$ is given by
 \[
  F=
  \begin{pmatrix}
   \cosh(az) & a^{-1}\sinh(az)-z \cosh(az) \\
   a\sinh(a z) & \cosh(az)- az \sinh(az) 
  \end{pmatrix}\;.
 \]
 In particular the hyperbolic Gauss map $G$ is given by
 \[
   G=a^{-1}\tanh (az)\;.
 \]
 The Enneper cousin is in the class $\gO(-4)$ and has a complete induced
 metric of total absolute curvature $4\pi$.
 If one takes the inverse of $F$, one gets the dual of the Enneper
 cousin (Figure~\ref{fig:enneper}).
 Since
 \[
   F d (F^{-1})=-dF F^{-1}=
   \begin{pmatrix}
     -a \cosh(az)\sinh(az) & \sinh^2(az) \\
     -a^2 \cosh^2(a z) & a \cosh(a z)\sinh(az) 
   \end{pmatrix}\;,
 \]
 the Weierstrass data $(g^\#,\omega^\#)$ of the dual of the Enneper cousin
 given by
 \[
    g^\#=a^{-1}\tanh (az),\qquad \omega^\#=a^2 \cosh^2(az)\, dz\;.
 \]
 This surface is also in  the class $\gO(-4)$ and has a 
 complete induced metric of infinite total absolute curvature
 (see Lemma~\ref{lem:complete}).
\end{example}
\begin{figure}
\begin{center}
\begin{tabular}{c@{\hspace{3em}}c@{\hspace{3em}}c}
       \includegraphics[width=0.9in]{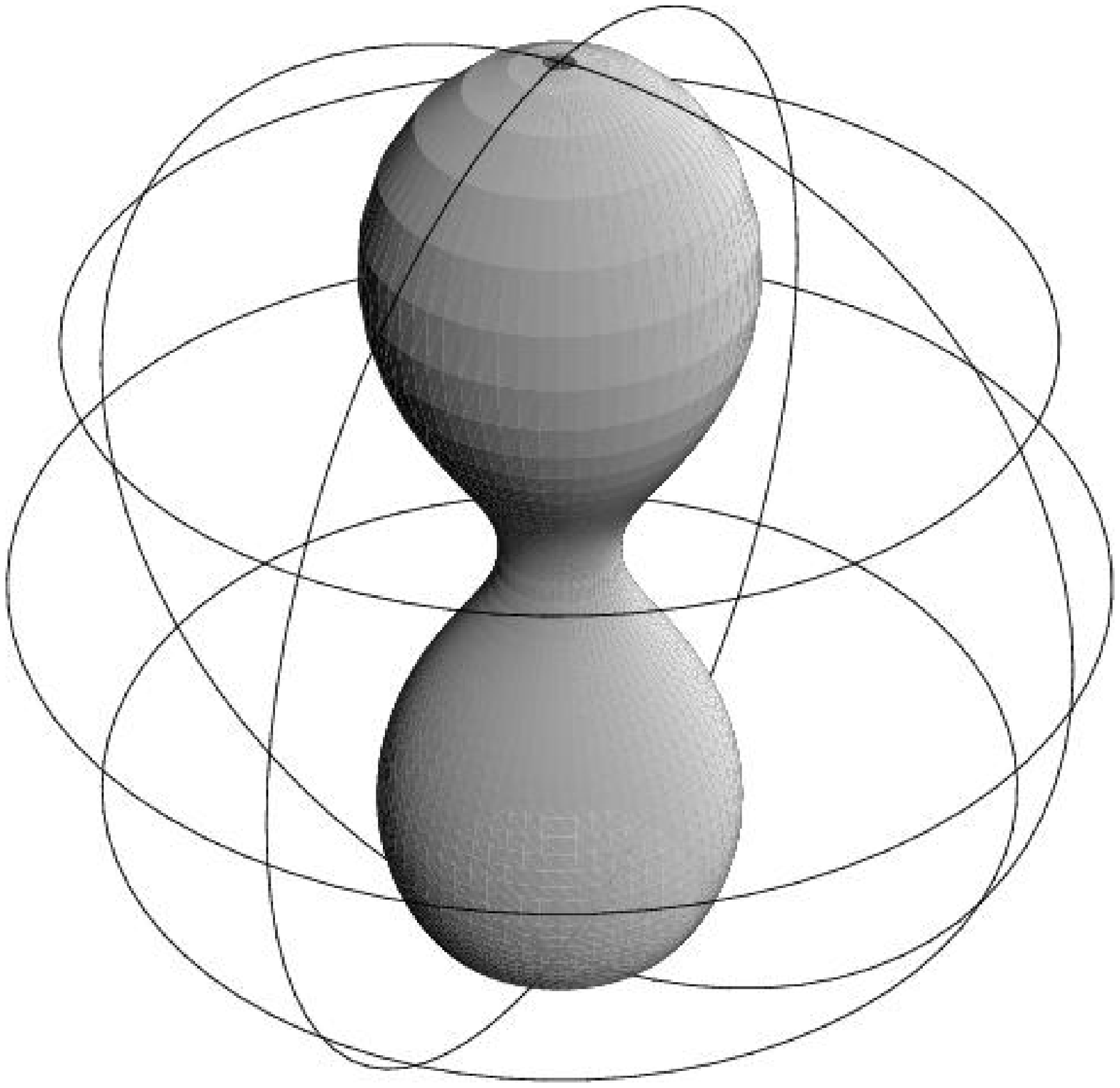}& 
       \includegraphics[width=0.9in]{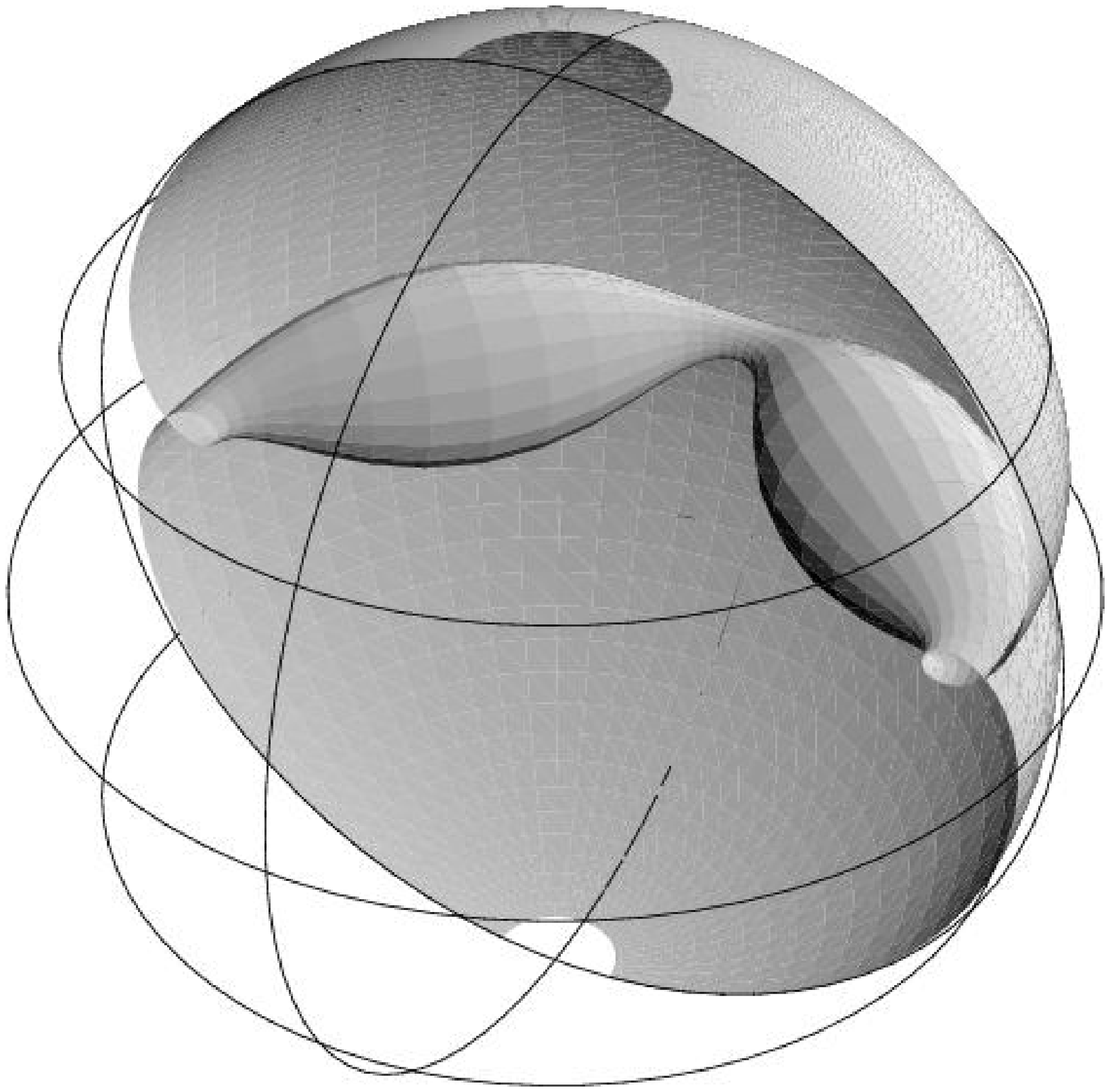}&
       \includegraphics[width=0.9in]{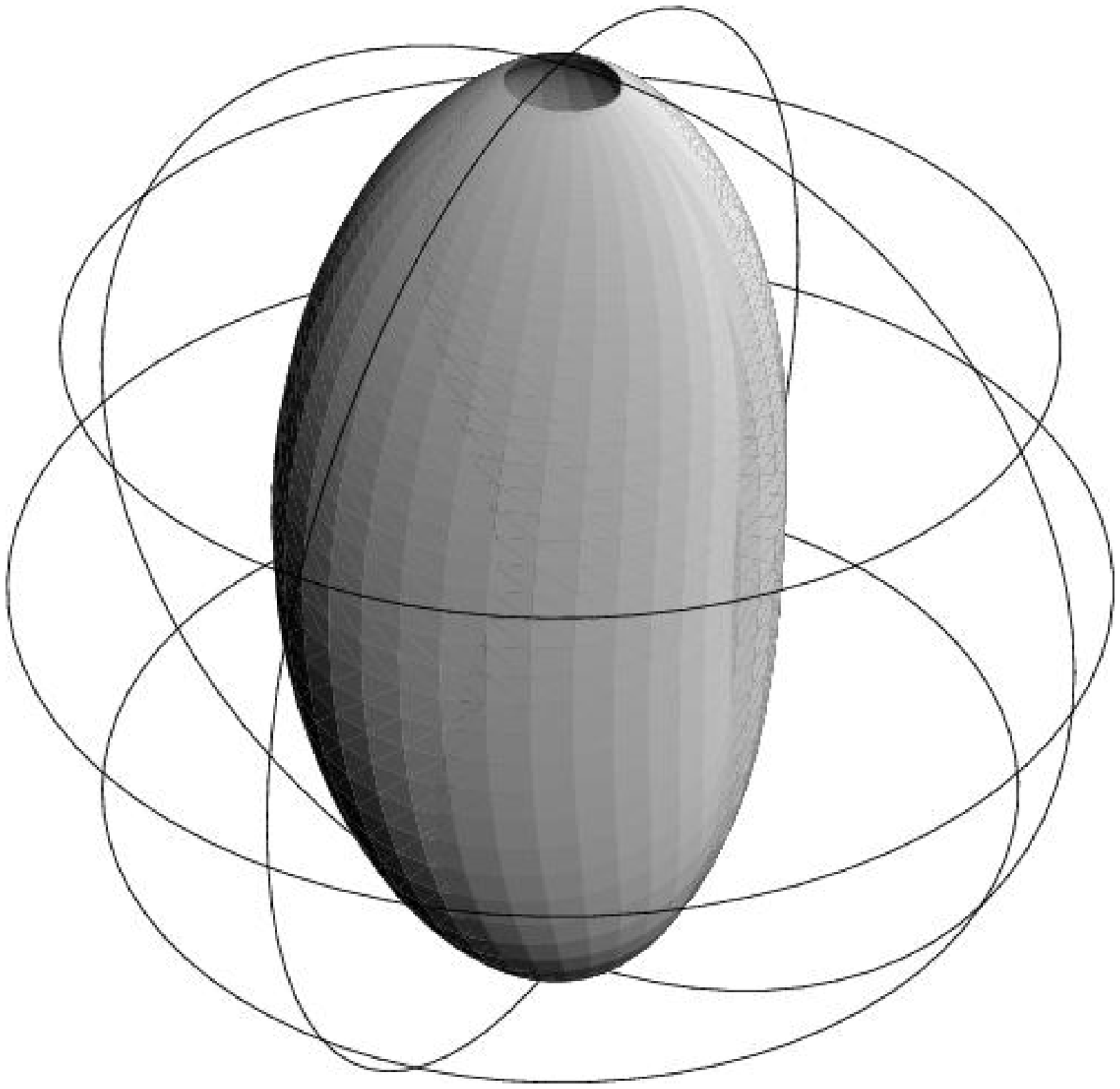} 
\end{tabular}
\end{center}
\caption{
 A catenoid cousin with $l=0.8$, and warped catenoid cousins with
 $(l,\delta,b)=(4,1,1/2)$ and $(1,2,1/2)$.
 The third surface has $\TA(f)=4\pi$ because $l=1$ even though its ends
 are not embedded.
}
\label{fig:catenoid}
\end{figure}
\begin{figure}
\begin{center}
\begin{tabular}{c@{\hspace{5em}}c@{\hspace{3em}}c}
       \includegraphics[height=0.9in]{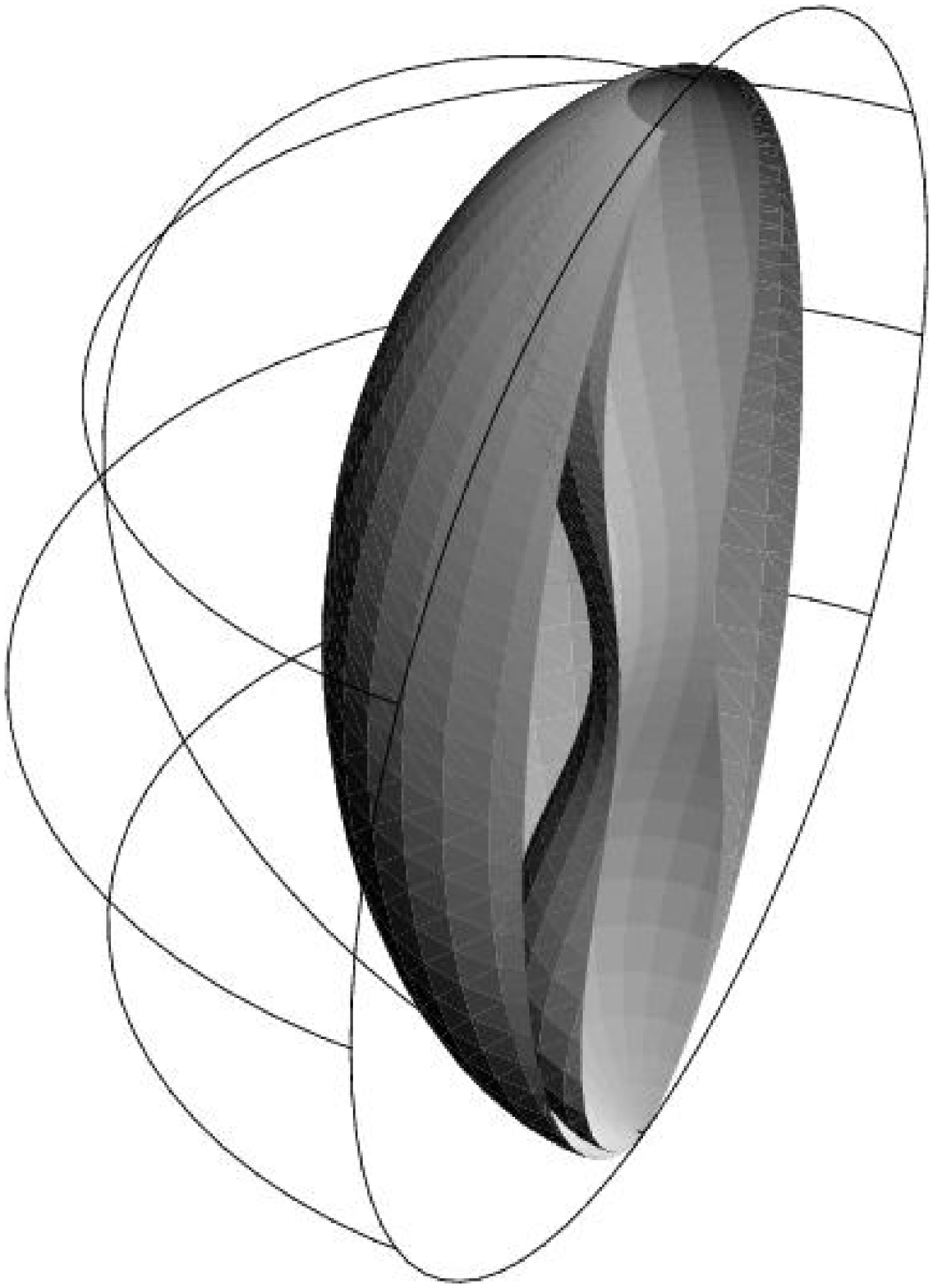}& 
       \includegraphics[height=0.9in]{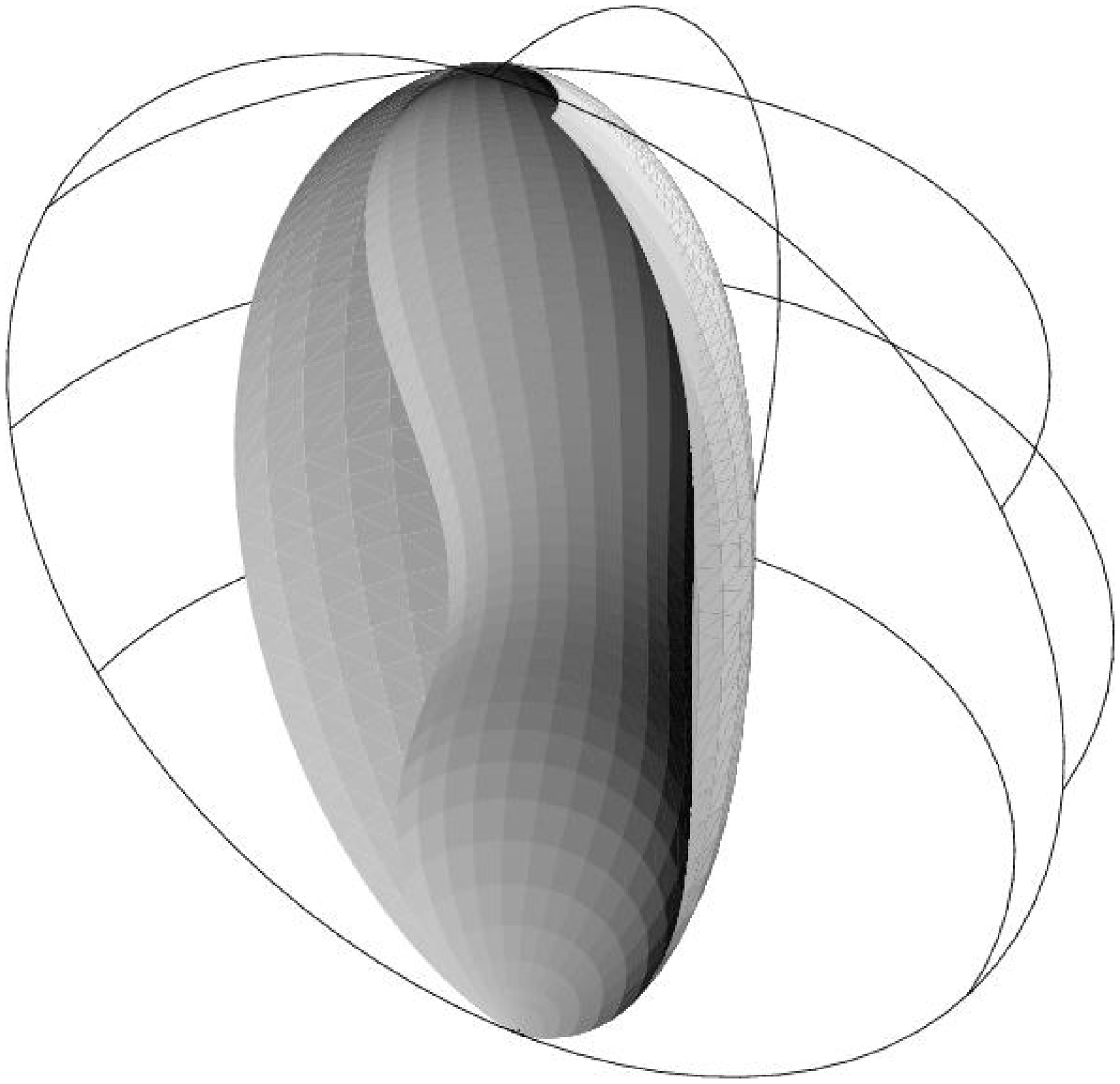}&
       \includegraphics[width=0.9in]{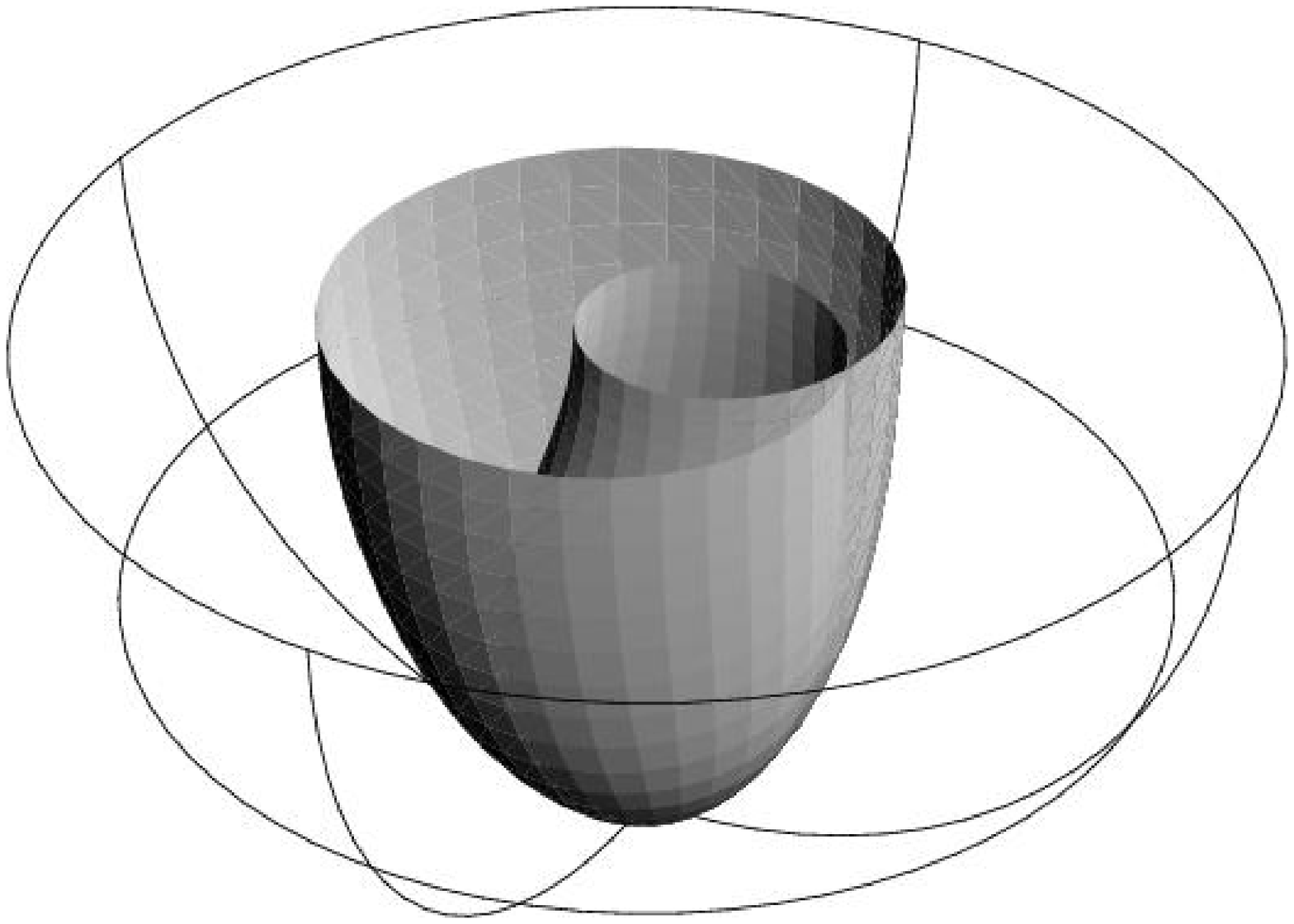}  
\end{tabular}
\end{center}
\caption{
  Cut-away views of the third warped catenoid cousin in
  Figure~\ref{fig:catenoid}.
}
\label{fig:warped}
\end{figure}
\begin{example}[Catenoid cousins and warped catenoid cousins]
\label{exa:catenoid}
 \cmcone{} surfaces \newline
 of type $\gO(-2,-2)$ are classified in Theorem~6.2
 in \cite{uy1}.
 Here we give a slightly refined version given  in \cite{ruy4}:  
 A complete conformal \cmcone{} immersion 
 $f\colon{}M=\C\setminus \{0\} \to H^3$ with regular ends have the
 following Weierstrass data
\begin{equation}
\label{eq:catcousdata}
    g = \frac{\delta^2-l^2}{4l}z^l+b \; ,\qquad
    \omega=\frac{Q}{dg} = z^{-l-1}dz \; , 
\end{equation}
 with $l>0$, $\delta\in\Z^+$, and $l \neq \delta$, and $b \geq 0$,
 where the case $b>0$ occurs only when $l\in \Z^+$. 
 When $b=0$ and $\delta=1$, the surface is called a 
 {\em catenoid cousin}, which is rotationally symmetric. 
 (The Weierstrass data of the catenoid cousin is often written as 
  $g=z^\mu$ and $\omega=(1-\mu^2)z^{-\mu-1}\,dz/(4\mu)$.  
 This is equivalent to \eqref{eq:catcousdata} for $b=0$ and $\delta=1$
 and $l=\mu$ by a coordinate change $z\mapsto ((1-\mu^2)/4\mu)^{(1/\mu)} z$.) 
 Catenoid cousins are embedded when
 $0<l<1$ and have one curve of self-intersection when $l>1$. 
 When $b=0$, $f$ is a $\delta$-fold cover of a catenoid cousin.  
 When $b>0$ (then automatically $l$ is a positive integer), 
 we call $f$ a {\em warped catenoid cousin}, 
 and its  discrete symmetry group is the natural $\Z_2$ extension of the
 dihedral group $D_l$.  
 Furthermore, the warped catenoid cousins can be written explicitly as 
 \[ 
   f= F F^{*},\qquad F = F_0 B \; , 
 \]
 where 
 \[ 
    F_0 = \sqrt{\frac{\delta^2-l^2}{\delta}}
              \begin{pmatrix}
                 \dfrac{1}{l-\delta}    z^{(\delta-l)/2}  &
                 \dfrac{\delta-l}{4l}   z^{(l+\delta)/2} \\
                 \dfrac{1}{l+\delta}    z^{-(l+\delta)/2} &
                 \dfrac{-(l+\delta)}{4l}z^{(l-\delta)/2}
              \end{pmatrix} \quad 
          \text{and} \quad
        B =   \begin{pmatrix}
                    1 & -b \\
                    0 & \hphantom{-}1
              \end{pmatrix}   \,.
 \]
 In particular, the hyperbolic Gauss map and Hopf differential 
 are given by 
 \[
    G=z^\delta,\qquad Q=\frac{\delta^2-l^2}{4z^2}\,dz^2 \; , 
 \]
 which are equal to the Gauss map and Hopf differential of
 the catenoids in $\R^3$.
 The dual total curvature of a catenoid cousin is $4\pi$, but
 its total curvature is $4\pi l$ ($l>0$), which can take
 any value in $(0,4\pi) \cup (4\pi,\infty)$. 
 On the other hand, the total absolute curvature and the dual total
 absolute curvature of warped catenoid cousins are always integer
 multiples of $4\pi$.
 (See Figures~\ref{fig:catenoid} and \ref{fig:warped}).
\end{example}

\begin{figure}
\begin{center}
\begin{tabular}{c@{\hspace{6em}}c}
       \includegraphics[width=0.9in]{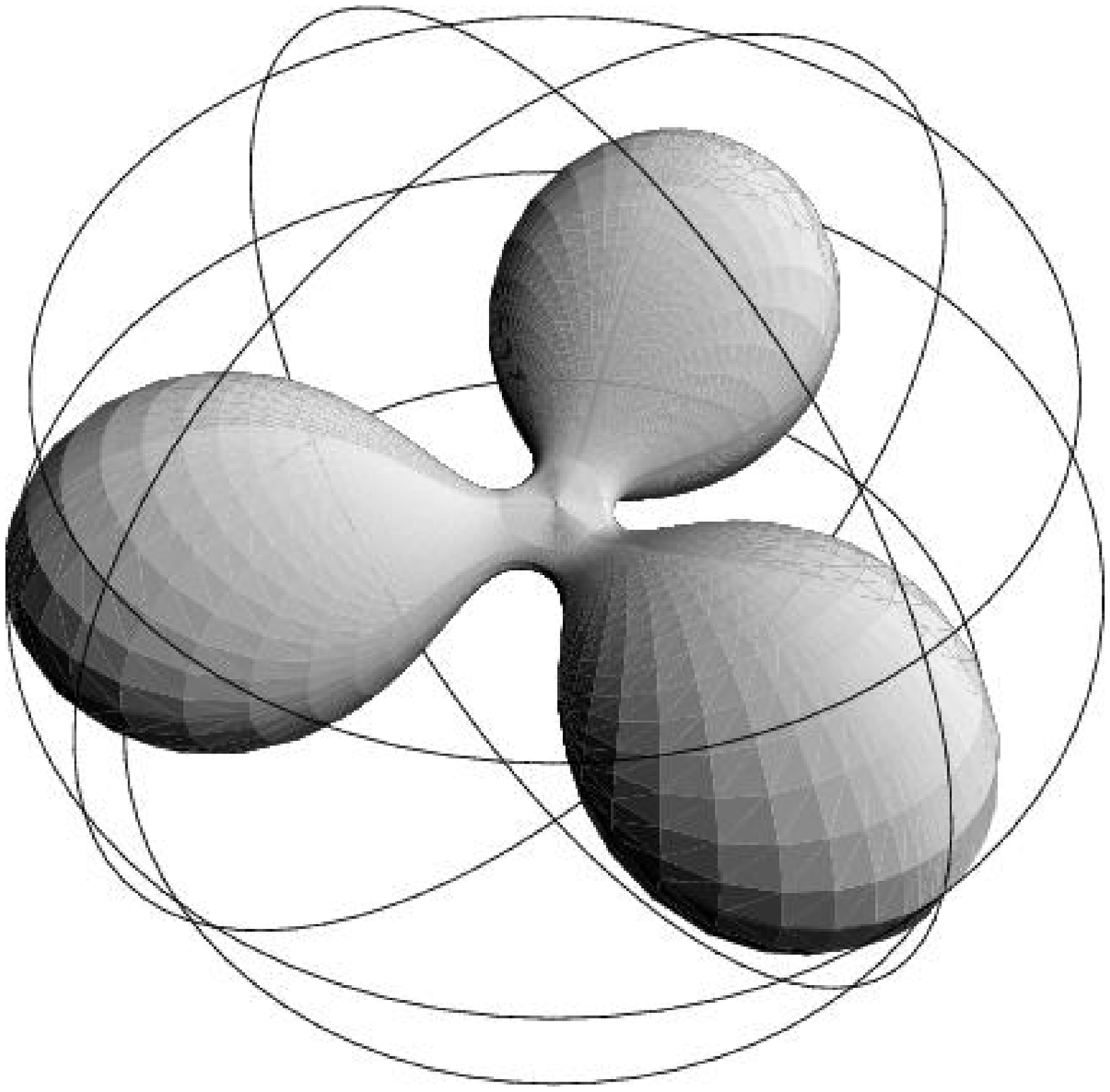} & 
       \includegraphics[width=0.9in]{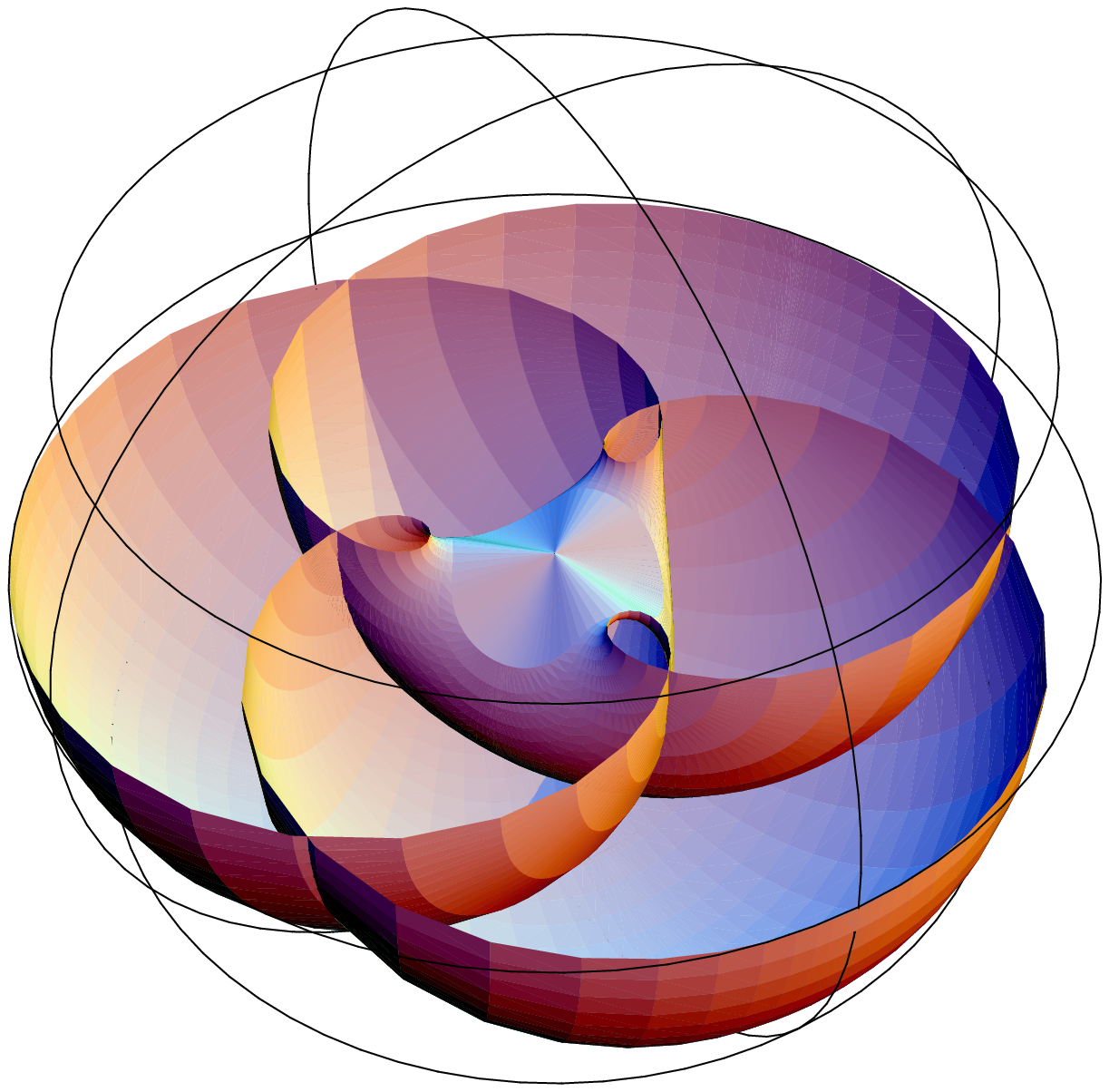}
\end{tabular}
\end{center}
\caption{
 Two different \cmcone{} trinoids (proven to exist in \cite{uy3}).
 Although these surfaces are proven to exist, and numerical 
 experiments show that some of them are embedded (as one of the 
 pictures here is), none have yet been proven to be embedded.}
\label{fig:trinoid2}
\end{figure}
\begin{example}[Irreducible trinoids]
\label{exa:trinoid}
 We take three real numbers $\mu_1$, $\mu_2$, $\mu_3>-1$ such that
 \begin{equation}\label{eq:condition-eq}
    \cos^2 B_1 +\cos^2 B_2 + \cos^2 B_3+
               2\cos B_1\cos B_2 \cos B_3< 1,
 \end{equation}
 where $B_j=\pi(\mu_j+1)$ $(j=1,2,3)$.
 We also assume
\begin{equation}\label{eq:nodouble}
    c_1^2+c_2^2+c_3^2-2(c_1c_2+c_2c_3+c_3c_1)\neq 0,
\end{equation}
 where $c_j=-\mu_j(\mu_j+2)/2\in\R$  ($j=1,2,3$).
 Then it is shown in \cite{uy7} that there exists a unique \cmcone{}
 surface $f_{\mu_1,\mu_2,\mu_3}\colon{}\C\setminus \{0,1\}\to H^3$ of 
 type $\gO(-2,-2,-2)$ such that the pseudometric $d\sigma^2=(-K)ds^2$
 defined by \eqref{eq:pseudo} is irreducible and has conical
 singularities of orders $\mu_1$, $\mu_2$, $\mu_3$ at $z=0,1,\infty$,
 respectively.
 Moreover, any irreducible \cmcone{} surface of type $\gO(-2,-2,-2)$
 whose ends are all embedded is congruent to some 
 $f_{\mu_1,\mu_2,\mu_3}$.
 All ends of these surfaces are asymptotic to catenoid cousin ends.
 The inequality \eqref{eq:condition-eq} implies  $\mu_1$, $\mu_2$ and
 $\mu_3$ are all non-integers.

 If we allow equality in \eqref{eq:condition-eq},
 one of the $\mu_1,\mu_2,\mu_3$ must be an integer.
 The corresponding \cmcone{} surface might not exist for such 
 $\mu_1,\mu_2,\mu_3$ in general \cite{uy7}. If it exists, its induced 
 pseudometric $d\sigma^2$ must be reducible (see
 Lemmas~\ref{lem:hyp-3-red-0} and \ref{lem:hyp-1-red-0}).

 The  Hopf differential $Q$
 of $f_{\mu_1,\mu_2,\mu_3}$ is given by
\begin{equation}
\label{eq:Q-surface}
  Q=\frac{1}{2}\left(
               \frac{c_3 z^2+(c_2-c_1-c_3)z+c_1}{z^2(z-1)^2}
            \right)\,dz^2\;.
\end{equation}
 Let $q_1$ and $q_2$ be zeros of $Q$, that is
\begin{equation}\label{eq:umbilic}
   c_3 q_l^2 + (c_2-c_1-c_3)q_l+c_1=0\qquad (l=1,2)\;.
\end{equation}
 By \eqref{eq:nodouble}, $q_1\ne q_2$ holds.
 The hyperbolic Gauss map is then given by
\begin{equation}\label{eq:G-surface}
      G =  z + \frac{(q_1-q_2)^2}{2\{2z-(q_1+q_2)\}}\;.
\end{equation}
 In particular, all of these surfaces have dual total 
 absolute curvature $8\pi$.
 On the other hand, the total curvature is equal to
 $2\pi(4+\mu_1+\mu_2+\mu_3)$. 
 If we set $\mu=\mu_1=\mu_2=\mu_3$, the condition \eqref{eq:condition-eq}
 implies that $\mu>-2/3$, and then there exist $f_{\mu,\mu,\mu}$ for
 any $\mu$ arbitrarily close to $-2/3$, whose total curvatures 
 tend to $4\pi$. This implies Theorem~\ref{thm:odd} is
 sharp for $m=1$.

\begin{figure}
\begin{center}
  \begin{tabular}{c@{\hspace{4em}}c}
      \includegraphics[width=1in]{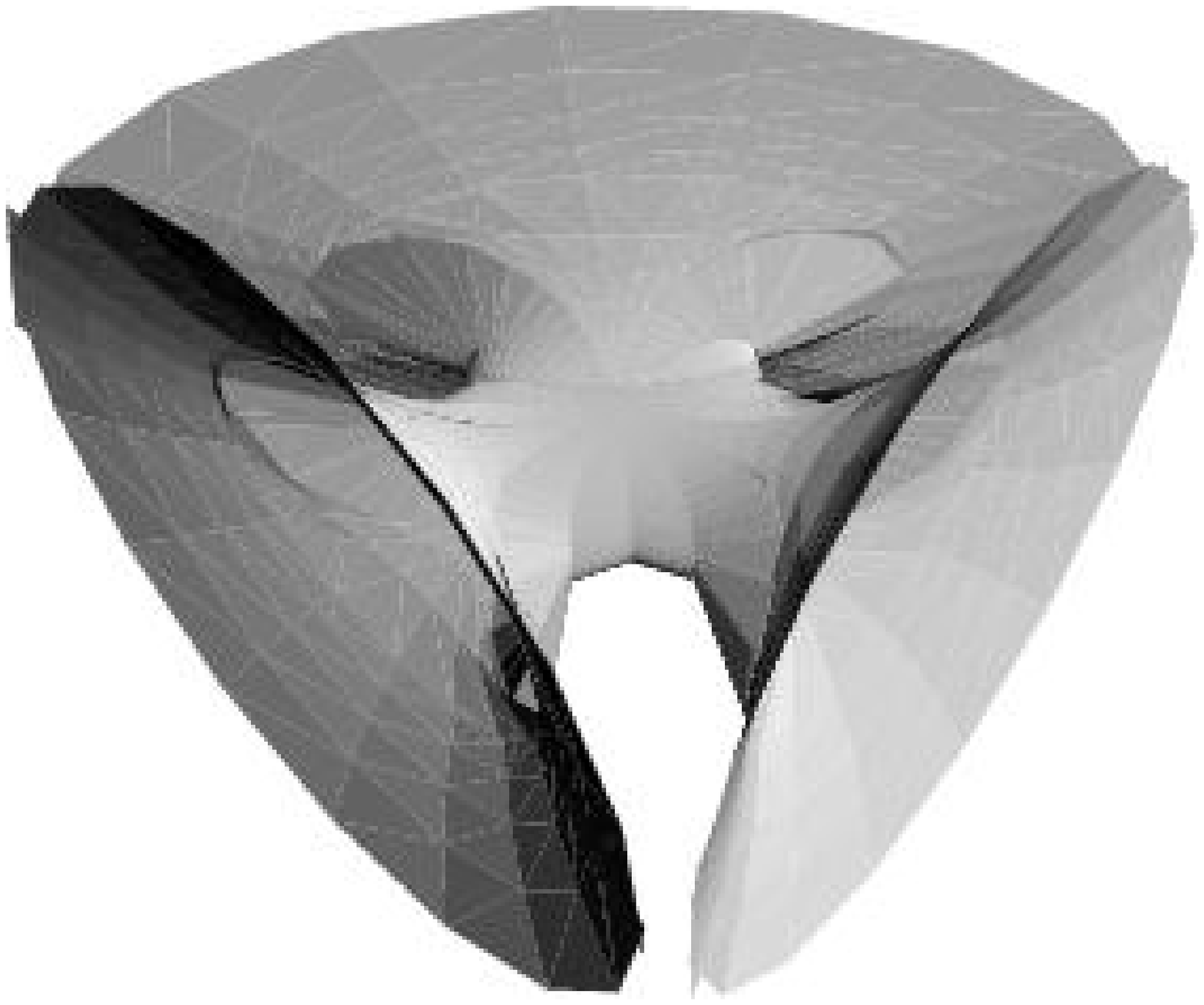} &
      \includegraphics[width=1in]{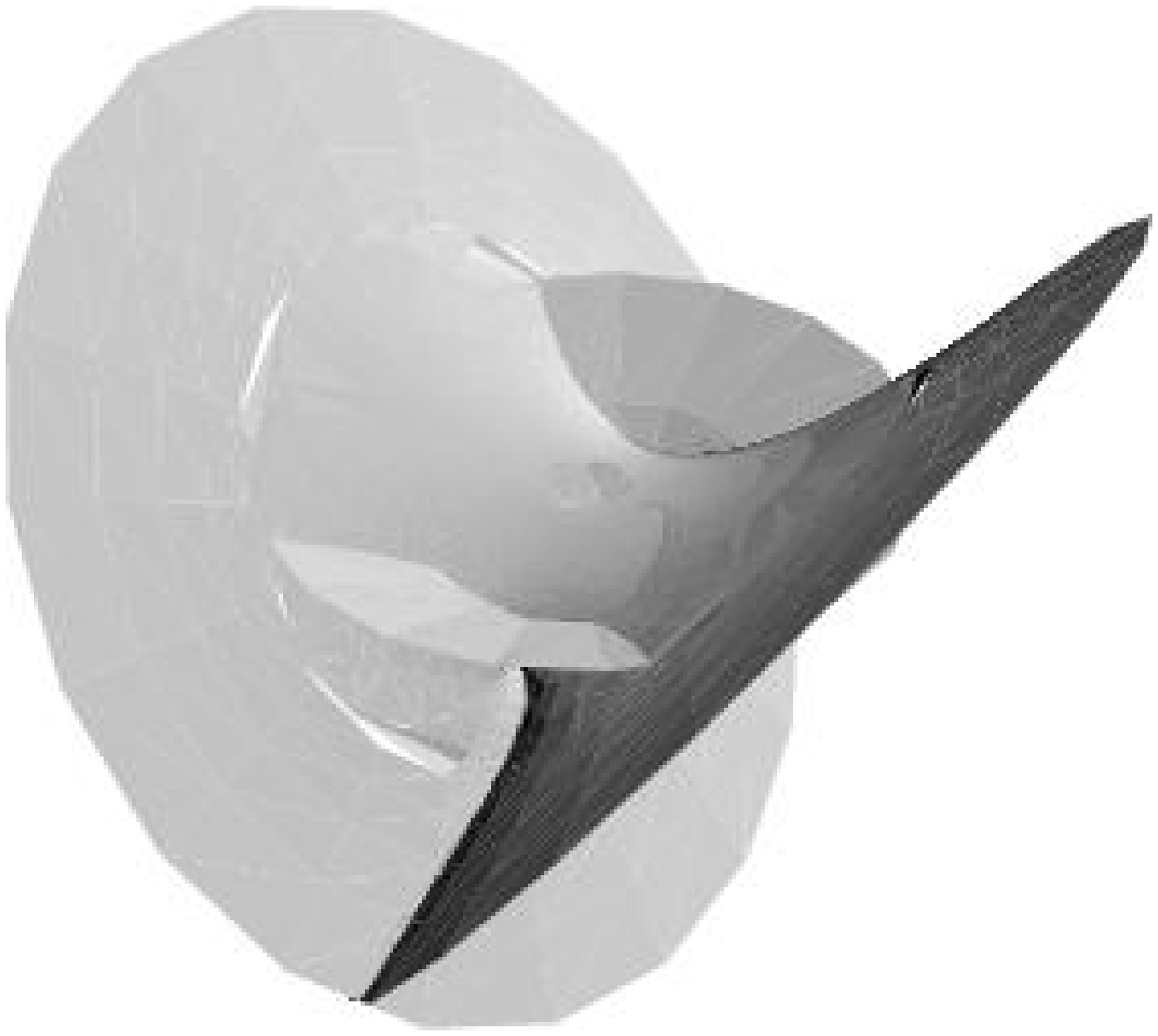} \\[6pt]
      {\footnotesize Type $\typeP$} &
      {\footnotesize Type $\typeN$} 
   \end{tabular}
\end{center}
\caption{Minimal trinoids of types $\typeP$ and $\typeN$.
         The graphics are made by S. Tanaka of Hiroshima University.}
\label{fig:PNtype}
\end{figure}
 It is interesting to compare these surfaces with minimal trinoids in 
 $\R^3$.
 Minimal trinoids with three catenoid ends are classified in 
 Barbanel~\cite{ba}, Lopez~\cite{Lopez} and Kato~\cite{Ka}.
 Here, we adopt Kato's notation \cite{Ka}:
 The Weierstrass data of these trinoids
 $x_0:\C\cup\{\infty\}\setminus\{0,p_1,p_2\}\to \R^3$
 are given by 
 \[
    g=z-\frac{b(p_1^2 {p_2}^2+p_1^2+{p_2}^2-p_1p_2)}{f(z)},\qquad
    \omega=-f(z)^2\,dz \qquad (b\in \R)
 \]
 where $p_1$ and $p_2$ are real numbers such that 
 $(p_1-p_2)(1+p_1p_2)\ne 0$ and
 \[
   f(z):=b\left(\frac{p_1(p_1-p_2)}{z-p_1}
            +\frac{p_2(p_2-p_1)}{z-p_2}+\frac{p_1p_2(p_1p_2+1)}{z}
          \right)\;.
 \]
 If the coefficients of $1/z^2$, $1/(z-p_1)^2$, $1/(z-p_2)^2$
 in the Laurent expansion of the Hopf differential  $Q=\omega\, dg$ at
 $z=0,p_1,p_2$ are all the same signature,
 the surface is called of {\em type $\typeP$} and otherwise it is
 called of {\em type $\typeN$}.
 Type $\typeP$ surface are all Alexandrov-embedded.
 On the other hand, type $\typeN$ surfaces are not.
 (For a definition of Alexandrov embedded, see Cos\'\i{}n and Ros \cite{CR}.)
 These two classes consist of the two connected  components of the set
 of minimal trinoids (Tanaka \cite{ta}; see Figure~\ref{fig:PNtype}).
\begin{figure}
\begin{center}
\footnotesize
\begin{tabular}{cccc}
  \includegraphics[width=0.9in]{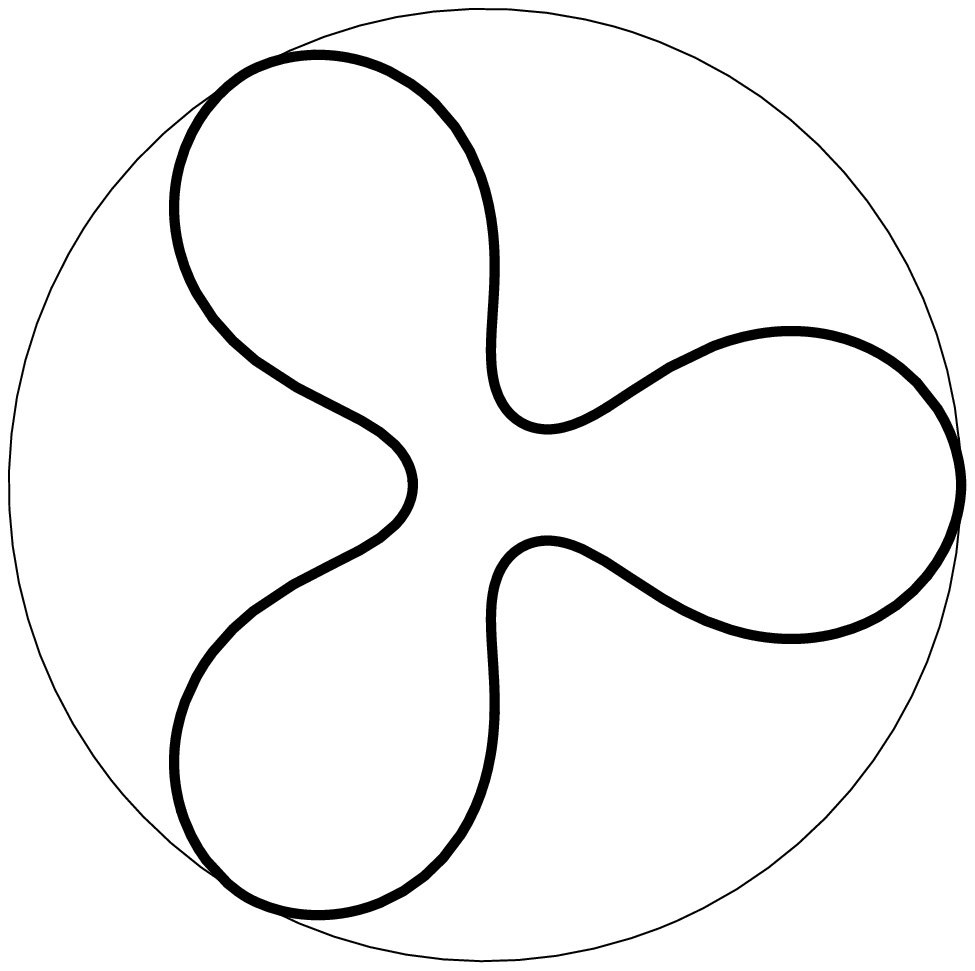} &
  \includegraphics[width=0.9in]{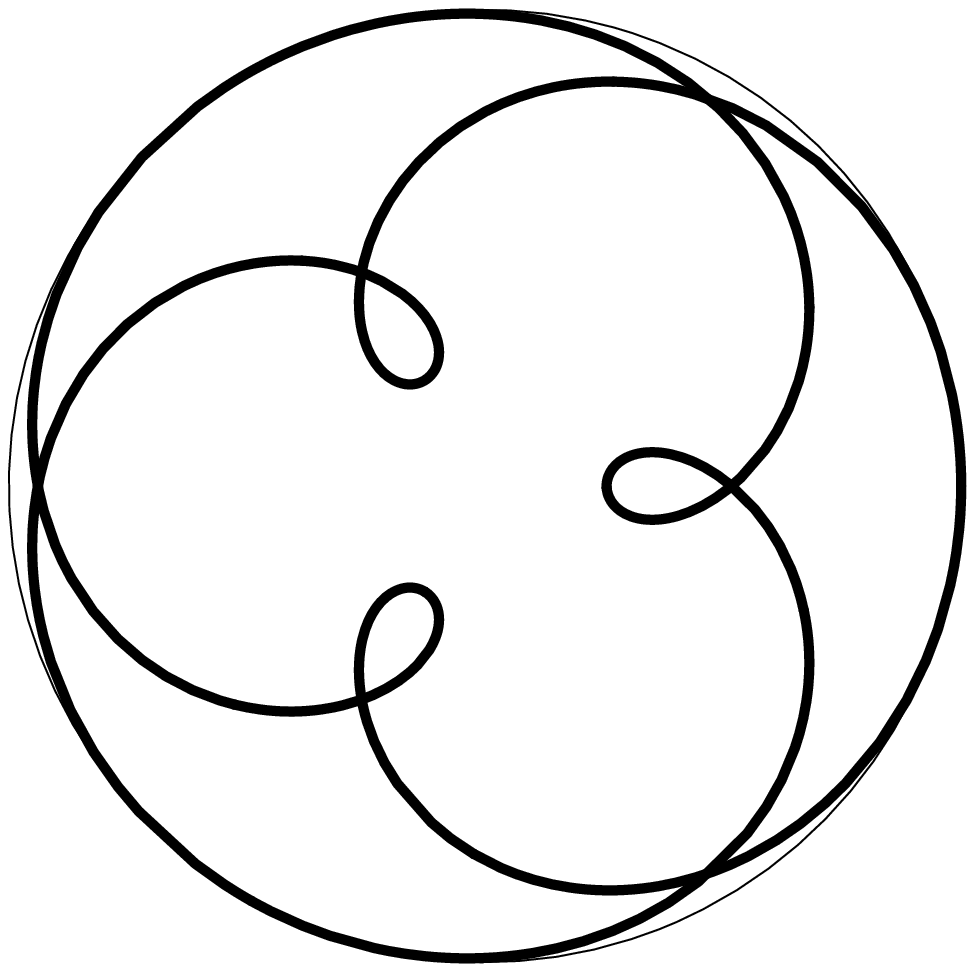} &
  \includegraphics[width=0.9in]{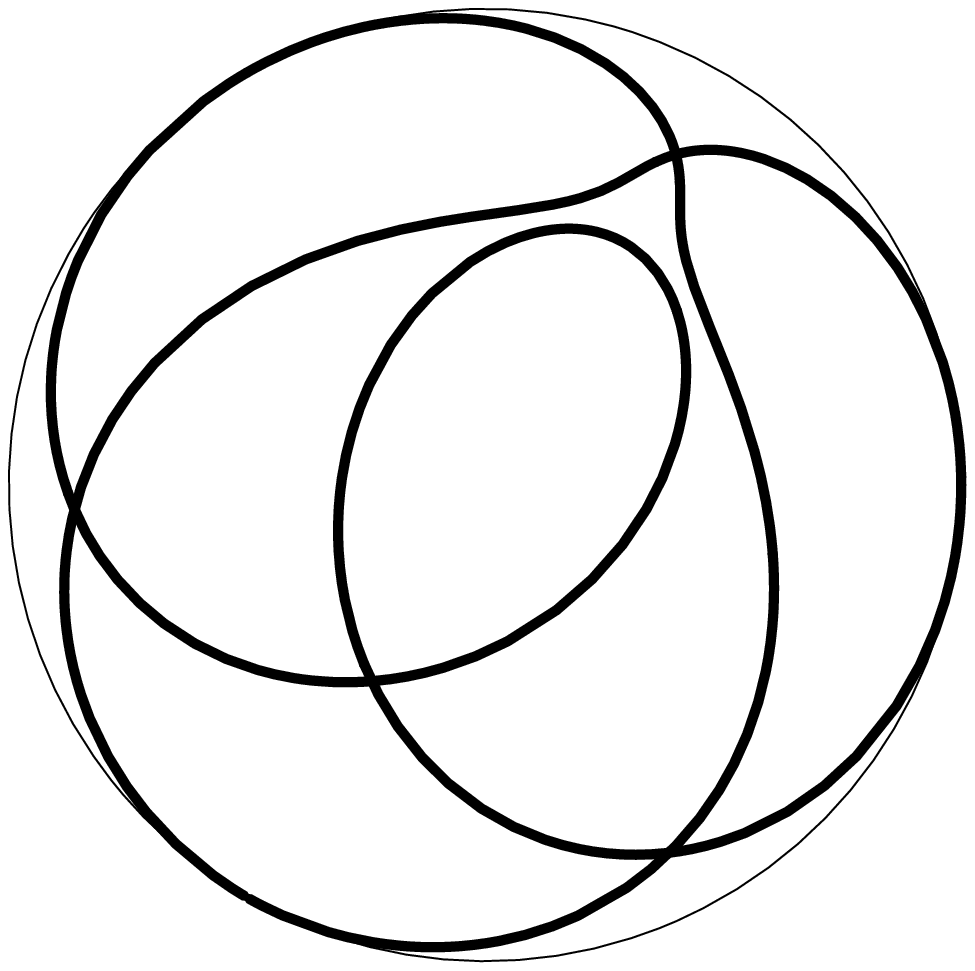} &
  \includegraphics[width=0.9in]{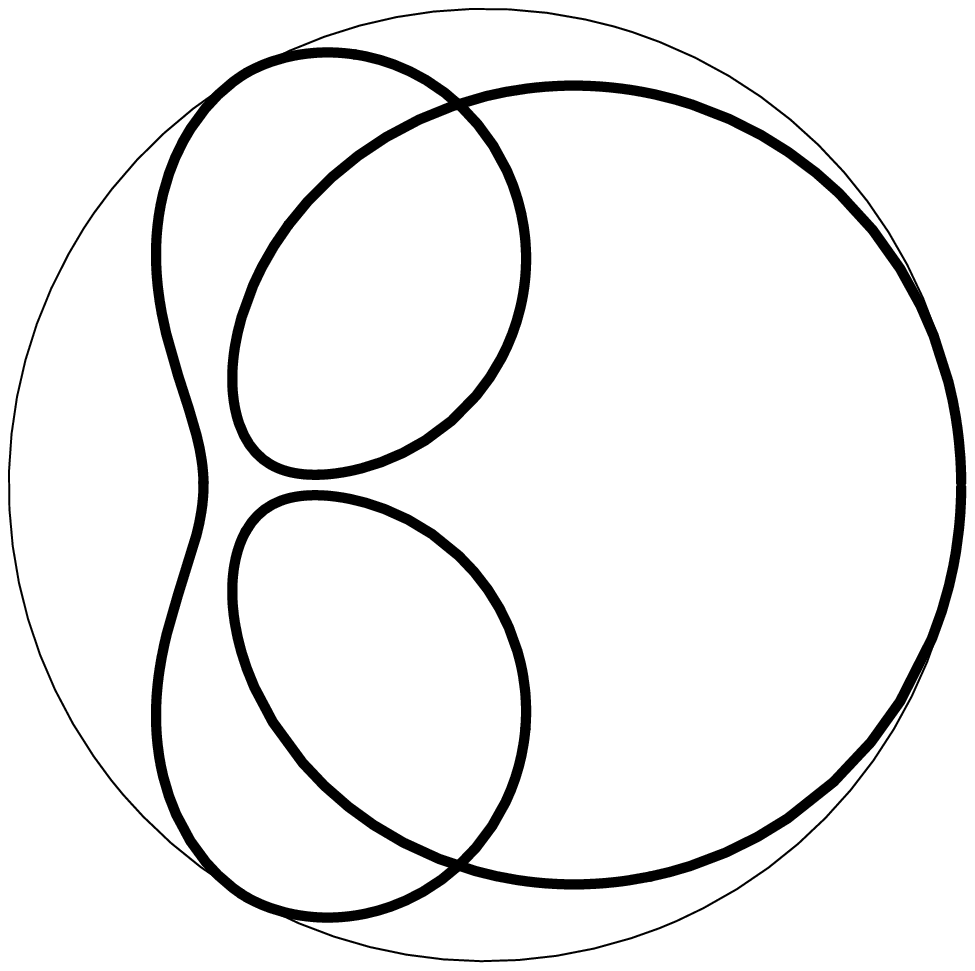} \\
  Type $(+,+,+)$ &
  Type $(-,-,-)$ &
  Type $(+,-,-)$ &
  Type $(+,+,-)$
\end{tabular}
\end{center}
\caption{Profile curves of trinoids $f_{\mu_1,\mu_2,\mu_3}$.}
\label{fig:trinoid}
\end{figure}
 In the case of \cmcone{} trinoids in $H^3$, 
 we would like to group the surfaces by the 
 signatures of $c_1$, $c_2$, $c_3$.
 For example, $f_{\mu_1,\mu_2,\mu_3}$ is called
 of type $(+,+,+)$ if $c_1$, $c_2$, $c_3$ are all positive, and it 
 is called of type $(-,+,+)$ if one of $c_1$, $c_2$, $c_3$
 is negative and the other two are positive.  
 By numerical experiment, we see that these four types $(+,+,+)$, $(-,+,+)$,
 $(-,-,+)$ and $(-,-,-)$ are topologically distinct (see 
 Figure~\ref{fig:trinoid}).
 Surfaces of type $(+,+,+)$ have total curvature less than $8\pi$, and it 
 seems that only surfaces in this class can be embedded.
\end{example}
\begin{figure}
\begin{center}
\begin{tabular}{c@{\hspace{3em}}c}
\setlength{\unitlength}{1.8cm}
\begin{picture}(2.2,2.2)(-0.5,-0.5)
  \put(-0.5,0){\vector(1,0){2}}
  \put(0,-0.5){\vector(0,1){2}}
  \put(0,0){\arc{2}{-1.57}{0}}
  \put(0.6,0){\circle*{0.05}}
  \put(0.3,0){\vector(0,1){0.2}}
  \put(0.3,0){\vector(0,-1){0.2}}
  \put(0.8,0){\vector(0,1){0.2}}
  \put(0.8,0){\vector(0,-1){0.2}}
  \put(0.707,0.707){\vector(1,1){0.2}}
  \put(0.707,0.707){\vector(-1,-1){0.2}}
  \put(0,0.5){\vector(1,0){0.2}}
  \put(0,0.5){\vector(-1,0){0.2}}
  \put(1.3,-0.15){\mbox{\footnotesize$\operatorname{Re}$}}
  \put(0.1,1.3){\mbox{\footnotesize$\operatorname{Im}$}}
  \put(1,-0.15){\mbox{\footnotesize$1$}}
  \put(0.6,-0.15){\mbox{\footnotesize$a$}}
  \put(0.3,-0.3){\mbox{\footnotesize$\tau_1$}}
  \put(0.8,-0.3){\mbox{\footnotesize$\tau_4$}}
  \put(0.9,0.9){\mbox{\footnotesize$\tau_3$}}
  \put(-0.4,0.5){\mbox{\footnotesize$\tau_2$}}
  \put(0.3,0.3){\mbox{\footnotesize$D$}}
\end{picture} &
  \raisebox{4ex}{
   \includegraphics[width=1.1in]{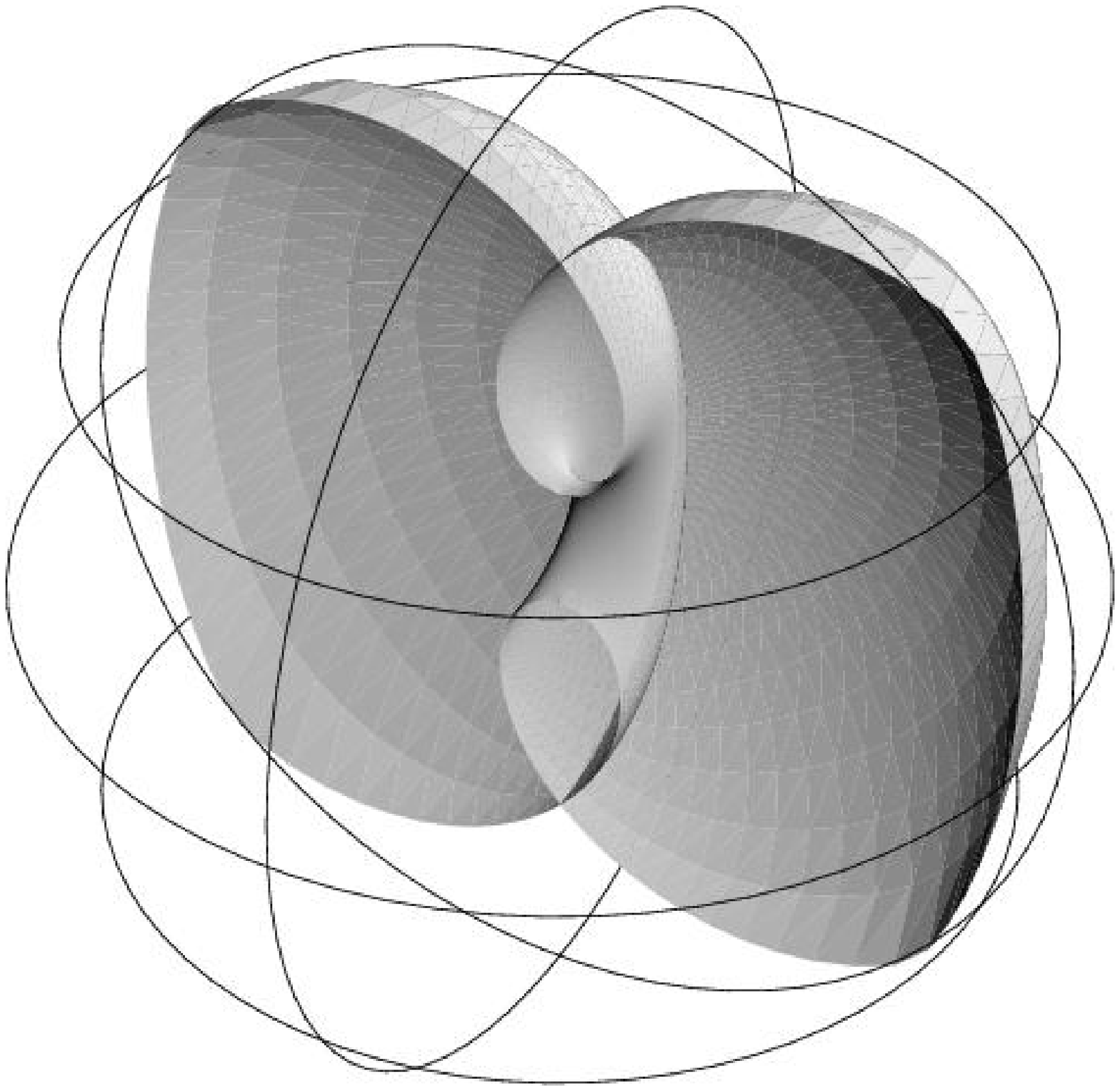}
  }
\\
   \footnotesize{Fundamental region of a $4$-noid} & 
   \footnotesize{A $4$-noid with $\TA(f)=5\pi$}
\end{tabular}
\end{center}
\caption{$4$-noid}
\label{fig:four-noid}
\end{figure}

\begin{example}[$4$-noids with $\TA(f)< 8\pi$]
\label{exa:four-noids}
 A \cmcone{} surface of genus $0$ with $4$ ends satisfies the
 Cohn-Vossen inequality $\TA(f)>4\pi$ (see
 \eqref{eq:cohn-vossen-general}).
 Though genus $0$ surfaces with an odd number of ends satisfy a sharper
 inequality (Theorem~\ref{thm:odd}), 
 it seems that the Cohn-Vossen inequality is sharp for $4$-noids, 
 by numerical experiment:
 Let $a\in (0,1)$ be a real number and
 $M=\C\cup\{\infty\}\setminus\{a,-a,a^{-1},-a^{-1}\}$.
 We set
\begin{align*}
   G &:= \frac{pz^3-z}{z^2-p}\;,\\
   Q &:=
      -\frac{\mu(\mu+2)a^2(a^2-a^{-2})^2}{(p a^4- ( 3p^2-1 )a^2+p) }\,
                 \frac{(pz^4 -(3p^2-1)z^2+p)}{
                ( z^2 - a^2 )^2(z^2-a^{-2} )^2}\,dz^2\;,
\end{align*}
 where $\mu>-1$ and $p\in\R\setminus\{0,1\}$ with
 $pa^4-(3p^2-1)a^2+p\neq 0$.
 If there exists a \cmcone{} immersion $f\colon{}M\to H^3$ with
 hyperbolic Gauss map $G$ and Hopf differential $Q$, then 
 $\TA(f)= 4\pi(2\mu+3)$.
 We shall solve the period problems using the method in \cite{ruy1}:
 Let $D:=\{ z=re^{i\theta}\in \C\; | \; 0<r<1, 0<\theta<\pi/2\}$.
 Then the Riemann surface $M$ is obtained by reflection of $D$ about 
 $\partial D$.
 Let $\tau_1$, $\tau_2$, $\tau_3$ and $\tau_4$ be 
 the reflections on the universal cover $\widetilde M$ of $M$, which are
 the lifts of the reflections on $M$ about the segment $(0,a)$ on the real
 axis, the segment $(0,i)$ on the imaginary axis, the unit circle $|z|=1$,
 and the segment $(a,1)$ on the real axis, respectively 
 (see Figure~\ref{fig:four-noid}, left).
 Let $F\colon{}\widetilde M\to\SL(2,\C)$ be a solution of \eqref{eq:Fsharp}.
 Since
\[
    \overline{G\circ\tau_j}=\sigma_j\star G,\qquad
    \overline{Q\circ\tau_j}=Q\qquad (j=1,2,3,4)
\]
 holds, where
\[
    \sigma_1=\sigma_4=\id,\qquad
    \sigma_2=\begin{pmatrix}
                i & \hphantom{-}0 \\
                0 & -i
             \end{pmatrix},\qquad
    \sigma_3=\begin{pmatrix}
                   0 & i \\
                   i & 0
             \end{pmatrix}\;,
\]
 there exist matrices $\rho_F(\tau_j)\in\SL(2,\C)$ $(j=1,2,3,4)$ 
 such that
\[
     \overline{F\circ\tau_j}=\sigma_j F\rho_F(\tau_j)\qquad
      (j=1,2,3,4)\;.
\]
 Moreover, by a similar argument to that in \cite[pp.~462--464]{ruy1},
 one can choose $F$ such that
\[
   \rho_F(\tau_1)=\id,\qquad
   \rho_F(\tau_2)=\sigma_2,\qquad
   \rho_F(\tau_j)=
       \begin{pmatrix}
          q_j & i \gamma_j^1 \\
          i\gamma_j^2 & \overline q_j
       \end{pmatrix}
   \qquad (j=3,4)\;,
\]
 where $\gamma_j^k\in\R$ and $q_j\bar q_j + \gamma_j^1\gamma_j^2=1$.
 Assume $\gamma_3^1\gamma_3^2>0$.
 Then there exists a unique solution $F$ of \eqref{eq:Fsharp} such that
\[
   \rho_F(\tau_1)=\id,\quad
   \rho_F(\tau_2)=\sigma_2,\quad
   \rho_F(\tau_3)=
       \begin{pmatrix}
          q & i \gamma \\
          i\gamma & \overline q
       \end{pmatrix},\quad
   \rho_F(\tau_3)=
       \begin{pmatrix}
          q_4 & i \gamma_4^1 \\
          i\gamma_4^2 & \overline q_4
       \end{pmatrix}\;.
\]
 For given $\mu$ and $a$, if one can choose $p$ so that 
 $\gamma_4^1=\gamma_4^2$, that is $\rho_F(\tau_j)\in\SU(2)$, then 
 there exists a \cmcone{} immersion $f$
 of $M$ into $H^3$ with hyperbolic Gauss map $G$ and Hopf differential $Q$,
 by Proposition~4.7 in \cite{ruy1}.
 
 By numerical calculation, for $\mu=-0.5$ and $a=0.8$, there exists
 $p\simeq 1.4$ such that the period problem is solved.  This surface thus 
 has $\TA(f)=8\pi$, and by continuity of the solvability of the period 
 problems, clearly there exist surfaces with $\TA(f) < 8\pi$.  
 Moreover, there exist such parameters $a$ and $p$ for
 $\mu\simeq -1$.
 So it seems that the Cohn-Vossen inequality for genus-zero $4$-ended
 \cmcone{} surfaces is sharp.
 Figure~\ref{fig:four-noid} shows the half cut of the surface with
 $\TA(f)=5\pi$.
\end{example}
\section{Reducibility}
\label{sec:red}
To state the results for higher $\TA(f)$ or $\TA(f^{\#})$, we review the
notion of {\em reducibility}.
For details, see \cite{uy3,uy7,ruy1}.
\subsection*{Metrics with conical singularities}
Let $\overline M$ be a compact Riemann surface.
A pseudometric $d\sigma^2$ on $\overline M$ is said to be an element
of $\metone(\overline M)$ if there exists a finite set of points
$\{p_1,\dots,p_n\}\subset \overline M$ such that
\begin{enumerate}
 \item $d\sigma^2$ is a conformal metric of constant curvature $1$
       on $\overline M\setminus\{p_1,\dots,p_n\}$, and
 \item $\{p_1,\dots,p_n\}$ is the set of {\em conical singularities\/}
       of $d\sigma^2$, that is, for each $j=1,\dots,n$, there exists
       a real number $\beta_j>-1$ so that $d\sigma^2$ is asymptotic to
       $c|z-p_j|^{2\beta_j}\,dz\,d\bar z$, where $z$ is a complex
       coordinate of $\overline M$ around $p_j$ and $c$ is a positive
       constant.
\end{enumerate}
We call the real number $\beta_j$ the {\em order\/} of the conical 
singularity $p_j$, and denote $\beta_j=\ord_{p_j}d\sigma^2$.
The formal sum
\begin{equation}\label{eq:metone-divisor}
   \beta_1p_1+\dots +\beta_np_n
\end{equation}
is called the {\em divisor corresponding to\/} $d\sigma^2$.

Let $d\sigma^2\in\metone(\overline M)$ with divisor
as in \eqref{eq:metone-divisor} and set $M:=\overline
M\setminus\{p_1,\dots,p_n\}$.
Then there exists a holomorphic map 
$g\colon{}\widetilde M\longrightarrow \C\cup\{\infty\}=\CP^1$
defined on the universal cover $\widetilde M$ of $M$ such that
\begin{equation}\label{eq:metone-develop}
    d\sigma^2 = \frac{4\,dg\,d\bar g}{(1+|g|^2)^2}=g^* ds_0^2\;,
\end{equation}
where $ds_0^2$ is the Fubini-Study metric of $\CP^1$.
We call $g$ the {\em developing map\/} of $d\sigma^2$.
The developing map is unique up the change
\begin{equation}\label{eq:metone-moebius}
 g\longmapsto a\star g\qquad (a\in\PSU(2))\;,
\end{equation}
where $a\star g$ denotes the M\"obius transformation of $g$
with respect to $a$ as in \eqref{eq:g-change}.
Here we write $a\in\PSU(2)$ as a $2\times 2$ matrix in $\SU(2)$
and identify $a$ with $-a$.

For each deck transformation $\tau\in\pi_1(M)$ on $\widetilde M$, 
$d\sigma^2=d\sigma^2\circ\tau$ holds.
So there exists a representation 
\begin{equation}\label{eq:metone-repr}
    \rho_g\colon{}\pi_1(M)\longrightarrow \PSU(2) 
    \quad\text{such that}\quad
    \rho\circ\tau^{-1}=\rho_g(\tau)\star g\quad
    \text{for $\tau\in\pi_1(M)$}\;.
\end{equation}
By a change of $g$ as in \eqref{eq:metone-moebius}, the corresponding
representation changes by conjugation:
\begin{equation}\label{eq:metone-repr-change}
    \rho_{a\star g}=a \rho_g a^{-1}\;.
\end{equation}
Let $\tau_j$ be a deck transformation induced from a small loop on
$\overline M$ surrounding a singularity $p_j$.
Then by \eqref{eq:metone-repr-change},
one can choose the developing map $g$ such that $\rho_g(\tau_j)$
is diagonal:
\[
    \rho_g(\tau_j) = \begin{pmatrix}
		         e^{\pi i \nu_j} & 0 \\
                         0 & e^{-\pi i \nu_j}
                     \end{pmatrix}
        \qquad (\nu_j\in\R)\;,
\]
namely, $g\circ\tau_j=e^{2\pi i \nu_j}g$.
This implies that $(z-p_j)^{-\nu_j}g$ is single-valued on a neighborhood
of $p_j$, where $z$ is a complex coordinate around $p_j$.
Then, replacing $\nu_j$ with $\nu_j+m$ ($m\in\Z$) if necessary,
we can normalize
\begin{equation}\label{eq:metone-g-norm}
     g = (z-p_j)^{\nu_j}\bigl(g_0 + g_1 (z-p_j) + g_2 (z-p_j)^2 +\dots\bigr)
          \qquad (g_0\neq 0)\;.
\end{equation}
By definition of the order and by equation \eqref{eq:metone-develop}, we have
\[
    \nu_j = \beta_j+1 \qquad\text{or}\qquad -\beta_j-1\;.
\]

\begin{definition}\label{def:metone-reducible}
 A pseudometric $d\sigma^2\in\metone(\overline M)$ is called
 {\em reducible\/} if the representation  $\rho_g$ can be diagonalized 
 simultaneously, where $g$ is the developing map of $d\sigma^2$.
 More precisely, a reducible metric $d\sigma^2$ is called 
 {\em $\Hyp^3$-reducible\/} if the representation is trivial, and called
 {\em $\Hyp^1$-reducible\/} otherwise.
 A pseudometric $d\sigma^2$ is called {\em irreducible\/} if it
 is not reducible.
\end{definition}

By definition, a developing map $g$ of an $\Hyp^3$-reducible metric is 
a meromorphic function on $\overline M$ itself.
Moreover, by \eqref{eq:metone-g-norm}, all conical singularities 
have integral orders, which coincide with the branching orders of the 
meromorphic function $g$.
In this case, for any $a\in\PSL(2,\C)$, $g_a:=a\star g$ induces a new
metric $d\sigma^2_a:=g_a^{\star}ds_0^2\in\metone(\overline M)$ with
the same divisor as $d\sigma^2$.
Since $d\sigma_a^2=d\sigma^2$ if $a\in\PSU(2)$, we have a non-trivial
deformation of $d\sigma^2$ preserving the divisor parametrized by a real
$3$-dimensional space $\Hyp^3=\PSL(2,\C)/\PSU(2)$, which is the
hyperbolic $3$-space.

On the other hand, assume $d\sigma^2\in\metone(\overline M)$ is
$\Hyp^1$-reducible.
Then there exists a developing map $g$ such that the image of $\rho_g$
consists of diagonal matrices.
Let $t$ be a positive real number and set
\[
    g_t:= t g =\begin{pmatrix}
		  t^{1/2} & 0 \\ 
                     0    & t^{-1/2}
               \end{pmatrix}\star g\;.
\]
Then by \eqref{eq:metone-repr-change}, $\rho_{g_t}=\rho_g$ holds.
Thus, $g_t$ induces a new metric $d\sigma_t^2\in\metone(\overline M)$.
So we have one parameter family of pseudometrics $\{d\sigma_t^2\}$
preserving the corresponding divisor.
This family is considered as a deformation of pseudometric parameterized
by a geodesic line in $\Hyp^3$.
For details, see the Appendix in \cite{ruy1}.

We introduce a criterion for reducibility:
\begin{lemma}\label{lem:metone-red}
  A metric $d\sigma^2\in\metone(\overline M)$ is reducible if and only
  if there exists a developing map such that $d\log g$ is a meromorphic 
  $1$-form on $\overline M$.
\end{lemma}
\begin{proof}
 Assume $d\sigma^2$ is reducible.
 Then one can choose the developing map $g$ such that $\rho_g$ is
 diagonal.
 Then for each deck transformation $\tau\in\pi_1(M)$,
\[
     g\circ\tau^{-1} = \begin{pmatrix}
			   e^{\pi i\nu_{\tau}}  &   0 \\
                              0         &   e^{-\pi i\nu_{\tau}}
                       \end{pmatrix}\star g
                     = e^{2\pi i\nu_{\tau}} g
         \qquad (\nu_{\tau}\in\R)
\]
 holds.
 Hence we have $\log g\circ\tau = g + 2\pi i\nu_{\tau}$.
 Differentiating this, $d\log g\circ\tau=d\log g$ holds.
 Hence $d\log g$ is single-valued on $\overline M$.

 Conversely, we assume $d\log g$ is well-defined on $\overline M$ 
 for a developing map $g$.
 Then $\log g\circ\tau- \log g$ is a constant. 
 Hence we have $g\circ\tau=\lambda_\tau g$ for some constant
 $\lambda_{\tau}$.
 Then $\rho_g$ is diagonal.
\end{proof}
\subsection*{Relationship with CMC-1 surfaces} 
Let $f\colon{}\overline M_{\gamma}\setminus\{p_1,\dots,p_n\}\to H^3$ be
a complete conformal \cmcone{} immersion, where $\overline M_{\gamma}$
is a compact Riemann surface.
If $\TA(f)<\infty$, then the pseudometric $d\sigma^2$ as in
\eqref{eq:pseudo} is considered as an element of 
$\metone(\overline{M}_{\gamma})$ 
(see \cite{Bryant}), and the secondary Gauss map $g$ is the developing
map of $d\sigma^2$. 
Let $\{q_1,\dots,q_m\}$ be the set of umbilic points of $f$, that is 
the zeros of $Q$ and set $\xi_k:=\ord_{q_k}Q$ ($k=1,\dots,m$).
Then by \eqref{eq:metrics-rel}, $d\sigma^2$ has a conical singularity
of order $\xi_k$ for each $k=1,\dots,m$.
Hence the divisor of $d\sigma^2$ is in the form
\begin{equation}\label{eq:cmcone-divisor}
    \mu_1 p_1 + \dots +\mu_n p_n + 
    \xi_1 q_1 + \dots +\xi_m q_m\;,
\end{equation}
where the $\mu_j$ ($j=1,\dots,n$) are the branch orders of $g$ at each 
$p_j$.

Let $F$ be a holomorphic lift of $f$ as in \eqref{eq:bryant}.
Then there exists a representation $\rho_F\colon{}\pi_1(M)\to \SU(2)$
as in \eqref{eq:F-repr}.
By \eqref{eq:g-F}, the secondary Gauss map $g$ of
$F$ changes as $g\circ \tau^{-1} = \rho_F(\tau)\star g$
for each deck transformation $\tau\in\pi_1(M)$.
Hence the representation $\rho_g$ defined in \eqref{eq:metone-repr} satisfies
\begin{equation}\label{eq:g-F-repr}
   \rho_g(\tau) = \pm \rho_F(\tau)\qquad\bigl(\tau\in\pi_1(M)\bigr)\;.
\end{equation}

The immersion $f$ is called $\Hyp^3$-reducible
(resp.~$\Hyp^1$-reducible) if the corresponding pseudometric $d\sigma^2$
is $\Hyp^3$-reducible (resp.~$\Hyp^1$-reducible).
\begin{lemma}\label{lem:hyp-3-red-f-sharp}
  A \cmcone{} immersion $f\colon{}M\to H^3$ is $\Hyp^3$-reducible
  if and only if the dual immersion $f^{\#}$ is well-defined on $M$.
\end{lemma}
\begin{proof}
  Let $F$ be a lift of $f$.
  Then $f^{\#}=F^{-1}(F^{-1})^*$ is well-defined on $M$ if and only if
  $\rho_F=\pm\id$.
  This is equivalent to $\rho_g$ being the trivial representation, 
  by \eqref{eq:g-F-repr}.
\end{proof}
\section{The case $\TA(f^\#) \leq 8 \pi$}
\label{sec:dualcurvatureatmost8pi}
We now have enough notations and facts to describe results on the case
$\TA(f^\#) \leq 8 \pi$ \cite{ruy3}.

Let $f\colon{}\overline M_{\gamma}\setminus\{p_1,\dots,p_n\}\to H^3$
be a complete, conformal \cmcone{} immersion, where 
$\overline M_{\gamma}$ is a Riemann surface of genus $\gamma$.
Now we assume $\TA(f^{\#})\leq 8\pi$.
If the hyperbolic Gauss map $G$ has an essential singularity at any end
$p_j$, then $\TA(f^{\#})=+\infty$, 
since $\TA(f^{\#})$ is the area of the image of $G$.  
So $G$ is meromorphic on all of $\overline M_{\gamma}$.
In particular, $\TA(f^{\#})=4\pi\deg G=0$, $4\pi$, or $8\pi$.  

\begin{table}
\begin{footnotesize}
  \begin{tabular}{|l||c|l|l|l|l|}
    \hline
      Type & {\footnotesize$\TA(f^{\#})$} & {\footnotesize Reducibility} 
          & Status & 
                                \multicolumn{1}{|c|}{c.f.} \\
    \hline\hline
      $\gO(0)$  & 0 & $\Hyp^3$-red.
                & classified$^0$
                & Horosphere    \\
    \hline
    \hline
      $\gO(-4)$ & $4\pi$ & $\Hyp^3$-red.
                & classified
                & \parbox[t]{4.4cm}{
                    Duals of Enneper cousins \\
                    \hspace*{\fill}\cite[Example~5.4]{ruy1}}    \\
    \hline
      $\gO(-2,-2)$ & $4\pi$
                & 
               reducible
                                & classified
                & \parbox[t]{4.6cm}{
                Catenoid cousins \\
                and warped
                catenoid cousins with embedded ends (i.e. $\delta=1$)\par
                \cite[Example~2]{Bryant},\par
                \cite{uy1,ruy3,ruy4}\hfill} \\
    \hline
    \hline
      $\gO(-5)$ & $8\pi$ & $\Hyp^3$-red.
                & classified
                & \cite{ruy3} \\
    \hline
      $\gO(-6)$ & $8\pi$ & $\Hyp^3$-red.
                & classified
                & \cite{ruy3} \\
    \hline
      $\gO(-2,-2)$ 
                & $8\pi$ & red.
                & classified
                & \parbox[t]{4.5cm}{Double covers of 
                    catenoid cousins and warped 
                    catenoid cousins with  $\delta=2$ \\
                    \cite[Theorem~6.2]{uy1},\\
                    \cite{ruy3,ruy4}}
                    \\
    \hline
      $\gO(-1,-4)$ 
                & $8\pi$ & $\Hyp^3$-red.
                & classified$^0$
                & \cite{ruy3} \\
    \hline
      $\gO(-2,-3)$ 
                & $8\pi$ & $\Hyp^1$-red.
                & classified
                & 
                \cite{ruy3} \\
    \hline
      $\gO(-2,-4)$ 
                & $8\pi$ & $\Hyp^1$-red.
                & classified &
              \cite{ruy3} \\
                &  & $\Hyp^3$-red.
                & classified &
              \cite{ruy3} \\
    \hline
      $\gO(-3,-3)$ 
                & $8\pi$ & red.
                & existence
                & \cite{ruy3} \\
    \hline
      $\gO(-1,-1,-2)$ 
                & $8\pi$ &  $\Hyp^3$-red.
                & classified$^0$
                & \cite{ruy3} \\
    \hline
      $\gO(-1,-2,-2)$ 
                & $8\pi$ & $\Hyp^1$-red.
                & classified
                & \cite{ruy3} \\

                &  & $\Hyp^3$-red.
                & classified    
                & \cite{ruy3} 
\\
    \hline
      $\gO(-2,-2,-2)$ 
                & $8\pi$ & irred.
                & classified
                &  \cite[Theorem~2.6]{uy6}            \\
                &  & $\Hyp^1$-red.
                & existence$^+$
                & \cite{ruy3}    \\
                &  & $\Hyp^3$-red.
                & existence$^+$
                & \cite{ruy3}    \\
    \hline
      $\gI(-3)$ 
                & $8\pi$ &
                & unknown
                & \\ 
    \hline 
      $\gI(-4)$ 
                & $8\pi$ &
                & existence
                & Chen-Gackstatter cousins \cite{ruy3}    \\
    \hline
      $\gI(-1,-1)$ 
                & $8\pi$ &
                & unknown$^{+}$
                & \cite{ruy3}    \\
    \hline
      $\gI(-2,-2)$ 
                & $8\pi$ &
                & existence
                & Genus $1$ catenoid cousins~\cite{rs}    \\
    \hline
  \end{tabular}
\end{footnotesize}
\caption{\cmcone{} surfaces in $H^3$ with $\TA(f^{\#})\leq 8\pi$ \cite{ruy3}.}
\label{tab:class}
\end{table}
Since $f^\#$ has finite total curvature, the Hopf differential
$Q^{\#}=-Q$ can be extended to $\overline M_{\gamma}$ as a meromorphic
$2$-differential~\cite[Proposition 5]{Bryant}.
Hence $d_j = \ord_{p_j}Q$ is finite for each $j=1,\dots,n$.  
Our results from \cite{ruy3} are shown in Table~\ref{tab:class}.
In the table, 
\begin{itemize}
  \item {\em classified\/} means the complete list of the surfaces
        in such a class is known (and this means not only that we know 
        all the possibilities for the form of the data $(G,Q)$, 
        but that we also know exactly for which $(G,Q)$ the period problems 
        of the immersions are solved).  
  \item {\em classified\/$^0$} means there exists a unique surface
        (up to isometries of $H^3$ and deformations that come from its
           reducibility).  
  \item {\em existence\/} means that examples exist, but they are not yet 
                  classified.  
  \item {\em existence{$^+$}} means that all possibilities for the 
        data $(G,Q)$ are determined, 
        but the period problems are solved only for special cases.  
  \item {\em unknown\/} means that neither existence nor non-existence is 
        known yet.
  \item {\em unknown{$^+$}} means that all possibilities for the 
        data $(G,Q)$ are determined, but the period 
	problems are still unsolved.
 \end{itemize}
Any class and type of reducibility not listed in Table \ref{tab:class} cannot 
contain surfaces with $\TA(f^{\#}) \leq 8\pi$.  For example, any irreducible 
or $\Hyp^3$-reducible surface of type $\gO(-2,-3)$ must have dual total
absolute curvature at least $12 \pi$.  
\begin{table}
\begin{footnotesize}
  \begin{tabular}{|l||c|l|l|}
    \hline
      Type & $\TA$ & The surface & \multicolumn{1}{|c|}{c.f.} \\
    \hline\hline
      $\gO(0)$  & 0 
                & Plane
                &     \\
    \hline
      $\gO(-4)$ & $4\pi$ 
                & Enneper's surface
                &    \\
    \hline
      $\gO(-5)$ & $8\pi$ 
                & 
                & \cite[Theorem~6]{Lopez}   \\
    \hline
      $\gO(-6)$ & $8\pi$ 
                & 
                & \cite[Theorem~6]{Lopez}   \\
    \hline
      $\gO(-2,-2)$ 
                & $4\pi$
                & Catenoid
                & \\
                & $8\pi$
                & Double cover of the catenoid
                & \\
    \hline
      $\gO(-1,-3)$ 
                & $8\pi$
                & 
                & \cite[Theorem~5]{Lopez}\\
    \hline
      $\gO(-2,-3)$ 
                & $8\pi$
                & 
                & \cite[Theorem~4, 5]{Lopez}\\
    \hline
      $\gO(-2,-4)$ 
                & $8\pi$
                & 
                & \cite[Theorem~5]{Lopez}\\
    \hline
      $\gO(-3,-3)$ 
                & $8\pi$
                & 
                & \cite[Theorem~4]{Lopez}\\
    \hline
      $\gO(-1,-2,-2)$ 
                & $8\pi$
                & 
                & \cite[Theorem~5]{Lopez}\\
    \hline
      $\gO(-2,-2,-2)$ 
                & $8\pi$
                & 
                & \cite[Theorem~5]{Lopez}\\
    \hline
      $\gI(-4)$ 
                & $8\pi$
                & Chen-Gackstatter surface
                & \cite[Theorem~5]{Lopez}, \cite{cg}\\
    \hline
  \end{tabular}
\end{footnotesize}
\caption{
 The classification of complete minimal surfaces in $\R^3$ with 
 $\TA\leq 8\pi$ (\cite{Lopez}), for comparison with Table~\ref{tab:class}.}
\label{tab:minimal}
\end{table}

Table \ref{tab:minimal} shows the corresponding results for minimal surfaces
in $\R^3$, the classification of complete minimal surfaces with $\TA\leq
8\pi$ \cite{Lopez}.
Comparing these two tables, one sees differences between the classes of
minimal surfaces with $\TA\leq 8\pi$ and the classes of 
\cmcone{} surfaces with $\TA(f^{\#})\leq 8\pi$.
For example, there exist no mimimal surfaces of classes $\gO(-1,-4)$ and 
$\gO(-1,-1,-2)$ with $\TA\leq 8\pi$, but \cmcone{} surfaces of such types 
do exist.  
\section{The case of $\TA(f)\leq 8\pi$}
\label{sec:higher}
In the remainder of this paper, we shall give new results on the case of
higher $\TA(f)$.  
\subsection*{Preliminaries}
First, we give further notations and facts that will be needed in our
discussion.
\subsubsection*{Orders of the Gauss maps}
Let $\overline M_{\gamma}$ be a compact Riemann surface of genus
$\gamma$.
For a complete conformal \cmcone{} immersion
$f\colon{}M=\overline M_{\gamma} \setminus\{p_1,\dots,p_n\}\to H^3$
with $\TA(f)<\infty$, we define $\mu_j$ and $\mu^{\#}_j$ to be the
branching orders of the Gauss maps $g$ and $G$, respectively, at an end
$p_j$.  
Then the pseudometric $d\sigma^2$ as in \eqref{eq:pseudo} has a conical 
singularity of order $\mu_j>-1$ at each end $p_j$ ($j=1,\dots,n$).
Let $d_j=\ord_{p_j}Q$ $(j=1,\dots,n)$.
Then an end $p_j$ is regular if and only if $d_j\geq -2$ (see
Section~\ref{sec:prelim}, or \cite{uy1}).
If an end $p_j$ is irregular, then $\mu_j^{\#}=\infty$.  
At a regular end $p_j$, the relation \eqref{eq:schwarz} implies that
the Hopf differential $Q$ expands as 
\begin{equation}\label{eq:q-first}
    Q = \left(\frac{1}{2}\frac{c_j-c_j^{\#}}{(z-p_j)^2} + \dots \right)
     \,dz^2 \; , 
\end{equation}
where
\begin{equation}\label{eq:c-j}
    c_j = -\frac{1}{2}\mu_j(\mu_j+2),\qquad
    c_j^{\#} = -\frac{1}{2}\mu_j^{\#}(\mu_j^{\#}+2)
\end{equation}
and $z$ is a local complex coordinate around $p_j$.

Let $\{q_1,\dots,q_m\} \subset M$ be the $m$ umbilic points of the surface, 
and let $\xi_k=\ord_{q_k}Q $.  
Since the total order of a holomorphic $2$-differential is
$-2\chi(\overline M_{\gamma})$, we have
\begin{equation}\label{eq:fact-a}
   \sum_{j=1}^n d_j+\sum_{k=1}^m 
             \xi_k=4\gamma-4, 
   \qquad
   \text{in particular,}
   \quad  \sum_{j=1}^n d_j\le 4\gamma-4\;.
\end{equation}
By \eqref{eq:metrics-rel} and \eqref{eq:schwarz}, it holds that
\begin{align}\label{eq:fact-h}
    \xi_k &= \text{[the branching order of $G$ at $q_k$]}
           = \text{[the branching order of $g$ at $q_k$]}\\
          &= \ord_{q_k} d\sigma^2
           = \ord_{q_k} Q\;.
\nonumber
\end{align}
As in (2.4) of \cite{ruy3}, the Gauss-Bonnet theorem for 
$(\overline M_{\gamma},d\sigma^2)$ implies 
\begin{equation}\label{eq:ta-non-dual}
   \frac{\TA(f)}{2\pi}=\chi(\overline M_{\gamma})+
                    \sum_{j=1}^{n}\mu_j + \sum_{k=1}^m \xi_k 
    =(2\gamma-2)+ \sum_{j=1}^{n}\mu_j + \sum_{k=1}^m \xi_k 
\end{equation}
as well as 
\begin{equation}\label{eq:ta-dual}
   \frac{\TA(f^{\#})}{2\pi}=\chi(\overline M_{\gamma})+
                    \sum_{j=1}^{n}\mu_j^{\#} + \sum_{k=1}^m \xi_k \;,
\end{equation}
which is obtained from the Gauss-Bonnet theorem for 
$d\sigma^2{}^{\#}=(-K^{\#})ds^2{}^{\#}$ \cite{ruy3}.
Combining this with \eqref{eq:fact-a}, we have 
\begin{equation}\label{eq:fact-b}
  \frac{\TA(f)}{2\pi}=2\gamma-2+\sum_{j=1}^n(\mu_j-d_j)\;.
\end{equation}
Proposition~4.1 in \cite{uy1} implies that
\begin{equation}\label{eq:fact-c}
   \mu_j-d_j>1, \qquad
   \text{in particular,} \qquad
   \mu_j-d_j\geq 2\quad 
   \text{if $\mu_j\in\Z$}\;.
\end{equation}
An end $p_j$ is regular if and only if $d_j\geq -2$, and then $G$ is 
meromorphic at $p_j$. Thus 
\begin{equation}\label{eq:fact-f}
  \text{$\mu_j^{\#}$ is a non-negative integer if $d_j\geq -2$}\;.
\end{equation}
In this case, it holds that (Lemma~3 of \cite{uy5})
\begin{equation}\label{eq:fact-dual}
 \mu_j^{\#}-d_j\geq 2\quad 
  \text{and the equality holds if and only if $p_j$ is embedded.}
\end{equation}  
By Proposition 4 of \cite{Bryant}, 
\begin{equation}\label{eq:fact-g}
   \mu_j > -1 \; , 
\end{equation}
hence equation \eqref{eq:q-first} implies 
\begin{equation}\label{eq:fact-d}
   \mu_j=\mu_j^\#\qquad 
   \text{if}\quad d_j\geq -1\;.
\end{equation}
Finally, we note that
\begin{equation}\label{eq:fact-e}
  \parbox[t]{10cm}{
    any meromorphic function on a Riemann surface 
    $\overline M_{\gamma}$ of genus $\gamma\geq 1$
    has at least three distinct branch points.
  }
\end{equation}
To prove this, let $\varphi$ be a meromorphic function on $\overline M_\gamma$ 
with branch points $\{q_1$, $\dots$, $q_N\}$ with branching order $\nu_k$ at 
$q_k$.  Then the Riemann-Hurwitz relation implies 
\[
   2\deg \varphi = 2 - 2\gamma + \sum_{k=1}^{N}{\nu_k}\;.
\]
On the other hand, since the multiplicity of $\varphi$ at $q_k$ 
is $\nu_k+1$, $\deg\varphi\geq \nu_k+1$ ($k=1,\dots,m$).  Thus 
\[
   (N-2)\deg\varphi \geq 2(\gamma-1)+N\;.
\]
If $\gamma\geq 1$, then $\deg \varphi \geq 2$, and so $N \geq 3$.  
\begin{remark*}
 Facts \eqref{eq:fact-b} and \eqref{eq:fact-c} imply that, 
 for \cmcone{} surfaces, equality never holds in the Cohn-Vossen 
 inequality (see \eqref{eq:cohn-vossen-general} and \cite{uy1}).
\end{remark*}
\subsubsection*{Flux for CMC-1 surfaces}
Let $f\colon{}\overline M_{\gamma}\setminus\{p_1,\dots,p_n\}\to H^3$ be
a complete \cmcone{} immersion.
For each end $p_j$, the {\em flux\/} at $p_j$ is defined as
(\cite{ruy2})
\begin{equation}\label{eq:flux}
 \Fl_j:=\frac{1}{2\pi i}\int_{\tau_j}
           \begin{pmatrix}
	      G & -G^2 \\
              1 & -G\hphantom{^2}
	   \end{pmatrix}
           \frac{Q}{dG}\in\sl(2,\C)\qquad (j=1,\dots,n)\;,
\end{equation}
where $G$ and $Q$ are the hyperbolic Gauss map and the Hopf differential
of $f$ respectively, and $\tau_j$ is a loop surrounding the end $p_j$.
Then the following balancing formula holds (Theorem~1 in \cite{ruy2}):
\begin{equation}\label{eq:balancing}
 \sum_{j=1}^n \Fl_j=0\;.
\end{equation}
Moreover, it holds that (Proposition~2 and Corollary 5 in \cite{ruy2}):
\begin{proposition}\label{prop:flux}
 For a complete \cmcone{} immersion 
 $f\colon{}\overline M_{\gamma}\setminus\{p_1,\dots,p_n\}\to H^3$,
 \begin{enumerate}
  \item\label{item:flux:1}
       If an end $p_j$ is regular and $\ord_{p_j}Q=-2$, 
       then $\Fl_j\neq 0$.
  \item\label{item:flux:2}
       If and end $p_j$ is regular and embedded, $\Fl_j=0$ if and only
       if $\ord_{p_j}Q\geq 0$.
 \end{enumerate}
\end{proposition}
Then by the balancing formula \eqref{eq:balancing}, we have
\begin{corollary}\label{cor:flux}
 There exists no complete \cmcone{} surface of finite total curvature
 with only one end $p$ that is regular, such that either one of the
 following holds{\rm :}
 \begin{enumerate}
  \item $\ord_pQ=-2$.
  \item $\ord_pQ<0$ and the end is embedded.
 \end{enumerate}
\end{corollary}


\subsection*{Results for \boldmath$\TA(f)\leq 8\pi$}
First, we prepare the following lemma.

\begin{lemma}\label{lem:general-class}
 Let $f\colon{}M \to H^3$ be a complete \cmcone{} immersion of genus
 $\gamma$ and $n$ ends with $\TA(f) \leq 2\pi\rho$.
 If $f$ is not totally umbilic {\rm (}not a horosphere{\rm )}, 
 then the following hold{\rm:}
\begin{enumerate}
 \item $2 \gamma < \rho+1$ and $1 \leq n < \rho-2\gamma + 2$.  
 \item If $n=1$, then $2\gamma-\rho-3<d_1 \leq 4\gamma-4$ and $d_1 \neq -2$.  
 \item If $\gamma=n=1$, then $-\rho-1<d_1 \leq -3$.  
 \item If $2 \leq n = \rho+1-2\gamma$, then $d_j=-2$ at all ends.  
 \item If $1= n = \rho+1-2\gamma$, then $d_1 \geq 0$ and $\mu_1 = 2 + d_1$.  
\end{enumerate}
\end{lemma}
\begin{proof}
 The first item of the lemma is obtained from the Cohn-Vossen inequality
 \eqref{eq:cohn-vossen-general}.
 In particular, if $n=1$, \eqref{eq:fact-b}, \eqref{eq:fact-g} and 
 \eqref{eq:fact-a} imply
 \[
    \rho\geq 2\gamma-2+\mu_1-d_1>2\gamma-3-d_1\qquad
    \text{and}\qquad
     d_1\leq 4\gamma-4\;.
 \]
 In this case, by the balancing formula \eqref{eq:balancing},
 the flux $\Fl_1$ must be vanish.
 Hence by Corollary~\ref{cor:flux},
 $d_1 \neq -2$ holds, and the second item of the theorem follows.  
 Even more particularly, if $\gamma=n=1$, then  $-\rho-1 < d_1 \leq 0$
 and $d_1 \neq -2$.  
 Assume $d_1\geq -1$.
 In this case, the end is regular, and then $G$ is a meromorphic
 function on $\overline M_{\gamma}$.
 On the other hand, \eqref{eq:fact-a} implies that there is at most 
 one umbilic point.
 Since a branch point of $G$ is an umbilic point or an end, this implies 
 that the number of branch points of $G$ is at most $2$, which
 contradicts  \eqref{eq:fact-e}.
 Hence the third item is proven.
 
 Suppose $n=\rho-2\gamma$.  
 Then \eqref{eq:fact-a} implies 
\begin{equation}\label{eq:critical-sum}
   n+1 \geq \sum_{j=1}^n (\mu_j-d_j) \; , 
\end{equation}
 and we consider two cases:  
\paragraph{\it Case 1}
 If $n \geq 2$, then \eqref{eq:fact-c} implies that 
 $1 < \mu_j - d_j < 2$ for all $j$, so $\mu_j \not\in \Z$ 
 for all $j$, 
 and hence \eqref{eq:fact-d} implies that $d_j \leq -2$ for all $j$.  
 But by \eqref{eq:critical-sum} and \eqref{eq:fact-g},
 we have $-2n\leq \sum_{j=1}^n d_j$,
 and so $d_j=-2$ for all $j$.
\paragraph{\it Case 2}
 If $n=1$, then $1 < \mu_1-d_1 \leq 2$ holds because of
 \eqref{eq:critical-sum} and \eqref{eq:fact-c}.
 Hence by \eqref{eq:fact-g}, $d_1 \geq -2$.
 But Corollary~\ref{cor:flux} implies $d_1\geq -1$.
 Then by \eqref{eq:fact-d}, $\mu_1\in\Z$ and $\mu_1-d_1=2$ holds.
 Suppose $d_1 = -1$.  
 Then $\mu_1^{\#}=\mu_1=d_1+2=1$, and then by \eqref{eq:fact-dual},
 the only end $p_1$ is regular and embedded.
 This contradicts to  \ref{item:flux:2} and \eqref{eq:balancing}.
 Hence $d_1 \geq 0$.  
\end{proof}
\begin{table}
\small
  \begin{tabular}{|l||c|l|c|l|}
    \hline
      Type & $\TA(f)$ & Reducibility &  Status & \multicolumn{1}{|c|}{cf.} \\
    \hline\hline
      $\gO(0)$  & 0 & $\Hyp^3$-red.
                & classified
                & Horosphere    \\
    \hline
      $\gO(-4)$ & $4\pi$ & $\Hyp^3$-red.
                & classified
                & Enneper cousins \hspace*{\fill}\cite{Bryant}    \\
    \hline
      $\gO(-5)$ & $8\pi$ & $\Hyp^3$-red.
                & classified
                & Same as ``dual'' case \\
    \hline
      $\gO(-6)$ & $8\pi$ & $\Hyp^3$-red.
                & classified
                & Same as ``dual'' case \\
    \hline\hline
      $\gO(-2,-2)$ & $(0,8\pi]$ & $\Hyp^1$-red.
                & classified
                & \parbox[t]{3.1cm}{
                     {\footnotesize Catenoid cousins and}\newline
                     {\footnotesize their $\delta$-fold covers}\newline
                \hspace*{\fill}{\footnotesize\cite[Ex.~2]{Bryant},\cite{uy1}}}   \\
    \hline
      $\gO(-2,-2)$ 
              &
               \parbox[t]{0.7cm}{
                    \hspace*{\fill}$4\pi$\hspace*{\fill}\\
                    \hspace*{\fill}$8\pi$\hspace*{\fill}}
                & $\Hyp^3$-red.
                & classified
                & \parbox[t]{3.1cm}{
                     {\footnotesize Warped cat.~cous.~$l=1$}\newline
                     {\footnotesize Warped cat.~cous.~$l=2$}\newline
 \hspace*{\fill}{\footnotesize \cite[Thm~6.2]{uy1}, Exa. \ref{exa:catenoid}}}   \\
    \hline
      $\gO(-1,-4)$
                & $8\pi$ & $\Hyp^3$-red.
                & classified
                & Same as ``dual'' case \\
    \hline
      $\gO(-2,-4)$ 
                & $8\pi$ 
                & $\Hyp^3$-red.
                & classified
                & Same as ``dual'' case
    \\
                & $(4\pi,8\pi)$
                & $\Hyp^1$-red.
                & existence
                & Remark~\ref{rem:o-2-4} \\
    \hline
      $\gO(-2,-5)$ 
                & $8\pi$
                & $\Hyp^1$-red.
                & existence
                & Remarks~\ref{rem:o-2-5}, \ref{rem:o-3-4} \\
    \hline
      $\gO(-3,-3)$ 
                &
                & reducible
                & unknown
                & Remark~\ref{rem:o-3-3} \\
    \hline
      $\gO(-3,-4)$ 
                & $8\pi$
                & reducible
                & unknown
                & Remark~\ref{rem:o-3-4} \\
    \hline\hline
      $\gO(0,-2,-2)$
                & $(4\pi,8\pi)$
                & $\Hyp^1$-red.
                & classified
                & Proposition~\ref{prop:o0-2-2} \\
   \hline
      $\gO(-1,-2,-3)$
                & $8\pi$
                & $\Hyp^1$-red.
                & unknown
                &      \\
   \hline
      $\gO(-1,-1,-2)$
                & $8\pi$ 
                & $\Hyp^3$-red.
                & classified
                & Same as ``dual'' case\\
    \hline
      $\gO(-1,-2,-2)$
                & $8\pi$
                & $\Hyp^3$-red.
                & classified
                & Same as ``dual'' case \\
                & $(4\pi,8\pi)$
                & $\Hyp^1$-red.
                & classified
                &  Proposition \ref{prop:o-1-2-2} \\
                & $8\pi$
                & $\Hyp^1$-red.
                & classified
                &  Proposition \ref{prop:o-1-2-2-a} \\
    \hline
      $\gO(-2,-2,-2)$
                & $(4\pi,8\pi]$
                & 
                & existence
                & \parbox[t]{3.1cm}{
                     {\footnotesize Classified for irred.}\\
                     {\footnotesize embedded end case}\\
                     {\footnotesize \hspace*{\fill}\cite{uy6}}
                  }\\
    \hline
      $\gO(-2,-2,-3)$
                & 
                & irred./$\Hyp^1$-red.
                & unknown
                &  \\
    \hline
      $\gO(-2,-2,-4)$
                & $8\pi$
                & irred./$\Hyp^1$-red.
                & unknown
                &  \\
    \hline
      $\gO(-2,-3,-3)$
                & $8\pi$
                & irred./$\Hyp^1$-red.
                & unknown
                &  \\
    \hline
    \hline
      $\gO(-2,-2,-2,-2)$
                &
                & 
                & existence
                & Example~\ref{exa:four-noids} \\
    \hline
      $\gO(-2,-2,-2,0)$
                & $8\pi$
                & 
                & existence
                & Remark~\ref{rem:o-2-2-20}\\
    \hline
      $\gO(-2,-2,-2,d)$
                & $8 \pi$ when 
                & 
                & unknown
                &  \\
      \hspace*{1em}$d=-3,-2,-1,1$
                & $d \geq -1$
                &
                &
                & \\
    \hline
    \hline
      $\gO(-2,-2,-2,-2,-2)$
                & $8\pi$
                &
                & unknown
                & Corollary~\ref{cor:8pi-class} \\
    \hline
    \hline
      $\gI(-3)$
                &
                & 
                & unknown
                & \\ 
    \hline 
      $\gI(-4)$
                &
                & 
                & unknown
                &    \\
    \hline
      $\gI(-1,-1)$
                & $8\pi$ 
                &
                & unknown
                &     \\
    \hline
      $\gI(-2,-2)$
                &
                & 
                & unknown
                & Remark~\ref{rem:i-2-2}  \\
    \hline
      $\gI(-2,-3)$
                &
                & 
                & unknown
                &   \\
    \hline
      $\gI(-2,-2,-2)$
                &
                & 
                & unknown
                & Remark~\ref{rem:i-2-2-2}  \\
    \hline
  \end{tabular}
\caption{Classification of \cmcone{} surfaces in $H^3$ with 
         $\TA(f) \leq 8\pi$}
\label{tab:ta-8pi}
\end{table}
Lemma~\ref{lem:general-class} gives the following corollary: 
\begin{corollary}
\label{cor:8pi-class}
  If $f\colon{}M \to H^3$ is a complete \cmcone{} immersion with 
  $\TA(f) \leq 8\pi$, then it is either 
\begin{enumerate}
 \item a surface of genus $0$ with at most $5$ ends.
       {\rm (}if it has $5$ ends, 
       then all $5$ ends are regular with
       $d_1=d_2=d_3=d_4=d_5=-2${\rm)},
       or 
 \item a surface of genus $1$ with at most $3$ ends 
       {\rm (}if it has $3$ ends, 
       then all $3$ ends are regular with $d_1=d_2=d_3=-2$; if it has 
       $1$ end, 
       then the end is irregular with $d_1=-3$ or $d_1=-4${\rm )}.  
\end{enumerate}
\end{corollary}
\begin{proof}
  We only have to show that a \cmcone{} surface with 
  $\TA(f) \leq 8\pi$ of genus $2$ and with $1$ regular end satisfying
  $0 \leq d_1 \leq 4$ cannot exist.  
  By \eqref{eq:fact-b}, \eqref{eq:fact-d} and \eqref{eq:ta-dual}, such a
  surface would satisfy  $\TA(f^\#) = \TA(f) \leq 8\pi$ and hence the
  hyperbolic Gauss map $G$ is a meromorphic function on a compact
  Riemann surface $\overline M_2$ of genus $2$ with $\deg G \leq 2$.  
  Therefore $\mu_1^\#$ can be only $0$ or $1$, and so $d_1 \leq \mu_1^\#-2 < 
  0$, a contradiction.  
\end{proof}


Now we compile an {\em unfinished\/} classification of 
\cmcone{} surfaces with $\TA(f) \leq 8\pi$ (see Table~\ref{tab:ta-8pi}).  
In the ``status'' column of the table, {\em classified\/} means that 
the surfaces of such a class are completely classified (i.e.\ not only is 
the holomorphic data known, but the period problems are also completely 
solved), {\em existence\/} means that there exists such a surface, and 
{\em unknown\/} means that it is unknown if such a surface exists.  
Surfaces of any type not appearing in the table cannot exist
with $\TA(f)\leq 8\pi$.
The proofs of the existence and non-existence results are given in
Appendix~\ref{app:detailed}.

\appendix
\section{Detailed discussion on the case $\TA(f)\leq 8\pi$}
\label{app:detailed}
In this appendix, we give a precise discussion on complete \cmcone{}
surfaces with $\TA(f)\leq 8\pi$.
\subsection*{Detailed Preliminaries}
Here, we review facts which will be used to prove existence and
non-existence for special cases.
\subsubsection*{Metrics in  $\metone(\C\cup\{\infty\})$}
In this subsection, we introduce special properties  of pseudometrics
in $\metone(\C\cup\{\infty\})$ (see Section~\ref{sec:red}). 
\begin{lemma}\label{lem:hyp-3-red-0}
  A pseudometric $d\sigma^2\in\metone(\C\cup\{\infty\})$ with divisor as
  in \eqref{eq:metone-divisor} is $\Hyp^3$-reducible if and only if all
  orders of conical singularities are integers.
\end{lemma}
\begin{proof}
  If $d\sigma^2$ is $\Hyp^3$-reducible, then the developing map $g$
  is a meromorphic function on $\C\cup\{\infty\}$.
  So the branch orders must all be integers.

  Conversely, assume all conical singularities have integral orders.
  Then by \eqref{eq:metone-g-norm}, $\rho_g(\tau_j)=\pm\id$ for each $j$,
  where $\tau_j$ is the deck transformation on
  $M:=\C\cup\{\infty\}\setminus\{p_1,\dots,p_n\}$ corresponding to the loop
  surrounding $p_j$.
  Since $\pi_1(M)$  is generated by  $\tau_1,\dots,\tau_n$,
  $\rho_g$ is the trivial representation.
\end{proof}
\begin{lemma}\label{lem:hyp-1-red-0}
  Let $d\sigma^2\in\metone(\C\cup\{\infty\})$ with divisor as in 
  \eqref{eq:metone-divisor}.
  Assume the orders $\beta_1$ and $\beta_2$ are not integers,
  and $\beta_j$ {\rm(}$j\geq 3${\rm)} are integers.
  Then $d\sigma^2$ is $\Hyp^1$-reducible.
\end{lemma}
\begin{proof}
  Let $g$ be a developing map such that $\rho_g(\tau_1)$ is diagonal.
  Here, as in the proof of the previous lemma, we have
  $\rho_g(\tau_j)=\pm\id$ ($j\geq 3$).
  Then we have $\rho_g(\tau_1)\rho_g(\tau_2)=\pm\id$ because
  $\tau_1\circ\dots\circ\tau_n=\id$.
  Hence $\rho_g(\tau_2)$ is also a diagonal matrix.
\end{proof}
\begin{lemma}[{\cite[Proposition~A.1]{ruy4}}]
\label{lem:metone-tear-drop}
  There exists no metric $d\sigma^2\in\metone(\C\cup\{\infty\})$ with 
  divisor as in \eqref{eq:metone-divisor} such that 
  only one $\beta_j$ is a non-integer and all others are integers.
\end{lemma}
A developing map of a reducible metric in $\metone(\C\cup\{\infty\})$
can be written explicitly as follows:
\begin{lemma}[{\cite[Proposition~B.1]{ruy4}}]\label{lem:metone-0-normal}
  Let $d\sigma^2\in\metone(\C\cup\{\infty\})$ be reducible with divisor 
  as in \eqref{eq:metone-divisor}.
  Assume 
\[
   p_{n}=\infty\;,\qquad
   \beta_1,\dots,\beta_m\not\in\Z\;,\qquad
   \beta_{m+1},\dots,\beta_{n-1}\in\Z\;.
\]
  Then there exists a developing map $g$ of $d\sigma^2$ such that
\[
    g = (z-p_1)^{\nu_1}\dots(z-p_m)^{\nu_m}\, r(z)
      \qquad (\nu_1,\dots,\nu_m\in\R\setminus\Z)\;,
\]
  where $r(z)$ is a rational function on $\C\cup\{\infty\}$.
\end{lemma}
\begin{corollary}\label{cor:metone-0-normal}
  Let $d\sigma^2\in\metone(\C\cup\{\infty\})$ be reducible with divisor as in 
  \eqref{eq:metone-divisor} and $p_n=\infty$.
  Then there exists a developing map $g$ such that
\begin{equation}\label{eq:devel-can}
    dg = t\,\frac{(z-p_1)^{\alpha_1}\dots(z-p_{n-1})^{\alpha_{n-1}}}{
                 \prod_{k=1}^{N}(z-a_k)^2}\,dz
    \qquad 
    (\alpha_j = \beta_j \quad\text{or}\quad -\beta_j-2)\;,
\end{equation}
  where $a_1,\dots,a_N\in\C\setminus\{p_1,\dots,p_{n-1}\}$ are mutually
  distinct, $t$ is a positive real number.
  Moreover, it holds that
\begin{equation}\label{eq:devel-can-inf}
    -(\alpha_1 + \dots + \alpha_{n-1})+2N -2 = 
         \beta_n \quad \text{or}\quad -\beta_n-2\;.
\end{equation}
\end{corollary}
\begin{proof}
  If $d\sigma^2$ is $\Hyp^3$-reducible, $g$ is a meromorphic 
  function on $\C\cup\{\infty\}$ which branches at $p_1,\dots,p_n$ with
  branch orders $\beta_j\in\Z^+$.
  Hence $p_j$ is a zero of order $\beta_j$ or a pole of order
  $\beta_j+2$ of $dg$ for each $j=1,\dots,n$.
  Let $\{a_1,\dots,a_N\}$ be the simple
  poles of $g$ on $\C\setminus\{p_1,\dots,p_{n-1}\}$, then each $a_k$ is a
  pole of order $2$ of $dg$.  (The $a_j$ are not branch points of $g$.)  
  The zeros and poles of $dg$ are the branch points and the 
  simple poles of $g$.  
  Hence we have \eqref{eq:devel-can} for $t\in\C\setminus\{0\}$.
  By a suitable change  $g\mapsto e^{i\theta}g$ (which is a special form
  of the change \eqref{eq:metone-moebius}), we can choose $g$ such
  that $t\in\R^+$.
  Since $\infty=p_n$ is a zero of order $\beta_n$ or a pole of order
  $\beta_n+2$ of $dg$, we have \eqref{eq:devel-can-inf}.

  Next we assume $d\sigma^2$ is $\Hyp^1$-reducible.
  Without loss of generality, we may assume
  $\beta_1,\dots,\beta_{m}\not\in\Z$ and
  $\beta_{m+1},\dots,\beta_{n-1}\in\Z$.
  Then by Lemma~\ref{lem:metone-0-normal}, we can choose the developing
  map $g$ as
  $g = (z-p_1)^{\nu_1}\dots(z-p_m)^{\nu_m}\, r(z)$,
  where $r(z)$ is a rational function.
  By \eqref{eq:metone-g-norm}, we have
  $\nu_j=\beta_j+1$ or $\nu_j=-\beta_j-1$
  ($j=1,\dots,m$).
  Differentiating this, we have
 \[
    dg = (z-p_1)^{\alpha_1}\dots (z-p_m)^{\alpha_m} r_1(z)\,dz
      \qquad   (\text{$\alpha_j = \beta_j$  or $\alpha_j=-\beta_j-2$})\;
 \]
  where $r_1(z)$ is a rational function.
  Since each $p_j$ ($j=m+1,\dots,n-1$) is a branch point of $g$ of order
  $\beta_j\in\Z$, we have \eqref{eq:devel-can} by an argument similar to the 
  $\Hyp^3$-reducible case.
  Moreover, since $\ord_{\infty}d\sigma^2=\beta_n$, we have 
  \eqref{eq:devel-can-inf}.
\end{proof}
\begin{remark}\label{rem:residue}
  Let $M=\C\cup\{\infty\}\setminus\{p_1,\dots,p_n\}$, and $\widetilde M$
  the universal cover.
  Then there exists a meromorphic function 
  $g\colon{}\widetilde M\to \C\cup\{\infty\}$ satisfying
  \eqref{eq:devel-can} if and only if all of the residues of $dg$ at
  the following points vanish: 
  \begin{enumerate}
   \item $a_k$ ($k=1,\dots,n$), and 
   \item $p_j$ such that $\alpha_j$ is a negative integer.
  \end{enumerate}
\end{remark}
  
\subsubsection*{Construction of \cmcone{} surface from two Gauss maps}
In addition to the $\SU(2)$-conditions for the period problem
(Proposition~\ref{prop:sutwo}), we introduce another method
to construct \cmcone{} surfaces \cite{uy3}:
Let $M$ be a Riemann surface and $\widetilde M$ the universal cover of
$M$.
\begin{proposition}\label{prop:uy3}
 Let $G$ and $g$ be meromorphic functions defined on $M$ and $\widetilde M$,
 respectively.
 Assume
 \begin{enumerate}
  \item  $d\sigma^2:=4\,dg\,d\bar g/(1 + g \bar g)^2$ is a pseudometric
	 with conical singularities which is single-valued on $M$.
  \item  The meromorphic differential $Q:=(S(g)-S(G))/2$ is holomorphic on
	 $M$.
  \item  The metric $ds^2:=(1+|g|^2)^2|Q/dg|^2$ is a non-degenerate
	 complete metric on $M$.
 \end{enumerate}
 then there exists a complete \cmcone{} immersion $f\colon{}M\to H^3$
 with hyperbolic Gauss map $G$ and secondary Gauss map $g$.
\end{proposition}
\begin{proof}
 By the second assumption , $\ord_p Q=\ord_p d\sigma^2$  for any point
 $p\in M$.
 Then, by Theorem~2.2 and Remark~2.3 in \cite{uy3}, there exists
 a \cmcone{} immersion $f\colon{}M\to H^3$ whose hyperbolic Gauss map,
 secondary Gauss map and Hopf differential are $G$, $g$ and $Q$,
 respectively.
 Moreover, by the third assumption, the induced metric is complete.
\end{proof}
\subsection*{Partial classification for \boldmath$\TA(f)\leq 8\pi$}
By Corollary~\ref{cor:8pi-class}, a complete \cmcone{} surface with
$\TA(f)\leq 8\pi$ is either a surface of genus $0$ with at most $5$
ends or a surface of genus $1$ with at most $3$ ends.
We denote by $\gamma$ and $n$ the genus and the number of the ends,
respectively.
\subsubsection*{The case {$(\gamma,n)=(0,1)$}}
  In this case, we may assume $M=\C$ and the only end is $p_1=\infty$.
  Since $M$ is simply-connected, the representation $\rho_g$ as in
  \eqref{eq:g-repr} is trivial, that is, such a surface is
  $\Hyp^3$-reducible.
  Then by Lemma~\ref{lem:hyp-3-red-f-sharp}, the dual immersion $f^{\#}$
  is also well-defined on $M$. 
  And since the dual surface of $f^{\#}$ is $f$ itself, 
  the classification reduces to that for \cmcone{} surfaces 
  with dual absolute total curvature at most $8\pi$, 
  which is done in \cite{ruy3}.  
  
\subsubsection*{The case {$(\gamma,n)=(0,2)$}}
  In this case, the pseudometric $d\sigma^2$ as in \eqref{eq:pseudo} has
  the divisor $\mu_1 p_1 + \mu_2 p_2 + \xi_1 q_1 + \dots + \xi_m q_m$
  (see \eqref{eq:cmcone-divisor}),
  where $p_1$ and $p_2$ are the ends and $q_1,\dots,q_m$ are umbilic
  points.
  Since $\xi_k$ ($k=1,\dots,m$) are integers,
  Lemma~\ref{lem:metone-tear-drop} implies that a surface in this class
  satisfies either 
  \begin{enumerate}
   \item\label{item:g02:1}
        Both $\mu_1$ and $\mu_2$ are integers (the case of
	$\Hyp^3$-reducible),
	or 
   \item\label{item:g02:2}
        Both $\mu_1$ and $\mu_2$ are non-integral real numbers (the
	case of	$\Hyp^1$-reducible).
  \end{enumerate}
  If both ends are regular, such surfaces are completely classified
  (see Example~\ref{exa:catenoid} or \cite{uy1}), 
  and the only possible case is $\gO(-2,-2)$.

  So we may assume at least one end is irregular: $d_2\leq -3$.
  If $d_1\geq -1$, then $\mu_1\in\Z$ by \eqref{eq:fact-d}, and hence
  we have the case \ref{item:g02:1}.
  Hence $g$ is a meromorphic function on the genus $0$ Riemann surface
  $\C\cup\{\infty\}$, and $\TA(f)=4\pi\deg g$.
  Thus $\deg g\leq 2$, and hence $\mu_1$ and $\mu_2$ are $0$ or $1$.
  Then by \eqref{eq:fact-c}, we have $d_1\leq -1$ and $\mu_1=1$.
  Moreover, by \eqref{eq:fact-b}, we have $d_2 \geq -4$.
  On the other hand, if $d_1\leq -2$,
  by \eqref{eq:fact-a}, \eqref{eq:fact-b} and \eqref{eq:fact-g}, we
  have $-7\leq d_1+d_2\leq -4$.
  Hence the possible cases are 
  \begin{align*}
    (d_1,d_2)=&(-1,-3)\;,\quad (-1,-4)\;,\quad
               (-2,-3)\;,\quad (-2,-4)\;, \\
              &(-2,-5)\;, \quad (-3,-3)
              \quad\text{and}\quad (-3,-4)\;.
  \end{align*}
  Throughout this subsection, we set $M=\C\setminus\{0\}$.
\begin{proposition}\label{prop:o-2-3}
  There exists no complete 
  \cmcone{} immersion $f\colon{}\C\setminus\{0\}\to H^3$
  with $\TA(f)\leq 8\pi$ and of class $\gO(-1,-3)$ or\/ $\gO(-2,-3)$.
\end{proposition}
\begin{proof}
  Assume $f$ is of class $\gO(-1,-3)$.
  In this case, $\mu_1\in\Z$ by \eqref{eq:fact-d}, and then $f$ is 
  $\Hyp^3$-reducible (the case of \ref{item:g02:1}).
  Then the dual immersion $f^{\#}$ is also well-defined on  $M$
  whose dual absolute total curvature is not greater than $8\pi$.
  Such a surface cannot exist because of the results in \cite{ruy3} (see 
  Table~\ref{tab:class}).

  Now suppose $f$ is of class $\gO(-2,-3)$.
  If $\mu_1\in\Z$, then for the same reason as in the $\gO(-1,-3)$
  case, such a surface does not exist.
  Now assume $\mu_1\not\in\Z$.
  Then the surface is of type \ref{item:g02:2}: $\mu_2\not\in\Z$.
  By the same argument as in the case $d_1+d_2=-5$ of 
  $(\gamma,n)=(0,2)$ in the proof in \cite{ruy4} of
  Theorem~\ref{thm:four-pi} (in this paper), such a surface cannot exist.
\end{proof}
  By similar argument as in the proof of Proposition~\ref{prop:o-2-3},
  if $\mu_1\in\Z$, the classification is the same as the dual case in
  \cite{ruy3}.
  Hence the case $\gO(-1,-4)$, and also the case $\gO(-2,-4)$ with 
  $\mu_1 \in \Z$ ($\Hyp^3$-reducible), are classified.  
  Furthermore, for the same reason, the $\gO(-2,-5)$ case with $\mu_1
  \in \Z$ and $\TA(f) \leq 8 \pi$ does not exist.  
\begin{remark}\label{rem:o-2-4}
 In the case $\gO(-2,-4)$ with \ref{item:g02:2} holding,
 we have the following examples: 
 Let 
\[
    dg = t\,z^{\mu}\frac{z^2-a^2}{(z^2-1)^2}\,dz\;, \qquad
    Q  = \theta \frac{z^2-a^2}{z^2}dz^2\;,
\]
 where
\[
    a^2 = \frac{\mu+1}{\mu-1}\;, \qquad 
    \theta = \frac{\mu(\mu+2)(\mu-1)}{4(\mu+1)}\;,
    \qquad -1<\mu<0\;.
\]
 Here $t$ is a positive real number corresponding to the one parameter 
 deformation coming from reducibility (see Section~\ref{sec:red}).
 Then the residues of $dg$ at $-1$ and $1$ vanish, and there exists the
 meromorphic function $g$ defined on the universal cover of 
 $\C\setminus\{0\}$ (see Remark~\ref{rem:residue}).
 We set $\omega=Q/dg$.
 Then by Theorem~2.4 in \cite{uy1}, one can check that 
 there exists an immersion  $f\colon{}\C\setminus\{0\}\to H^3$ with
 data $(g,Q)$.
 For this example, 
 $\mu_1 = \mu_2 = |\mu+1|-1 = \mu$
 because $-1<\mu<0$ (see Corollary~\ref{cor:metone-0-normal}).
 Then, 
 $\TA(f)/2\pi = 2(\mu+2)\in (4\pi,8\pi)$.
\end{remark}
\begin{remark}\label{rem:o-2-5}
  For the $\gO(-2,-5)$ case, the following data gives examples:  
  We set 
  \[
     dg = t\,z^{\mu}\frac{z^3-a^3}{(z^3-1)^2}\,dz\;, \quad
      Q = \theta \frac{z^3-a^3}{z^2}\,dz^2\quad
     \left(a^3 = \frac{\mu+1}{\mu-2}\;, ~
             \theta = \frac{\mu(\mu^2-4)}{4(\mu+1)}\right),
  \]
 where $\mu\in \R\setminus\{0,-1,\pm 2\}$, $t\in\R^+$.
 Here $t$ is a parameter corresponding to a deformation 
 which comes from reducibility (see Section~\ref{sec:red}).
 The ends are $0,\infty$ and the umbilic points are
 $a, a e^{2\pi/3 i}, a e^{4\pi/3 i}$.

 In this case, we have
 $\mu_1 = |\mu+1|-1$ and $\mu_2 = |\mu|-1$.
 Hence
\[
    \mu_1 + \mu_2 = |\mu+1|+|\mu|-2 \geq -1\;,
\]
 where equality holds if and only if $-1\leq  \mu\leq 0$.
 Thus the total absolute curvature is
 $\TA(f) = 2\pi (-2 + \mu_1 + \mu_2 -d_1-d_2 ) \geq 8\pi$
 and equality holds if and only if $-1<\mu<0$.
\end{remark}
\begin{remark}\label{rem:o-3-3}
  For the cases $\gO(-3,-3)$ and $\gO(-3,-4)$, all ends are irregular, 
  and then one cannot solve the period problem immediately.
  In the dual total curvature case, a deformation procedure as in 
  \cite{ruy1} can be used to construct examples of type $\gO(-3,-3)$ \cite{ruy3}.
  Unfortunately, this procedure cannot be used here, because
  the hyperbolic Gauss map is not a rational function.  
\end{remark}

\begin{remark}\label{rem:o-3-4}
 In the cases of $\gO(-3,-4)$ and $\gO(-2,-5)$, it can be shown that 
 $\TA(f)\geq 8\pi$.  
 In fact, in these cases, the divisor corresponding the pseudometric 
 $d\sigma^2$ is $\mu_1 p_1 + \mu_2 p_2 + \xi_1 q_1+\dots+\xi_mq_m$,
 and by \eqref{eq:fact-a}, we have $\xi_1+\dots+\xi_m = 3$ is an odd integer.
 Then by Corollary~4.7 of \cite{ruy4}, we have $\mu_1+\mu_2\geq -1$. 
 This shows that $\TA(f)\geq 8\pi$.
\end{remark}
\subsubsection*{The case $(\gamma,n)=(0,3)$}
 If $\mu_1$, $\mu_2$ and $\mu_3$ are integers, then 
 by Lemma~\ref{lem:hyp-3-red-0} and Lemma~\ref{lem:hyp-3-red-f-sharp},
 the surface is $\Hyp^3$-reducible  and its dual is also well-defined on
 $M$ with dual total absolute curvature at most $8 \pi$. 
 By \cite{ruy3}, such surfaces must be of type $\gO(-1,-1,-2)$, 
 $\gO(-1,-2,-2)$, or $\gO(-2,-2,-2)$, and the first two cases are classified. 
 Also, examples exist in the third case as well \cite[Example~4.4]{ruy3}.
 Moreover, for any surface of type $\gO(-1,-1,-2)$, $\mu_1$ and $\mu_2$
 are integers, by \eqref{eq:fact-d}.
 Then, by Lemma~\ref{lem:metone-tear-drop}, $\mu_3$ is also an integer.
 Thus, surfaces of type $\gO(-1,-1,-2)$ must be $\Hyp^3$-reducible and are 
 completely classified.

 Next, we assume all $\mu_j \not\in \Z$.  
 Then \eqref{eq:fact-a}, \eqref{eq:fact-b} and \eqref{eq:fact-g}
 imply that $-8\leq d_1+d_2+d_3\leq -4$, and  \eqref{eq:fact-d} implies
 that $d_j\leq -2$ ($j=1,2,3$).
 Hence the possible cases are $\gO(-2,-2,-2)$, $\gO(-2,-2,-3)$,
 $\gO(-2,-2,-4)$ and $\gO(-2,-3,-3)$.
 For the case $\gO(-2,-2,-2)$, that is, for surfaces with three regular ends,
 the second and third authors classified the irreducible ones with
 embedded ends (\cite{uy7}, see Example~\ref{exa:trinoid}).

 For the cases $\gO(-2,-3,-3)$ and $\gO(-2,-2,-4)$, the sum of the 
 orders of the umbilic points are an even integer, by \eqref{eq:fact-a}. 
 Then by Corollary~4.7 in \cite{ruy4} and \eqref{eq:fact-b}, we have 
 $\TA(f) \geq 8\pi$, hence $\TA(f) = 8 \pi$.  
 
 By Lemma~\ref{lem:metone-tear-drop},
 there exists no surface with only one non-integer $\mu_j$.  
 Then the remaining case is to  assume that one $\mu_j$, say $\mu_1$, is
 an integer and $\mu_2, \mu_3 \not\in \Z$.  
 Then by \eqref{eq:fact-d}, $d_2$, $d_3\leq -2$.
 Also, by \eqref{eq:fact-b}, \eqref{eq:fact-c} and
 \eqref{eq:fact-g}, we have $-5\leq d_2+d_3$.  
 Hence we have two possibilities: $(d_2,d_3)=(-2,-2)$ or
 $(d_2,d_3)=(-2,-3)$.

 When $(d_2,d_3)=(-2,-2)$, by \eqref{eq:fact-b} and \eqref{eq:fact-c}, we 
 have $\mu_1-d_1=2$ or $3$.
 And by \eqref{eq:fact-a}, $d_1\leq 0$.
 Hence we have the possibilities
 $\gO(-3,-2,-2)$, $\gO(-2,-2,-2)$, $\gO(-1,-2,-2)$ and $\gO(0,-2,-2)$.

 Similarly, when $(d_2,d_3)=(-2,-3)$, we have $\mu_1-d_1=2$ and $d_1\leq
 1$.
 Hence the possibilities are $\gO(-2,-2,-3)$, $\gO(-1,-2,-3)$,
 $\gO(0,-2,-3)$, $\gO(1,-2,-3)$.
 In this case, the corresponding divisor of the pseudometric $d\sigma^2$
 is 
\[
      \mu_1 p_1 + \mu_2 p_2 + \mu_3 p_3
          + \sum_{k=1}^m \xi_k q_k
         = \mu_2 p_2 + \mu_3 p_3 + (2+d_1) p_1 + \sum_{k=1}^m \xi_k q_k\;,
\]
 where the $q_k$ ($k=1,\dots,m$) are the umbilic points and $\xi_k$ is the 
 order of $Q$ at $q_k$ (see \eqref{eq:cmcone-divisor}).
 Here, by \eqref{eq:fact-a}, $\xi_1+\dots+\xi_m=1-d_1$, so 
\[
     \mu_1 + \sum_{k=1}^m \xi_k = d_1+2+\sum_{k=1}^m \xi_k=3 \; . 
\]
 Hence if $d_1 \geq -1$ (and so $\mu_1 \in \Z^+$), Corollary 4.7 of 
 \cite{ruy4} implies $\mu_2+\mu_3\geq -1$.
 This implies that $\TA(f)\geq 8\pi$, and so
\begin{equation}\label{eq:g0n3-8pi}
   \text{For a surface of type $\gO(d,-2,-3)$ ($d\geq -1$),}\qquad
   \TA(f)\geq 8\pi\;.
\end{equation}

\begin{proposition}\label{prop:o0-2-3}
 There exists no complete \cmcone{} surface $f$ with $\TA(f)\leq 8\pi$
 and of type $\gO(0,-2,-3)$.
\end{proposition}
\begin{proof}
 Assume such an immersion
 $f\colon{}\C\cup\{\infty\}\setminus\{p_1,p_2,p_3\}\to H^3$ exists.
 By \eqref{eq:fact-a}, there is the only umbilic point $q_1$.
 We set the ends $(p_1,p_2,p_3)=(0,1,\infty)$ and the umbilic point
 $q_1=q\in\C\setminus\{0,1\}$.
 Then the Hopf differential $Q$ has a zero only at $q$ with order $1$, 
 and two poles at $1$ and $\infty$ with orders $2$ and $3$, respectively.
 Thus, $Q$ is written as
\[
       Q  := \theta\frac{z-q}{(z-1)^2}\,dz^2
       \qquad (\theta\in\C\setminus\{0\})\;.
\]
 
 On the other hand, by \eqref{eq:fact-b}, \eqref{eq:fact-g},
 \eqref{eq:fact-c} and \eqref{eq:fact-d}, we have $\mu_1=2$, 
 and
\begin{equation}\label{eq:o0-2-3-murange}
    -2 < \mu_2 + \mu_3 \leq -1\qquad\text{and}\qquad
      -1<\mu_j<0 \qquad (j=2,3)\;.
\end{equation}
 The secondary Gauss map branches at $(p_1,p_2,p_3)$ and $q$ with
 branch orders $2$, $\mu_2$, $\mu_3$ and $1$, respectively.
 Then by Corollary~\ref{cor:metone-0-normal}, we can take the secondary
 Gauss map $g$ such that
\[
   dg = t\,\frac{z^{\alpha}(z-1)^{\nu}(z-q)^{\beta}}{
              \prod_{j=1}^N(z-a_j)^2}\,dz
        \qquad (t\in\R\setminus\{0\})\;,
\]
 where
\[
    \nu=\mu_2\quad\text{or}\quad -\mu_2-2\;,\qquad
    \alpha=2\quad\text{or}\quad -4\;,\qquad
    \beta=1\quad\text{or}\quad -3\;,
\]
 and $a_j\in\C\setminus\{0,1,q\}$ ($j=1,\dots,N$) are mutually distinct
 points.
 
 Without loss of generality, we may assume $\nu=\mu_2$ (if not, we can
 take $1/g$ instead of $g$).
 Then by \eqref{eq:devel-can-inf} in Corollary~\ref{cor:metone-0-normal}, we
 have
\[
    -(\alpha+\beta)-\mu_2+2N-2= \mu_3 
      \quad\text{or}\quad -\mu_3-2\;,
\]
 so $\mu_2+\mu_3$ or $\mu_2-\mu_3$ is an odd integer.
 Then by \eqref{eq:o0-2-3-murange}, we have $\mu_2+\mu_3=-1$ and
 $(\alpha,\beta,N)=(2,1,2)$ or $(2,-3,0)$.
 
 First, we assume $(\alpha,\beta,N)=(2,1,2)$, and we set $\mu_2=\mu$.
 Then 
\begin{equation}\label{eq:o0-2-3-dg-a}
    dg= t\, \frac{z^2 (z-1)^{\mu}(z-q)}{(z-a)^2(z-b)^2}\,dz
    \qquad (a,b\in\C\setminus\{0,1,q\}, a\neq b)\;.
\end{equation}
 Such a $g$ exists if and only if the residues of the right-hand side 
 of \eqref{eq:o0-2-3-dg-a} vanish:
\begin{align}
\label{eq:o0-2-3-residue-a-1}
  \frac{2}{a}+\frac{\mu}{a-1}+\frac{1}{a-q}-\frac{2}{a-b}&=0 \;,\\
\label{eq:o0-2-3-residue-a-2}
  \frac{2}{b}+\frac{\mu}{b-1}+\frac{1}{b-q}-\frac{2}{b-a}&=0 \;.
\end{align}
 Since $a\neq b$, these equations are equivalent to 
 $\mbox{\eqref{eq:o0-2-3-residue-a-1}}- \mbox{\eqref{eq:o0-2-3-residue-a-2}}$
 and 
 $b\times\mbox{\eqref{eq:o0-2-3-residue-a-1}}- 
       a\times\mbox{\eqref{eq:o0-2-3-residue-a-2}}$:
\begin{equation}\label{eq:o0-2-3-residue-a}
  \begin{aligned}
  (\mu+1)(a^2+b^2)-(\mu q +1)(a+b) -2 a b +2 q &=0 \;,\\
   (\mu+1)(a+b)ab - 2 (\mu q +  q +2) ab + 2 q (a+b)&=0\;.
  \end{aligned}
\end{equation}
 On the other hand, let $\omega=Q/dg$, and consider the equation
\begin{equation}\label{eq:o0-2-3-E1}
      X'' -\bigl(\log\hat\omega\bigr)'X'-\hat Q X=0\qquad
      \bigl(\omega=\hat\omega\,dz, \quad Q = \hat Q\, dz^2\bigr)\;,
\end{equation}
 which is named (E.1) in \cite{uy1}.
 The roots of the indicial equation of \eqref{eq:o0-2-3-E1} at $z=0$
 are $0$ and $-1$.
 By Theorem 2.2 of \cite{uy1}, the log-term coefficient of 
 \eqref{eq:o0-2-3-E1} at $z=0$ must vanish if the surface exists:
\begin{equation}\label{eq:o0-2-3-E1-log}
     \mu+2-\frac{2}{a}-\frac{2}{b}=0\;.
\end{equation}
 (See Appendix A of \cite{ruy3} or Appendix A of \cite{uy1}).
 Here, the solution of equations \eqref{eq:o0-2-3-residue-a} 
 and  \eqref{eq:o0-2-3-E1-log} is 
 $a = b = q = 4/(\mu+2)$,  a contradiction.
 Hence the case $(\alpha,\beta,N)=(2,1,2)$ is impossible.

 Next, we consider the case $(\alpha,\beta,N)=(2,-3,0)$.
 Then it holds that 
\[
  dg = t\, \frac{z^2(z-1)^{\mu}}{(z-q)^3}\,dz\qquad 
       (t\in\C\setminus\{0\})\;.
\]
 The residue at $z=q$ vanishes if and only if 
\begin{equation}\label{eq:o0-2-3-residue-b}
  (\mu+2)(\mu+1)q^2 - 4 (\mu+1)q +2 =0\;.
\end{equation}
 On the other hand, in the same way as the first case, 
 the log-term coefficient of \eqref{eq:o0-2-3-E1} at $z=0$ vanishes if
 and only if 
\begin{equation}\label{eq:o0-2-3-log-b}
    \mu+2 = \frac{4}{q}\;.
\end{equation}
 However, there is no pair $(\mu,q)$ satisfying
 \eqref{eq:o0-2-3-residue-b} and \eqref{eq:o0-2-3-log-b} simultaneously.
 Hence this case is also impossible.
\end{proof}
\begin{proposition}\label{prop:o1-2-3}
 There exists no \cmcone{} immersion of type $\gO(1,-2,-3)$ with
 $\TA(f)\leq 8\pi$.
\end{proposition}

\begin{proof}
 Assume such an immersion
 $f\colon{}\C\cup\{\infty\}\setminus\{p_1,p_2,p_3\}\to H^3$ exists.
 Then we have $\TA(f)=8\pi$ because of \eqref{eq:g0n3-8pi}, 
 and by \eqref{eq:fact-b}, \eqref{eq:fact-c} and 
 \eqref{eq:fact-g}, it holds that
\begin{equation}\label{eq:go1-2-3-mu-rel}
      \mu_1=3 \qquad \text{and} \qquad
      \mu_2+\mu_3=-1,\qquad -1<\mu_j<0\quad (j=2,3) \;.
\end{equation}

 We set $(p_1,p_2,p_3)=(1,0,\infty)$.
 By \eqref{eq:fact-a}, there are no umbilic points.
 Then the Hopf differential $Q$ is written as
\[
   Q = \theta\,\frac{z-1}{z^2}\,dz^2 \qquad (\theta\in\C\setminus\{0\})\;.
\]
 The secondary Gauss map $g$ branches at $0$, $\infty$ and 
 $1$ with orders $\mu_2$, $\mu_3$ and $3$, respectively.
 Then by Corollary~\ref{cor:metone-0-normal},
 $dg$ can be put in the following form:
\[
    dg = t\, z^{\mu_2}\frac{(z-1)^{\alpha}}{\prod_{j=1}^N(z-a_j)^2}\,dz \; ,
\]
 where $a_j\in \C\setminus\{0,1\}$ $(j=1,\dots,N)$ are mutually distinct
 numbers, $t$ is a positive real number, and $\alpha=3$ or $-5$, and
\[
    -\mu_2 -\alpha +2N-2 = \mu_3\quad \text{or}\quad -\mu_3-2\;.
\]
 The second case is impossible because of \eqref{eq:go1-2-3-mu-rel}.
 Hence $2N = \alpha+1$, and then 
 $\alpha=3$ and $N=2$.  Thus we have the form
\[
     dg = t\, z^{\mu}\frac{(z-1)^3}{(z-a)^2(z-b)^2}\,dz
      \qquad (\mu=\mu_2)\;.
\]
 Such a $g$ exists if the residues at $z=a$ and $z=b$ vanish:
\[
     \frac{\mu}{a}+\frac{3}{a-1}-\frac{2}{a-b}=0 \; \quad
         \text{and}\quad
     \frac{\mu}{b}+\frac{3}{b-1}-\frac{2}{b-a}=0 \; .
\]
 By direct calculation, we have
\begin{align*}
   a &=
 \frac{-2+\mu+\mu^2+\sqrt{2}\sqrt{2-\mu-\mu^2}}{(\mu+1)(\mu+2)} \; ,\\
   b &=
 \frac{-2+\mu+\mu^2-\sqrt{2}\sqrt{2-\mu-\mu^2}}{(\mu+1)(\mu+2)} \; .
\end{align*}
  Consider the equation (E.1) in \cite{uy1} with $\omega=Q/dg$.
  Then, the indicial equation at $z=1$ has the two roots $0$ and $-1$
  with difference $1$.
  By direct calculation again, the log-term at $z=1$ vanishes
  if and only if
\[
     \mu + 2 -\frac{2}{1-a}-\frac{2}{1-b}
        = -\frac{1}{3}(\mu+2)=0 \; , 
\]
  which is impossible because $\mu=\mu_2>-1$.
\end{proof}
\begin{proposition}\label{prop:o0-2-2}
 For a non-zero real number $\mu$ $(-1<\mu<0)$ and
 positive integer $m$, set
\begin{equation}\label{eq:data-o0-2-2}
    G= z^{m+1} \frac{mz-(m+2)}{(m+2)z-m}\qquad
    \text{and}
    \qquad
    g = z^{\mu+1} \frac{\mu z-(\mu+2)}{(\mu+2)z-\mu}\;.
\end{equation}
 Then there exists a one parameter family of conformal \cmcone{}
 immersions $f\colon{}\C\setminus\{0,1\}\to H^3$ of type $\gO(0,-2,-2)$
 with $\TA(f)=4\pi(\mu+2)$,
 whose hyperbolic Gauss map and secondary Gauss map are $G$ and $tg$
 $(t\in\R^+)$, respectively.
 
 Conversely, any \cmcone{} surface of type $\gO(0,-2,-2)$ with
 $\TA(f)\leq 8\pi$ is obtained in such a manner.
 In particular, $\TA(f)< 8\pi$.
\end{proposition}
\begin{proof}
 For $g$ and $G$ as in \eqref{eq:data-o0-2-2}, set
\[
   Q:=\frac{1}{2}\left(S(g)-S(G)\right)
     =\frac{m(m+2)-\mu(\mu+2)}{4}\frac{dz^2}{z^2}\;.
\]
 Since $\mu\not\in\Z$, the right-hand side is not identically zero.
 Moreover, one can easily check that the assumptions of
 Proposition~\ref{prop:uy3} hold.
 Hence there exists a complete \cmcone{} immersion
 $f\colon{}\C\setminus\{0,1\}\to H^3$
 with hyperbolic Gauss map $G$, secondary Gauss map $g$ and Hopf 
 differential $Q$.

 Conversely, suppose such a surface exists.
 Then without loss of generality, we set $M=\C\setminus\{0,1\}$ and 
 $(p_1,p_2,p_3)=(1,0,\infty)$.
 By \eqref{eq:fact-a}, there are no umbilic points.
 Then the Hopf differential $Q$ has poles of order $2$ at $0$ and
 $\infty$, and has no zeros. Hence we have
 $Q = \theta\,z^{-2}\,dz^2$ ($\theta\in\C\setminus\{0\})$.
 This implies that  the secondary and hyperbolic Gauss maps branch 
 only at the ends.
 Hence both $S(g)$ and $S(G)$ have poles of order $2$ at $z=0,1,\infty$
 and holomorphic on $M$.
 More precisely, we have
 \begin{equation}\label{eq:schwarz-0-2-2}
  \begin{aligned}
    S(g)&= \frac{c_3 z^2+(c_1-c_2-c_3)z+c_2}{z^2(z-1)^2}\,dz^2,\\
    S(G)&= \frac{c_3^{\#} z^2+(c_1^{\#}-c_2^{\#}-c_3^{\#})z + c_2^{\#}}
          {z^2(z-1)^2}\,dz^2\;,
  \end{aligned}
 \end{equation}
 where $c_j$ and $c_j^{\#}$ are as in \eqref{eq:c-j}.
 Here, $\mu_1 = \mu_1^{\#}$ because of \eqref{eq:fact-d}.
 Hence we have
  \begin{align*}
       2 Q &= S(g)-S(G)\\
           &= \left(
               \frac{c_2-c_2^{\#}}{z^2}
             + \frac{(c_2-c_3)-(c_2^{\#}-c_3^{\#})}{z}
             - \frac{(c_2-c_3)-(c_2^{\#}-c_3^{\#})}{z-1}\right)\,dz^2 \; .
  \end{align*}
 Thus we have
  \begin{equation}\label{eq:claim-one-0-2-2}
      (\mu_2-\mu_3)(\mu_2+\mu_3+2) =
      (\mu_2^\#-\mu_3^\#)(\mu_2^\#+\mu_3^\#+2)\;.
  \end{equation}

 On the other hand, \eqref{eq:fact-b}, \eqref{eq:fact-c} and 
 \eqref{eq:fact-g} imply that $\mu_1=\mu_1^{\#}=2$ or $3$.  
 If $\mu_1=3$, \eqref{eq:fact-b} implies that $\mu_2+\mu_3\leq -1$.
 Then by \eqref{eq:fact-g}, $-1<\mu_j<0$ for $j=2,3$.
 Hence using \eqref{eq:claim-one-0-2-2}, we have
\[
    1 > |\mu_2-\mu_3| > |\mu_2^\#-\mu_3^\#|\;.
\]
 Here $\mu_2^{\#}$ and $\mu_3^{\#}$ are integers, 
 hence $\mu_2^{\#}=\mu_3^{\#}$.  
 The hyperbolic Gauss map $G$ is a meromorphic function on $\C\cup\{\infty\}$.
 Then the Riemann-Hurwitz relation implies that
\[
    \Z\ni\deg G = \frac{1}{2}(2 + \mu_1^{\#}+\mu_2^{\#}+\mu_3^{\#})
                = \mu_2^{\#}+2 +\frac{1}{2}\;.
\]
 This is impossible.

 When $\mu_1=\mu_1^{\#}=2$, by similar arguments, we have
 $-1<\mu_j<1$ $(j=2,3)$ and $\mu_2+\mu_3\leq 0$.
 This implies that $|\mu_2-\mu_3|<2$.
 Thus by \eqref{eq:claim-one-0-2-2}, we have
\[
    |\mu_2^{\#}-\mu_3^{\#}|= 0 \quad\text{or}\quad 1\;.
\]
  We may assume that 
  $\mu_3^{\#}\geq \mu_2^{\#}$ (if not, exchange the ends $0$ and
  $\infty$).  Assume $\mu_3^{\#}-\mu_2^{\#}=1$.  In this case, 
\[
   \Z\ni\deg G = \frac{1}{2}(2 + \mu_1^{\#}+\mu_2^{\#}+\mu_3^{\#}) 
            =\mu_2^{\#}+2 +\frac{1}{2} \; , 
\]
  which is impossible.
  Hence, using also \eqref{eq:claim-one-0-2-2}, we have 
  $\mu_3^{\#}-\mu_2^{\#}=\mu_3-\mu_2=0$.
  Moreover, $\mu_2+\mu_3=2\mu_2\leq 0$, so $\mu_2\leq 0$.  Putting all this 
  together, we have
\[
    \mu_1=\mu_1^{\#}=2 \; ,\quad
    -1<\mu_2=\mu_3 \leq 0 \; ,\quad
    \mu_2^{\#}=\mu_3^{\#} \; , 
    \quad\text{and} \quad 
    Q = \frac{c_2-c_2^{\#}}{2 z^2}\,dz^2\; .  
\]
 If $\mu_2=0$, the secondary Gauss map $g$ is a meromorphic function on
 $\C\cup\{\infty\}$ with only one branch point, which is impossible.
 Hence $\mu_2<0$.
 In this case, the pseudometric $d\sigma^2$ branches on the divisor
 $\mu_1 p_1 + \mu_2 p_2 + \mu_3 p_3$
 because there are no umbilic points.
 Thus the secondary Gauss map $g$ satisfies \eqref{eq:schwarz-0-2-2}.
 One possibility of such a $g$ is in the form \eqref{eq:data-o0-2-2}
 with $\mu=\mu_2$.
 On the other hand, since the surface is $\Hyp^1$-reducible, 
 $g$ can be normalized as in \eqref{eq:data-o0-2-2} because 
 of Corollary~\ref{cor:metone-0-normal}.
 Since $S(G)=S(g)-2Q$, the Schwarzian derivative of the hyperbolic 
 Gauss map $G$ is uniquely determined, and $G$ is determined 
 up to M\"obius transformations.
 Then such a surface is unique, with given $g$ and $G$.
\end{proof}

\begin{proposition}\label{prop:o-1-2-2}
 Let $\mu\in(-1,0)$, $m\geq 2$ an integer, 
\[
    a:=-\frac{m+\mu+2}{m-\mu-2},\qquad
    p:=\frac{a\mu+a-a^2}{a\mu+a-1}\;,
\]
 and $M=\C\cup\{\infty\}\setminus\{0,1,p\}$.
 Then there exist a meromorphic function $G$ on $\C\cup\{\infty\}$
 and a meromorphic function $g$ on the universal cover $\widetilde M$
 of $M$ such that
\begin{equation}\label{eq:go-1-2-2-gauss}
    dG = z \frac{(z-p)^{m-2}}{(z-1)^{m+2}}\,dz\qquad\text{and}\qquad
    dg = t\, z \frac{(z-1)^{\mu}(z-p)^{-\mu-2}}{(z-a)^2}\,dz
\end{equation}
 respectively, where $t\in\R^+$, 
 and there exists a complete \cmcone{} immersion $f\colon{}M\to H^3$
 whose hyperbolic Gauss map and secondary Gauss map are $G$ and $g$, 
 respectively.
 Moreover $\TA(f)=4\pi(\mu+2)\in(4\pi,8\pi)$.

 Conversely, an $\Hyp^1$-reducible complete \cmcone{} surface of class
 $\gO(-1,-2,-2)$ with $\TA(f)< 8\pi$ is obtained in such a way.
\end{proposition}

\begin{proof}
 The residue of $dG$ in \eqref{eq:go-1-2-2-gauss}
 at $z=1$ and the residue of $dg$ in \eqref{eq:go-1-2-2-gauss} at $z=a$
 vanish.
 Thus there exist $G$ and $g$ such that \eqref{eq:go-1-2-2-gauss}
 hold.  Moreover, by direct calculation, we have
\[
    Q:=\frac{1}{2}\bigl(S(g)-S(G)\bigr)
      = \frac{4m^2\bigl(m(m+2)-\mu(\mu+2)\bigr)}{
              (m+\mu)^2(2-m+\mu)^2}
        \frac{dz^2}{z(z-1)^2(z-p)^2}\;.
\]
 Then by Proposition~\ref{prop:uy3}, there exists a \cmcone{} immersion
 $f\colon{}M\to H^3$.
 One can easily check that $f$ is complete and $\TA(f)=4\pi(\mu+2)$.

 Conversely, assume such a surface exists.
 Then by \eqref{eq:fact-a}, there is only one umbilic point $q_1$ of
 order one.
 We set $(p_1,p_2,p_3)=(0,1,p)$ and $q_1=\infty$, where
 $p\in\C\setminus\{0,1\}$.

 By \eqref{eq:fact-b}, \eqref{eq:fact-c}, \eqref{eq:fact-d}, \eqref{eq:fact-f} 
 and \eqref{eq:fact-g}, we have $\mu_1=1$ or $2$.  
 When $\mu_1=2$, by Corollary 4.7 of \cite{ruy4}, $\mu_2+\mu_3\geq -1$ 
 holds, and then $\TA(f)\geq 8\pi$.

 Assume $\mu_1=1$ and $\TA(f)\leq 8\pi$. 
 If one of $\mu_j$ $(j=2,3)$ is an integer, by
 Lemma~\ref{lem:metone-tear-drop}, the other is also an integer, and
 hence the surface is $\Hyp^3$-reducible.
 Thus both $\mu_2$ and $\mu_3$ are non-integers.
 By \eqref{eq:fact-b} and \eqref{eq:fact-g}, we have
\begin{equation}\label{eq:mu-range-o-1-2-2}
     -2 < \mu_2 + \mu_3 \leq 0 \; , \qquad 
      -1 < \mu_j < 1 \quad (j=2,3)\;.
\end{equation}
 Then the secondary Gauss map $g$ branches at $0$, $1$, $p$ and $\infty$
 with orders $\mu_1$, $\mu_2$, $\mu_3$ and $1$, respectively.
 Then by Corollary~\ref{cor:metone-0-normal}, $g$ can be chosen in the form
\[
    dg = t\, \frac{(z-1)^{\nu_2} (z-p)^{\nu_3} z^{\alpha}}{
                  \prod_{k=1}^N (z-a_k)^2}\,dz\qquad 
        (t\in\R^+)\;,
\]
 where $\nu_j = \mu_j$ or $ -\mu_j-2$, and 
 $\alpha = 1$ or $-3$,
 and $\{a_1,\dots,a_N\}\subset \C\setminus\{0,1,p\}$ are mutually
 distinct points.  
 We may assume $\nu_2=\mu_2$ (if not, take $1/g$ instead of $g$).
 Then by \eqref{eq:devel-can-inf}, 
\[
        -\mu_2 -\nu_3 -\alpha+2N-2= 1 \quad\text{or}\quad -3\;
\]
 holds.
 This implies that $\mu_2+\mu_3$ or $\mu_2-\mu_3$ is an even integer.
 Then by \eqref{eq:mu-range-o-1-2-2}, we have
\[
     \mu:=\mu_2 = \mu_3\in(-1,0),\qquad \nu_3 = -\mu-2,\qquad
     \alpha=1,\quad \text{and}\quad N=1\;.
\]
 Hence we have
\begin{equation}\label{eq:dg-o-1-2-2}
      dg = t\, \frac{z(z-1)^\mu (z-p)^{-\mu-2}}{(z-a)^2}\,dz\qquad
       (t\in\R^+, \quad a\in \C\setminus\{0,1,p\})\;.
\end{equation}
 Such a map $g$ exists on the universal cover of $\C\setminus\{0,1,p\}$ if 
 and only if the residue at $z=a$ of the right-hand
 side of \eqref{eq:dg-o-1-2-2} vanishes, that is, if 
 $p = (a\mu + a - a^2)/(a\mu+a-1)$.
 The Hopf differential of such a surface is written in the form
\begin{equation}\label{eq:hopf-o-1-1-2}
   Q = \frac{\theta\,dz^2}{z(z-1)^2(z-p)^2}\qquad(\theta\in\C\setminus\{0\})
\end{equation}
 because it has poles of order $2$ at $z=1$ and $p$, a pole of order $1$
 at $z=0$ and a zero of order $1$ at $z=\infty$.
 Let $\mu_1^{\#}$, $\mu_2^{\#}$ and $\mu_3^{\#}$ be the branch orders
 of the hyperbolic Gauss map at $p_1$, $p_2$ and $p_3$, respectively.
 Then by \eqref{eq:q-first}, we have
\begin{equation}\label{eq:c-o-1-2-2}
    c_2-c_2^{\#}=\frac{2\theta}{(1-p)^2},\qquad
    c_3-c_3^{\#}=\frac{2\theta}{(1-p)^2p}\;,
\end{equation}
 where $c_j$ and $c_j^{\#}$ are as in \eqref{eq:c-j}.
 Then $p$ and $\theta$ are positive real numbers.
 Without loss of generality, we may assume $\mu_2^{\#}\geq \mu_3^{\#}$.
 Then we have
\begin{equation}\label{eq:c2-c3-o-1-2-2}
    0 \geq c_2^{\#}-c_3^{\#}=\frac{2\theta}{p(1-p)}\;.
\end{equation}
 Hence $\mu_2^{\#}\neq \mu_3^{\#}$, that is, $\mu_2^{\#}>\mu_3^{\#}$.
 Since $\mu_1^{\#}=1$ by \eqref{eq:fact-d}, the hyperbolic Gauss map
 branches at $0$, $1$, $p$ and $\infty$ with branching order $1$,
 $\mu_2^{\#}$, $\mu_3^{\#}$ and $1$, respectively.
 Then the Riemann-Hurwitz relation implies that
 $\deg G = 2 + (\mu_2^{\#}+\mu_3^{\#})/2<\mu_2^{\#}+2$.
 On the other hand, we have $\deg G\geq\mu_2^{\#}+1$.
 Hence we have 
\[
  \deg G = \mu_2^{\#}+1 \qquad \text{and}\qquad
     \mu_3^{\#}=\mu_2^{\#}-2\;.
\]
 We set $m:=\mu_2^{\#}$.
 By a suitable M\"obius transformation, we may set $G(p_2)=G(1)=\infty$.
 Since $z=1$ is a point of multiplicity $m+1$, $G$ has no pole except
 $z=1$.
 Then $dG$ is written in the form
\[
    dG = c\,z\frac{(z-p)^{m-2}}{(z-1)^{m+2}}\,dz\qquad 
      \bigl(c\in\C\setminus\{0\}\bigr)\;,
\]
 and we can choose $c=1$ by a  M\"obius transformation again.
 Moreover, the Hopf differential $Q=(S(g)-S(G))/2$ is as in 
 \eqref{eq:hopf-o-1-1-2} if and only if 
 $a = -(m+\mu+2)/(m-\mu-2)$.
\end{proof}
\begin{proposition}\label{prop:o-1-2-2-a}
 Let $m$ be a positive integer and $\mu\in(-1,0)$ a real number.
 \begin{enumerate}
  \item\label{item:o-1-2-2-a:1} 
       If $m\geq 3$ and
    \[
        \hspace*{3em}
        p := \frac{m(m+2)-\mu(\mu+2)}{(m-2)^2-\mu^2},\quad
        \theta:=\frac{(\mu-3m+2)^2(m(m+2)-\mu(\mu+2))}{((m-2)^2-\mu^2)^2}\;,
    \]
	then there exists a complete \cmcone{} immersion
	$f\colon{}M:=\C\cup\{\infty\}\setminus\{0,1,p\}\to H^3$
        with hyperbolic Gauss map $G$ and Hopf differential $Q$ 
        so that 
    \[
        dG = z^2 \frac{(z-p)^{m-3}}{(z-1)^{m+2}}\,dz\;,\qquad
         Q = \frac{\theta}{z(z-1)^2(z-p)^2}\,dz^2\;.
    \]
  \item\label{item:o-1-2-2-a:2} 
       If $m\geq 1$ and
    \begin{multline*}
        \hspace*{3em}p := \frac{\mu+m+2}{\mu+m}\;,\quad
        \theta:=\frac{(m-\mu)(\mu+m+2)}{(m+\mu)^2}\\
        \text{and}\quad
        a:=\frac{m-\mu\pm\sqrt{9(m-\mu)^2+16m(\mu+1)+16\mu(m+1)}}{2(\mu+m)}\;,
    \end{multline*}
	then there exists a complete \cmcone{} immersion
	$f\colon{}M:=\C\cup\{\infty\}\setminus\{0,1,p\}\to H^3$
        with hyperbolic Gauss map $G$ and Hopf differential $Q$ 
        so that
    \[
        dG = z^2 \frac{(z-p)^{m-1}}{(z-1)^{m+2}(z-a)^2}\,dz\;,\qquad
         Q = \frac{\theta}{z(z-1)^2(z-p)^2}\,dz^2\;.
    \]
  \end{enumerate}
  In each case, the immersion $f$ is complete, of type $\gO(-1,-2,-2)$, 
  $\Hyp^1$-reducible and satisfies $\TA(f)=8\pi$.

  Conversely, any $\Hyp^1$-reducible \cmcone{} immersion of class
  $\gO(-1,-2,-2)$ with $\TA(f)=8\pi$ is obtained in this way.
\end{proposition}
\begin{proof}
 Since the residue of  $dG$ as in \ref{item:o-1-2-2-a:1} at $z=1$
 vanishes, there exists a meromorphic function $G$.
 Since the metric $ds^2{}^{\#}$ as in \eqref{eq:fund-forms} is
 non-degenerate and complete on $M:=\C\cup\{\infty\}\setminus\{0,1,p\}$, 
 there exists a \cmcone{} immersion $f\colon{}\widetilde M\to H^3$
 with hyperbolic Gauss map $G$ and Hopf differential $Q$ as in 
 \ref{item:o-1-2-2-a:1}, where $\widetilde M$ is the universal
 cover of $M.$
 (In fact, there exists a \cmcone{} immersion 
 $f^{\#}\colon\widetilde M\to H^3$ with Weierstrass data $(G,-Q/dG)$.
 Then taking the dual yields the desired immersion.)
 Let $F$ be the lift of $f$.
 Then $F$ is a solution of \eqref{eq:Fsharp}, and there 
 exists a representation $\rho_F$ as in \eqref{eq:F-repr}.

 The components $F_{21}$ and $F_{22}$ of $F$ 
 satisfy the equation $\mbox{(E.1)}^{\#}$ in
 \cite{ruy3}:
\begin{equation}\label{eq:o-1-2-2-ode}
    X'' - \bigl(\log(\hat\omega^{\#})'\bigr)X'+\hat QX=0,\quad
     \left(\omega^{\#}:=\hat\omega^{\#}\,dz=\frac{Q}{dG},\quad
           Q=\hat Q\,dz^2\right)\;.
\end{equation}
 By a direct calculation, the roots of the indicial equation of
 \eqref{eq:o-1-2-2-ode} at $z=0$ are $0$ and $-2$, and the log-term
 coefficient at $z=0$ vanishes (see Appendix~A of \cite{ruy3}).
 Hence $F_{21}$ and $F_{22}$ are meromorphic on a neighborhood of $z=0$,
 and then, the secondary Gauss map $g=-dF_{22}/dF_{21}$ is meromorphic 
 at $z=0$.
 Hence, by \eqref{eq:g-F-repr},
 $\rho_F(\tau_1)=\pm\rho_g(\tau_1)=\pm\id$, 
 where $\rho_F$ is a representation  corresponding to the secondary Gauss
 map $g$, and $\tau_1$ is a deck transformation corresponding to a loop 
 surrounding $z=0$.
 Moreover, the difference of the roots of the indicial equation 
 at $z=1$ is $\mu+1\not\in\Z$.
 This implies that one can choose the secondary Gauss map $g$ 
 such that $g\circ\tau_2=e^{2\pi i\mu} g$,
 where $\tau_2$ is a deck transformation of $\widetilde M$ corresponding
 to a loop surrounding $z=1$.
 Then $\rho_F(\tau_2)=\pm\rho_g(\tau_2)=
  \diag\{e^{\pi i \mu}, e^{-\pi i \mu} \}\in\SU(2)$.
 Hence the representation $\rho_F$ lies in $\SU(2)$,
 since $\tau_1$ and $\tau_2$ generate the fundamental group of $M$. 
 Then by Proposition~\ref{prop:sutwo}, the immersion $f$ is well-defined
 on $M$, and by Lemma~\ref{lem:complete}, $f$ is a complete immersion.
 Using \eqref{eq:q-first}, we have
 $\mu_1 =2,\quad \mu_2 = \mu,\quad \mu_3 = -\mu-1$.
 Then by \eqref{eq:fact-b}, we have $\TA(f)=8\pi$.

 In the case \ref{item:o-1-2-2-a:2}, we can prove the existence of $f$
 in a similar way.

 Conversely, we assume a complete $\Hyp^1$-reducible immersion $f\colon M\to H^3$ 
 of type $\gO(-1,-2,-2)$ with $\TA(f)=8\pi$ exists.
 Without loss of generality, we may set $(p_1,p_2,p_3)=(0,1,p)$ and the
 only umbilic point $q=\infty$.
 As shown in the proof of Proposition~\ref{prop:o-1-2-2}, we have
 $\mu_1=\mu_1^{\#}=2$.
 Thus, by \eqref{eq:fact-b} and the assumption $\TA(f)=8\pi$,
 we have $\mu_2+\mu_3=-1$.
 Hence by \eqref{eq:fact-g}, we can set
\[
    \mu_2 = \mu,\qquad \mu_3 = -1-\mu\qquad (-1<\mu<0)\;.
\]
 Without loss of generality, we may assume $\mu_2^{\#}\geq \mu_3^{\#}$.
 Then by the Riemann-Hurwitz relation, we have
\begin{equation}\label{eq:o-1-2-2-rh}
     \deg G=\frac{1}{2}\bigl(2 +\mu_1^{\#}+\mu_2^{\#}+\mu_3^{\#}+1\bigr)
           =\frac{5}{2}+\frac{\mu_2^{\#}+\mu_3^{\#}}{2}\leq 
            \frac{5}{2}+\mu_2^{\#}\;.
\end{equation}
 On the other hand, $\deg G\geq \mu_2^{\#}+1$.
 Thus we have $\deg G=\mu_2^{\#}+1$ or $\mu_2^{\#}+2$.
 We set $m:=\mu_2^{\#}$.

 Assume $\deg G=\mu_2^{\#}+1=m+1$.
 Then by \eqref{eq:o-1-2-2-rh}, $\mu_3^{\#}=m-3$.
 Hence, the hyperbolic Gauss map $G$ branches at $0$, $1$, $p$ and
 $\infty$ with branch orders $2$, $m$, $m-3$ and $1$, respectively.
 By a suitable M\"obius transformation, we assume $G(1)=\infty$.
 The multiplicity of $G$ at $z=1$ is $m+1=\deg G$. 
 Then $G$ has no other poles on $\C\cup\{\infty\}$.
 Thus, $dG$ can be written in the form
\[
    dG = c\,z^2 \frac{(z-p)^{m-3}}{(z-1)^{m+2}}\,dz \;,
\]
 where $c\in\C\setminus\{0\}$.
 By a suitable M\"obius transformation, we may set $c=1$.
 On the other hand, the Hopf differential $Q$ is written in the form
\[
    Q = \frac{\theta}{z(z-1)^2(z-p)^2}\,dz^2
\]
 because $f$ is type $\gO(-1,-2,-2)$ and $\infty$ is the umbilic point
 of order $1$.
 Thus, by \eqref{eq:q-first}, we have
\begin{equation}\label{eq:o-1-2-2-a-qfirst}
    c_2-c_2^{\#}=\frac{2\theta}{(1-p)^2}\;,\qquad
    c_3-c_3^{\#}=\frac{2\theta}{(1-p)^2p}\;,
\end{equation}
 where $c_j$ and $c_j^{\#}$ are as in \eqref{eq:c-j}.
 Thus we have the case \ref{item:o-1-2-2-a:1}

 Next, we assume $\deg G=m+2$.
 Then by \eqref{eq:o-1-2-2-rh}, we have $\mu_3^{\#}=m-1$.
 If we set $G(1)=\infty$, then $G$ has only one simple pole 
 other than the pole  $z=1$, since the multiplicity of $G$ at $z=1$ is $m+1$.
 So $dG$ is written in the form
\[
    dG = c\,z^2\frac{(z-p)^{m-1}}{(z-1)^{m+2}(z-a)^2}\,dz\qquad
      (a\in\C\setminus\{0,1,p\})\;,
\]
 where $c\in\C\setminus\{0\}$, which can be set to $c=1$
 by a suitable M\"obius transformation.
 The residue of $dG$ at $z=a$ vanishes if and only if 
 $p = (a(m+1)+a^2)/(ma+2)$.
 On the other hand, the relation \eqref{eq:o-1-2-2-a-qfirst} also 
 holds in this case.
 Thus we have the case \ref{item:o-1-2-2-a:2}.
\end{proof}
\subsubsection*{The case \boldmath{$(\gamma,n)=(0,4)$}}
  If two of the $\mu_j$ are integers, \eqref{eq:fact-b} and
  \eqref{eq:fact-c} imply that $\TA(f)>8\pi$.
  So at most one $\mu_j$ is an integer.

  By \eqref{eq:fact-b}, \eqref{eq:fact-a} and \eqref{eq:fact-g}, we 
  have  $-9 \leq d_1+d_2+d_3+d_4 \leq -4$.
  When all $\mu_j \not\in \Z$, all $d_j\leq -2$.  
  Hence the possible cases are $\gO(-2,-2,-2,-3)$ and
  $\gO(-2,-2,-2,-2)$ (see Example~\ref{exa:four-noids}).

  Assume $\mu_1\geq 0$ is an integer.  
  Then $\mu_2, \mu_3, \mu_4 \not\in \Z$ and $d_2, d_3, d_4 \leq -2$.  
  In this case, by \eqref{eq:fact-b} and \eqref{eq:fact-c}, 
  we have  $-6 \leq d_2+d_3+d_4$. 
  Hence $d_2=d_3=d_4=-2$ and $\mu_1-d_1=2$.
  This implies that $d_1\geq -2$.
  Moreover, by \eqref{eq:fact-a}, we have $d_1\leq 2$.
  Hence the possible cases are $\gO(d,-2,-2,-2)$ with $-2 \leq d \leq 2$.  
  Moreover, $\mu_1=2+d$ holds and there are $2-d$ umbilic
  points.  Then when $d \geq -1$, we have $\mu_1 \in \Z^+$, and so 
  by Corollary 4.7 in \cite{ruy4}, we have 
  $\TA(f)\geq 8\pi$.  So $\TA(f) = 8 \pi$.  
\begin{proposition}\label{prop:go2-2-2-2}
  There exist no \cmcone{} surfaces in $H^3$ with $\TA(f)\leq 8\pi$ of
  class $\gO(2,-2,-2,-2)$.
\end{proposition}
\begin{proof}
  Assume such an immersion
 $f\colon{}\C\cup\{\infty\}\setminus\{p_1,\dots,p_4\}\to H^3$
 exists.
 Then there are no umbilic points, and by \eqref{eq:fact-b},
 \eqref{eq:fact-c} and \eqref{eq:fact-g}, $\mu_1=\mu^{\#}_1=4$ holds.

 Let $G$ be the hyperbolic Gauss map.
 Then $\deg G\geq 5$ because $\mu^{\#}_1=4$.  
 Hence by the Riemann-Hurwitz relation,
\[
    10\leq 2\deg G =\sum_{j=2}^4 \mu^{\#}_j + \mu^{\#}_1 +2
                   =\sum_{j=2}^4 \mu^{\#}_j + 6
\]
  holds.
  This implies that $\mu_2^{\#}+\mu_3^{\#}+\mu_4^{\#}$ is an even
  number not less than $4$:
\begin{equation}\label{eq:sum-musharp-go2-2-2-2}
  \mu_2^{\#}+\mu_3^{\#}+\mu_4^{\#}=2l,\qquad (l\in\Z, l\geq 2)\;.
\end{equation}
  Since $\TA(f)=8\pi$, we have
\begin{equation}\label{eq:sum-mu-go2-2-2-2}
   \mu_2 + \mu_3 + \mu_4 = -2\;.
\end{equation}
  Hence by \eqref{eq:fact-g},
\begin{equation}
    -1 < \mu_j < 0 \qquad (j=2,3,4)\;.
\end{equation}

 We set $(p_1,p_2,p_3,p_4)=(p,0,1,-1)$, where $p\in\C\cup\{\infty\}$.
 We may assume $p_1\in\C$.
 In fact, if $p_1=\infty$, the M\"obius transformation
 \begin{equation}\label{eq:mobius}
  z\longmapsto \frac{z-1}{3z+1} \; ,
 \end{equation}
 maps the ends $(p_1,p_2,p_3,p_4)=(\infty ,0,1,-1)$ to 
 $(1/3,-1,0,1)$.  

 The Hopf differential can be written as
\[
     Q = 2 \theta^2 \frac{(z-p)^2}{z^2(z^2-1)^2}\,dz^2\qquad
      (\theta\in\C\setminus\{0\}) \; .
\]
By the relation \eqref{eq:q-first},
we have
\begin{equation}\label{eq:exponent}
  c_2-c_2^{\#} = 4 \theta^2 p^2,\quad
  c_3-c_3^{\#} = \theta^2 (1-p)^2,\quad
  c_4-c_4^{\#} = \theta^2 (1+p)^2,
\end{equation}
where $c_j$ and $c_j^{\#}$ are as in \eqref{eq:c-j}.
Since $-1<\mu_j<0$ and $\mu_j^{\#}$ is a non-negative integer,
we have
\begin{equation}\label{eq:csign}
   0<c_j<\frac{1}{2} \; , \qquad c_j^{\#}\leq 0 \; ,
\end{equation}
and consequently, $c_j-c_j^{\#}>0$ ($j=2,3,4$).
Let 
\begin{equation}\label{eq:alphadef}
    \alpha_2 = \theta p \; ,\qquad 
    \alpha_3 = \frac{1}{2}\theta(1-p) \; ,\qquad
    \alpha_4 = \frac{1}{2}\theta(1+p) \; .
\end{equation}
Then we have $4 \alpha_j^2 = c_j-c_j^{\#}$,
which implies that the $\alpha_j$ ($j=2,3,4$) are real numbers.  
And then $p = \alpha_2/(\alpha_3+\alpha_4)$ and $\theta = \alpha_3+\alpha_4$ 
are real numbers.
Here, without loss of generality, we may set $0<p<1$.
(In fact, if $p<0$, applying the coordinate change $z\mapsto -z$, 
we have $p>0$.
Moreover, if $p>1$, by the transformation \eqref{eq:mobius},
we have $0<p<1$.)

We choose the sign of $\theta$ as $\theta>0$.
Then, we have $\alpha_2$, $\alpha_3$ and $\alpha_4$ are positive numbers.
Moreover, by \eqref{eq:alphadef},
we have
\begin{equation}\label{eq:alpharel}
   \alpha_2+\alpha_3=\alpha_4 \; .
\end{equation}
Using this, we have
\begin{equation}\label{eq:murel}
  \mu_2^{\#} \; , \quad \mu_3^{\#}~\leq~ \mu_4^{\#} \; .
\end{equation}
In fact, by \eqref{eq:alpharel} we have $\alpha_j < \alpha_4$ for
$j=2,3$.  
Then $c_j-c_j^{\#} < c_4-c_4^{\#}$.  
Hence by \eqref{eq:csign}, $-c_j^{\#} < \frac{1}{2}-c_4^{\#}$.  
By definition, this implies that 
\[
      \mu_j^{\#}(\mu_j^{\#}+2) < 1+\mu_4^{\#}(\mu_4^{\#}+2) \; .  
\]
Thus we have 
\[
      (\mu_j^{\#}-\mu_4^{\#})(\mu_j^{\#}+\mu_4^{\#}+2) < 1 \;\;\; 
      \text{for} \;\; j=2,3 \; .
\]
As the $\mu_j^{\#}$ are non-negative integers, 
$\mu_j^{\#}-\mu_4^{\#}\leq 0$, which implies \eqref{eq:murel}.

By \eqref{eq:exponent} and the definition of $\alpha_j$, we have
\[
 \mu_j(\mu_j+2) =\mu_j^{\#}(\mu_j^{\#}+2)- 8 \alpha_j^2 \qquad (j=2,3,4) \; . 
\]
Since $\mu_j\in(-1,0)$, this implies that
\begin{equation}\label{eq:muval}
    \mu_j +1 = \sqrt{1+\mu_j^{\#}(\mu_j^{\#}+2)-8 \alpha_j^2}
             = \sqrt{(\mu_j^{\#}+1)^2-8 \alpha_j^2}\qquad (j=2,3,4) \; .
\end{equation}
Now, defining $m_j:=\mu_j^{\#}+1\geq 1$ ($j=2,3,4$), 
\eqref{eq:murel} and \eqref{eq:alpharel} imply 
\begin{equation}\label{eq:marel}
    m_2 \; , \; m_3 \; \leq \; m_4\qquad 
    \text{and}\qquad
    \alpha_2+\alpha_3=\alpha_4 \; . 
\end{equation}
Moreover, 
\begin{equation}\label{eq:mrel}
   m_2+m_3\neq m_4
\end{equation}
holds.
To prove this, if $m_2+m_3=m_4$, then $\mu_2^{\#}+\mu_3^{\#}+2 = 
\mu_4^{\#}+1$.  
This implies that $\mu_2^{\#}+\mu_3^{\#}+\mu_4^{\#}$ is an
odd number, contradicting \eqref{eq:sum-musharp-go2-2-2-2}.  

Using \eqref{eq:muval} and \eqref{eq:marel}, 
the equality \eqref{eq:sum-mu-go2-2-2-2} is written as
\begin{equation}\label{eq:musum2}
   \sqrt{m_2^2-8 \alpha_2^2}+\sqrt{m_3^2-8 \alpha_2^2}+
      \sqrt{m_4^2-8 (\alpha_2+\alpha_3)^2}=1 \; .
\end{equation}
We shall prove that \eqref{eq:musum2} cannot hold, making a 
contradiction.  
Let $m_2$, $m_3$, $m_4$ be positive integers which satisfy
\eqref{eq:marel} and \eqref{eq:mrel}.  Define 
\[
    \varphi(\alpha_2,\alpha_3):=
     \sqrt{m_2^2-8 \alpha_2^2}+\sqrt{m_3^2-8 \alpha_3^2}+
           \sqrt{m_4^2-8 (\alpha_2+\alpha_3)^2} \; .
\]
on the closure $\overline D$ of the open domain 
\[
    D:=\left\{(\alpha_2,\alpha_3)\,:\, 0<\alpha_2<\frac{m_2}{\sqrt{8}}, 
                                  0<\alpha_3<\frac{m_3}{\sqrt{8}}, 
                                  0<\alpha_2+\alpha_3<\frac{m_4}{\sqrt{8}}
       \right\}
\]
in the $\alpha_2\alpha_3$-plane.
Then it holds that
\[
 \varphi(\alpha_2,\alpha_3) >1 \quad \text{if} \quad (\alpha_2,\alpha_3)\in 
 D \; .
\]
To prove this, note that 
since $\varphi$ is a continuous function on a compact set $\overline D$,
it takes a minimum on $\overline D$.
By a direct calculation, we have
\[
  \frac{\partial \varphi}{\partial\alpha_3}
    =\frac{-8 \alpha_3}{\sqrt{m_3^2-8 \alpha_3^2}}+
     \frac{-8 \alpha_2-8 \alpha_3}{\sqrt{m_4^2-8 
                         (\alpha_2+\alpha_3)^2}}<0\qquad
     \text{on}\quad D \; .
\]
So $\varphi$ does not take its minimum in the interior $D$ of $\overline D$, 
but rather on $\partial D$.  Similarly, $\partial \varphi / \partial 
\alpha_2 < 0$ on $D$, so the minimum occurs at $(\alpha_2,\alpha_3)=
(m_2/\sqrt{8},m_3/\sqrt{8})$, where 
\[ \varphi (m_2/\sqrt{8},m_3/\sqrt{8}) = \sqrt{m_4^2-
(m_2+m_3)^2} \geq 1 \; , \] 
if the line $\alpha_2+\alpha_3 = m_4/\sqrt{8}$ does not intersect 
$\partial D$.  
If the line $\alpha_2+\alpha_3 = m_4/\sqrt{8}$ does intersect $\partial D$, 
 then $m_2+m_3 >m_4$ holds, and the minimum occurs somewhere on this line 
 with $\alpha_2$ in the interval $[(m_4-m_3)/\sqrt{8},m_2/\sqrt{8}]$.  
 We have 
\[
  \varphi(\alpha_2,m_4/\sqrt{8}-\alpha_2)=
  \sqrt{m_2^2-8 \alpha_2^2} + \sqrt{m_3^2-8 (m_4/\sqrt{8}-\alpha_2)^2} \; , 
\]
which minimizes at the endpoints of the interval 
$[(m_4-m_3)/\sqrt{8},m_2/\sqrt{8}]$, where its values are 
\[ 
       \sqrt{(m_3+m_2-m_4)(m_3+m_4-m_2)} \; , \;\; 
       \sqrt{(m_3+m_2-m_4)(m_2+m_4-m_3)} \geq 1 \; . 
\] Hence $\varphi>1$ on $D$, contradicting \eqref{eq:sum-mu-go2-2-2-2} and 
 proving the theorem.  
\end{proof}

\begin{remark}\label{rem:o-2-2-20}
 There exist \cmcone{} surfaces of class $\gO(-2,-2,-2,0)$
 with $\TA(f)=8\pi$:
 We set $(p_1,p_2,p_3,p_4)=(1,-1,\infty,0)$ and 
 $M:=\C\cup\{\infty\}\setminus\{p_1,p_2,p_3,p_4\}$.
 Set
\begin{equation}\label{eq:o-2-2-20-gauss}
   dg := \frac{(z^2-1)^{\mu}(z^2-q^2)z^2}{(z^2-a^2)^2}\,dz\qquad
   (a,q \in\C\setminus\{0, 1\}, \; a\neq \pm q, \; \mu \in \R)
\end{equation}
 and
\begin{equation}\label{eq:q}
   Q := \frac{-\mu (\mu+2)}{q^2-1}\frac{z^2-q^2}{(z^2-1)^2}\,dz^2\;,
\end{equation}
   where
\[
  \frac{2\mu a}{a^2-1}+\frac{2a}{a^2-q^2}+\frac{1}{a}=0\;.
\]
 Then the residues of $dg$ at $z=\pm a$ vanish, and thus, there exists
 the secondary Gauss map $g$ as in \eqref{eq:o-2-2-20-gauss}.

 We assume
\[
    -1<\mu<-\frac{1}{2} \qquad 
    \text{and}\qquad
    a^2 = -\frac{1-\mu-q^2}{3+\mu-3q^2}\;.
\]
 Then by Theorem~2.4 of \cite{uy1}, there exists a \cmcone{} immersion 
 $f\colon{}M\to H^3$ with given $g$ and $Q$.
 One can check that such an immersion is complete and has 
 $\TA(f)=8\pi$.
\end{remark}

\subsection*{The case \boldmath{$(\gamma,n)=(0,5)$}}
In this case, by Corollary~\ref{cor:8pi-class}, the only possible case
is $\gO(-2,-2,-2,-2,-2)$.

\subsection*{The case of \boldmath$\gamma=1$}
By Corollary~\ref{cor:8pi-class}, a surface of this type has at most
three ends, and if the surface has $3$ ends, the only possible case is
$\gI(-2,-2,-2)$.
If a surface has only one end, 
part (3) of Lemma~\ref{lem:general-class} implies that
it must be of type $\gI(-3)$ or $\gI(-4)$.  

Now suppose there are two ends.  By \eqref{eq:fact-a}, \eqref{eq:fact-b} 
and \eqref{eq:fact-g}, we have 
\begin{equation}\label{genusonecase}
 -5 \leq d_1+d_2 \leq 0 \; . 
\end{equation}
Also, by \eqref{eq:fact-b} and \eqref{eq:fact-c}, $\TA(f) = 8 \pi$ if 
$d_1,d_2 \geq -1$.
  
Suppose that both ends are regular (i.e.\ $d_1,d_2 \geq -2$).  
Then Theorem~7 of \cite{ruy2} implies that if $d_1=-2$, then 
also $d_2=-2$.  
Furthermore, by Lemma~3 of \cite{uy5} combined with \eqref{eq:fact-b},
\eqref{eq:fact-c} and \eqref{eq:fact-d}, 
if $d_j \geq -1$, then the end at $p_j$ is embedded.  
Therefore, when $d_j \geq -1$, 
Proposition~\ref{prop:flux} implies that the flux at the end $p_j$ is
zero if and only if $d_j \geq 0$.  
By the balancing formula \eqref{eq:balancing} and
Proposition~\ref{prop:flux}, we conclude that 
the only possibilities are $\gI(-2,-2)$, $\gI(-1,-1)$, and $\gI(0,0)$.  
But in fact the case $\gI(0,0)$ cannot occur, because then \eqref{eq:fact-a} 
and \eqref{eq:fact-h} imply that the hyperbolic Gauss map $G$ has at most two 
branch points, contradicting \eqref{eq:fact-e}.  

If the end $p_1$ is irregular, $d_1\leq -3$.
Then by \eqref{genusonecase}, we have $d_2 \geq -2$.
In particular, the other end $p_2$ is regular.
When $d_2\geq -1$, then $\mu_1,\mu_2 \in \Z$, and 
\eqref{eq:fact-b} and \eqref{eq:fact-c} imply $\mu_1-d_1=\mu_2-d_2=2$.  In 
particular $d_1 = \mu_1-2 > -3$, a contradiction.  
Hence the only possible case is $\gI(-2,-3)$.
\begin{remark}\label{rem:i-2-2}
 The genus one catenoid cousin in \cite{rs} 
 is of type $\gI(-2,-2)$ (Figure~\ref{fig:genus-one} in
 Appendix~\ref{app:further-ex}). 
 However, the total absolute curvature
 seems to be strictly greater than $8\pi$.
\end{remark}
\begin{remark}\label{rem:i-2-2-2}
 There exists an example of \cmcone{} surface of type $\gI(-2,-2,-2)$,
 which is so-called the genus one trinoid
 (\cite{ruy1}, see Figure~\ref{fig:genus-one} in
 Appendix~\ref{app:further-ex}), which is obtained by
 deforming minimal surface in $\R^3$.
 However, the absolute total curvature of the original  minimal surface is
 $12\pi$, so the obtained \cmcone{} surface has $\TA(f)$ close to $12\pi$.  
 Thus, surfaces obtained by deformation 
 are far from satisfying $\TA(f)\leq 8\pi$.
\end{remark}
\section{Further examples}
\label{app:further-ex}
In this appendix, we introduce examples of interesting \cmcone{}
surfaces which do not appear in the classification table
(Table~\ref{tab:ta-8pi}).
\begin{figure}
\footnotesize
\begin{center}
  \begin{tabular}{c@{\hspace{4em}}c@{\hspace{4em}}c}
   \includegraphics[width=0.9in]{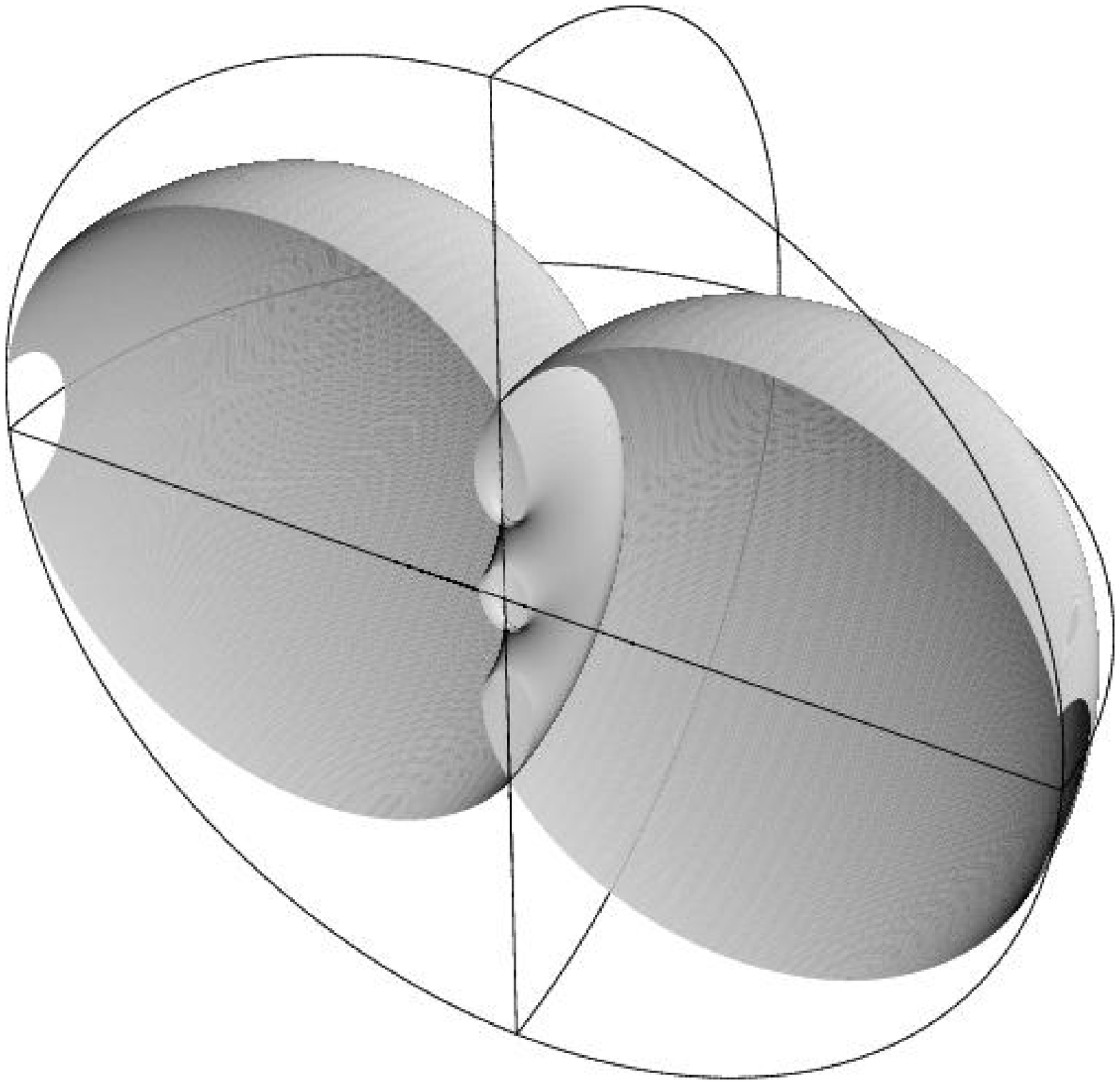} &
   \includegraphics[width=0.9in]{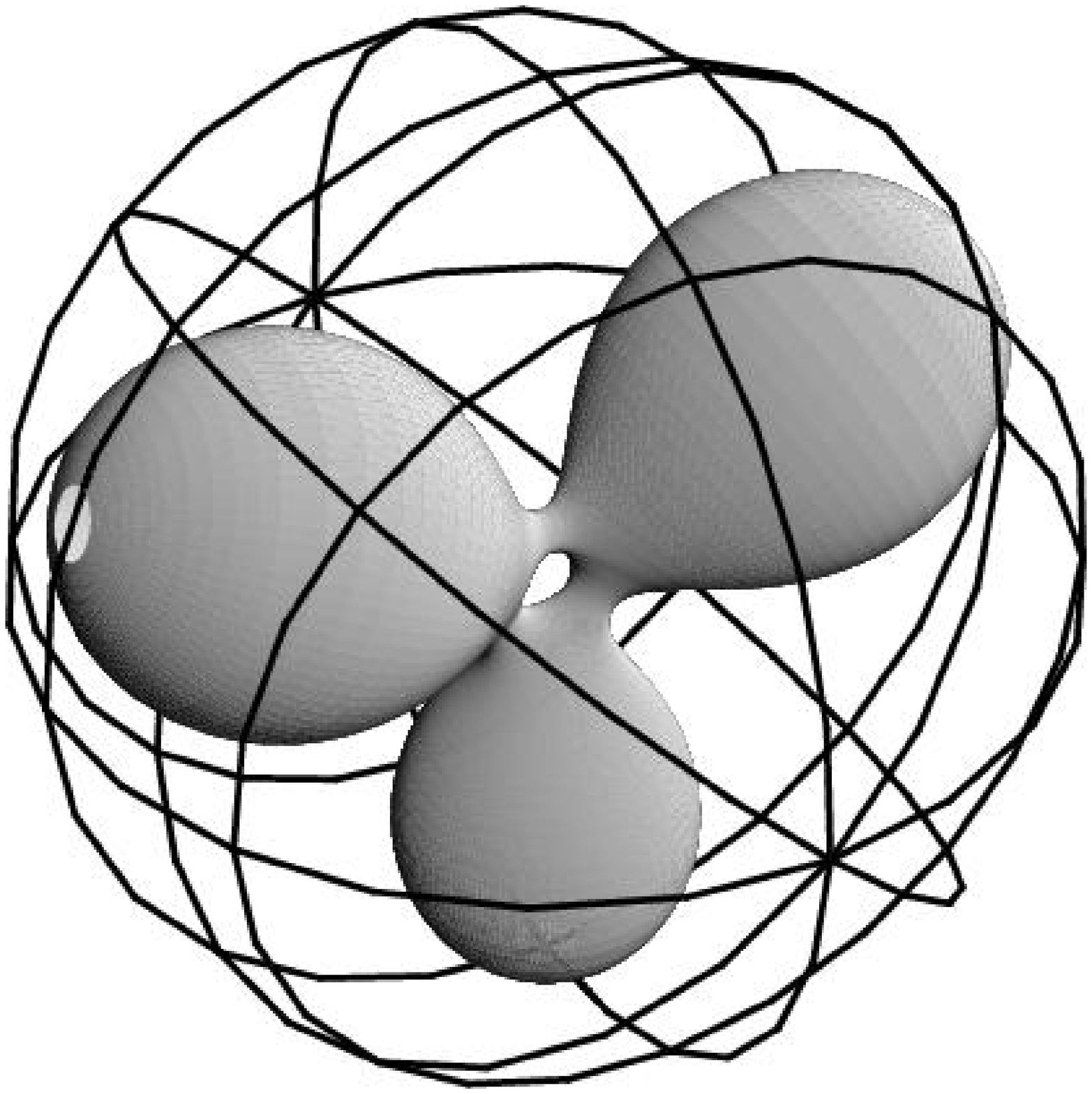} &
   \includegraphics[width=0.9in]{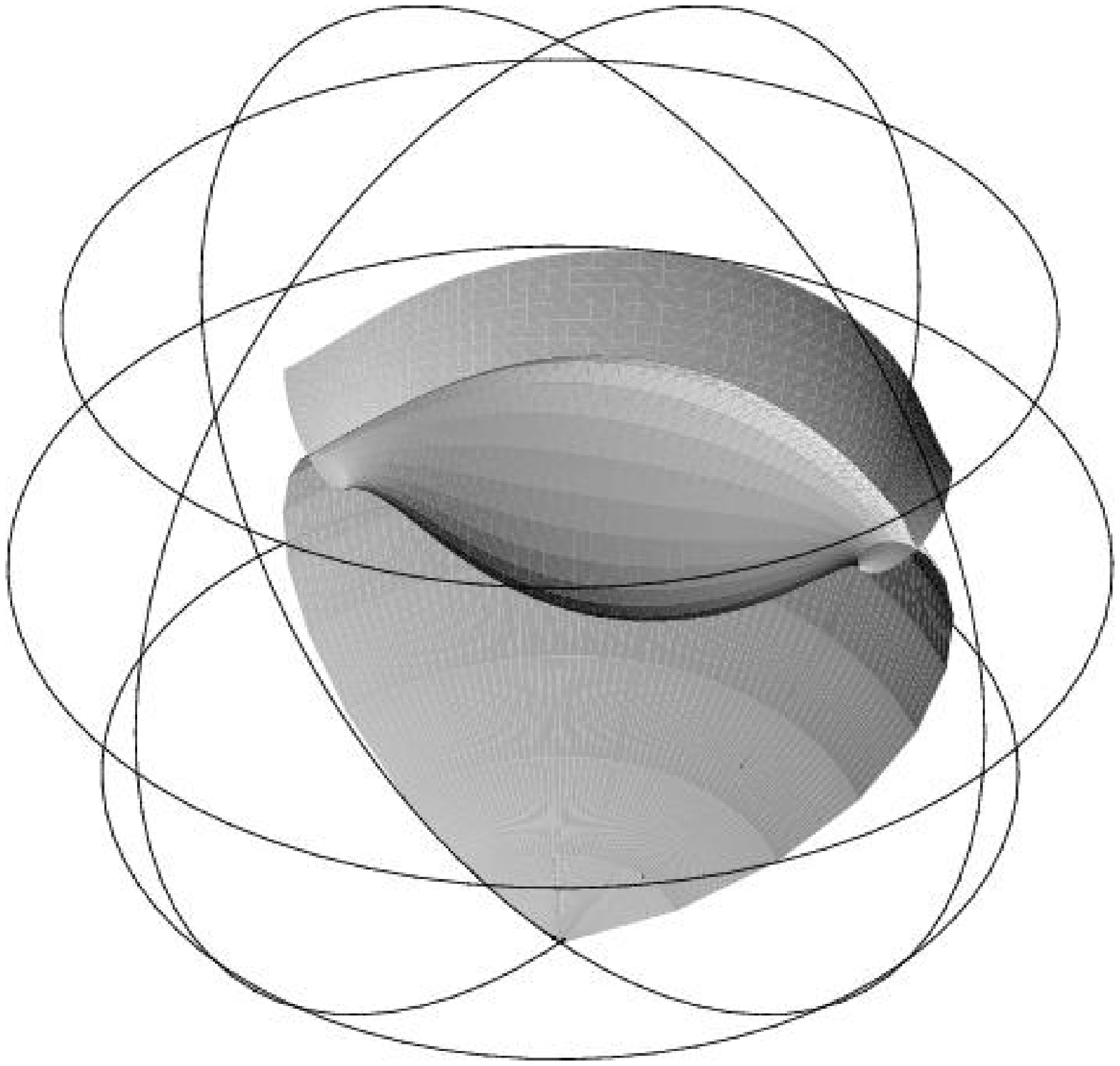} \\[6pt]
    (a) & (b) & (c)
  \end{tabular}
\end{center}
\caption{
 Examples~\ref{ex:genus-one-catenoid}, \ref{ex:genus-one-trinoid}
 and \ref{ex:uy1-5noid}.
 The first two graphics were made by Katsunori Sato of Tokyo Institute
 of Technology.  
}
\label{fig:genus-one}
\end{figure}
\begin{figure}
\begin{center}
\begin{tabular}{c@{\hspace{3em}}c}
  \includegraphics[width=0.9in]{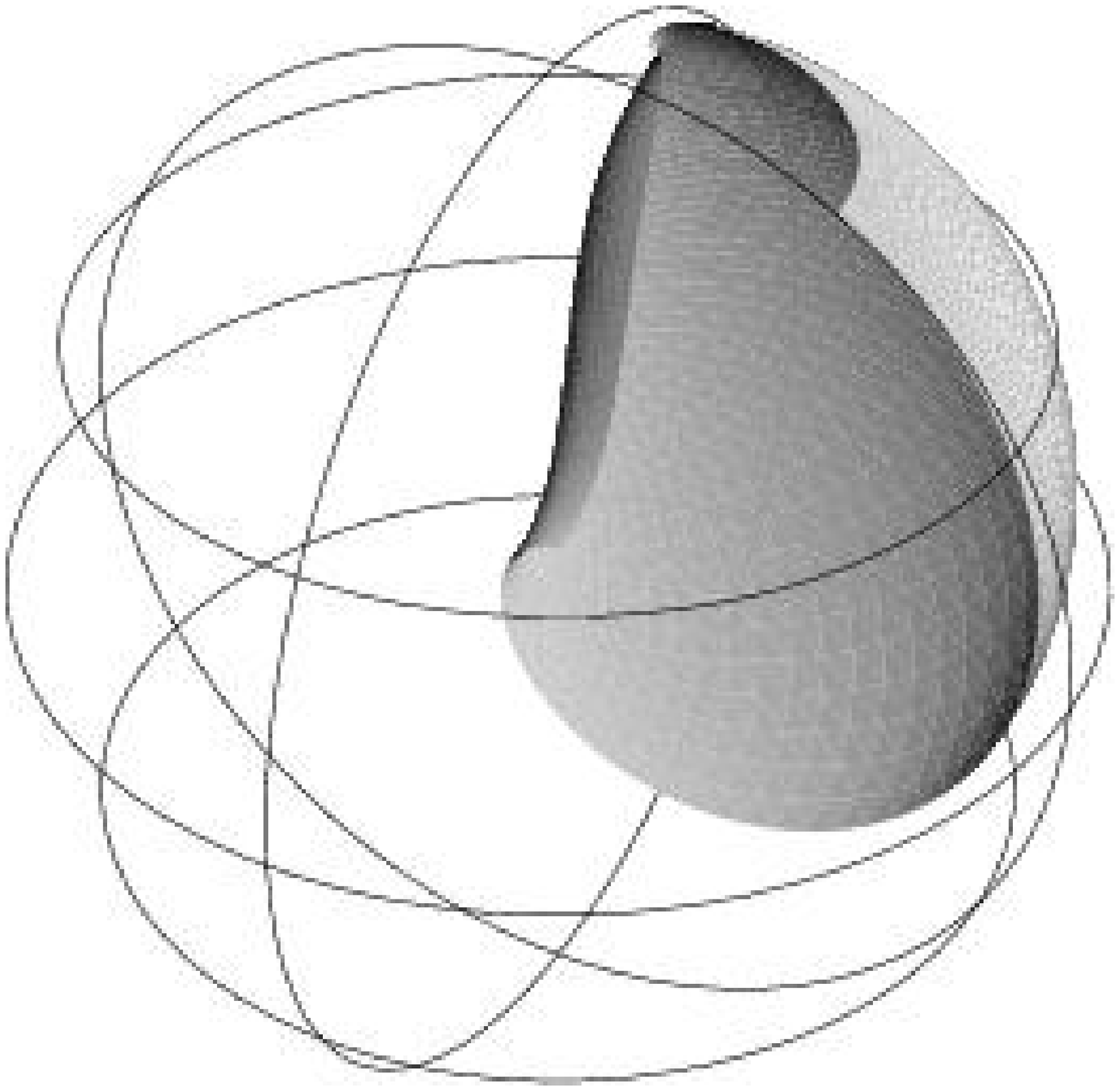} &
  \includegraphics[width=0.9in]{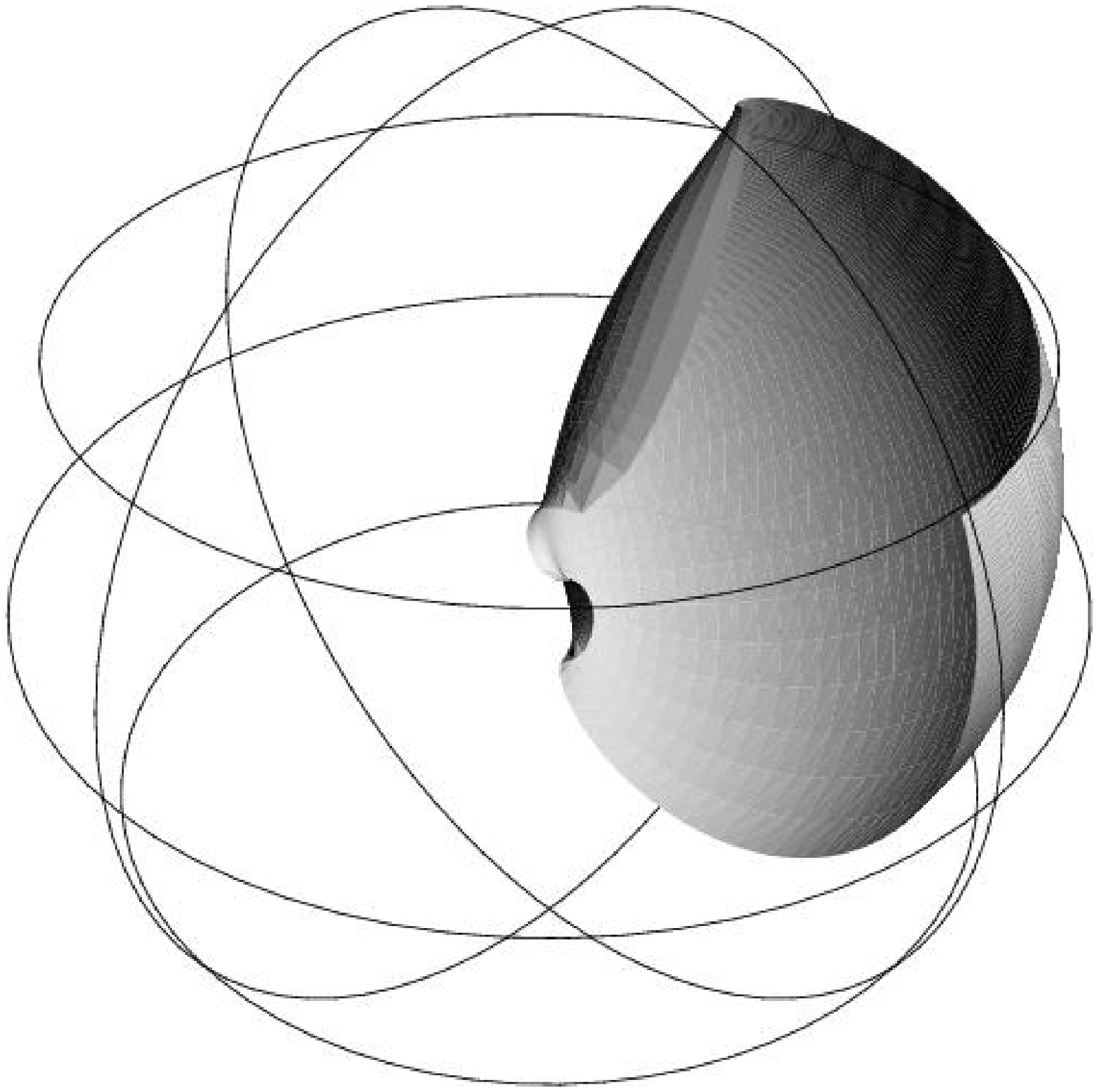}
\end{tabular}
\end{center}
\caption{Example~\ref{ex:genus-enneper}}
\label{fig:enneper2}
\end{figure}
\begin{figure}
\begin{center}
\begin{tabular}{c@{\hspace{3em}}c@{\hspace{3em}}c}
       \includegraphics[width=0.9in]{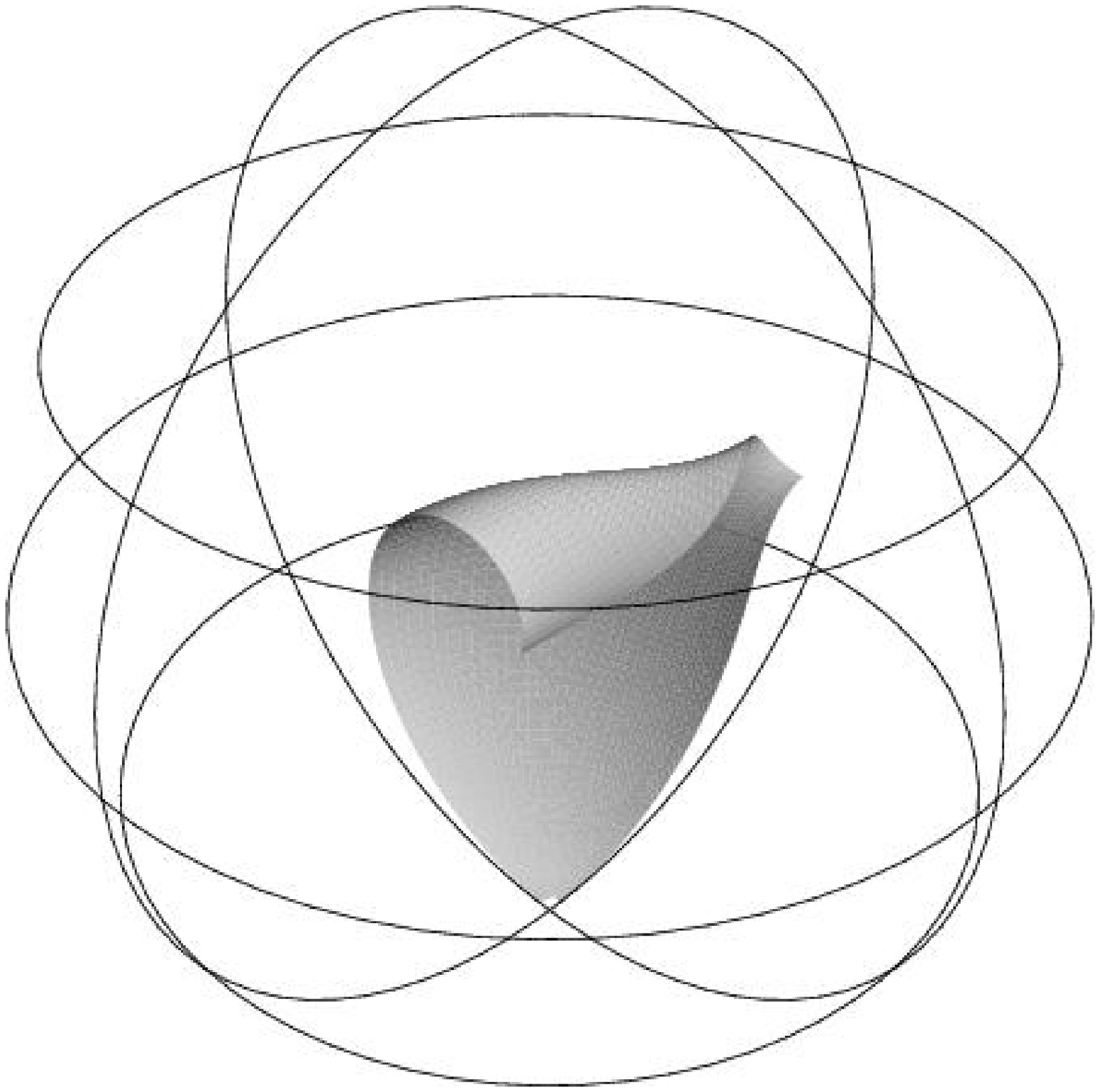}&
       \includegraphics[width=0.9in]{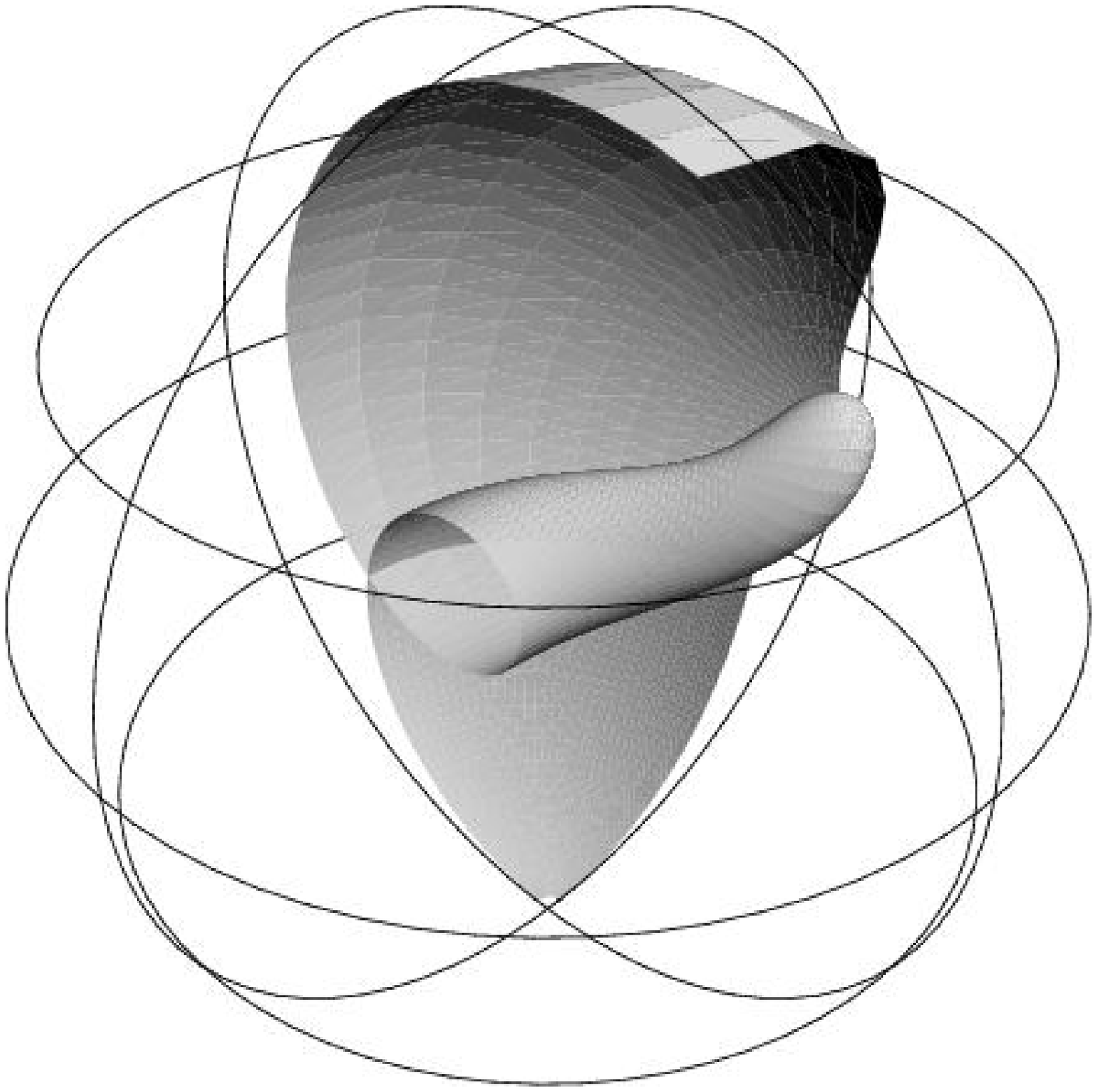}&
       \includegraphics[width=0.9in]{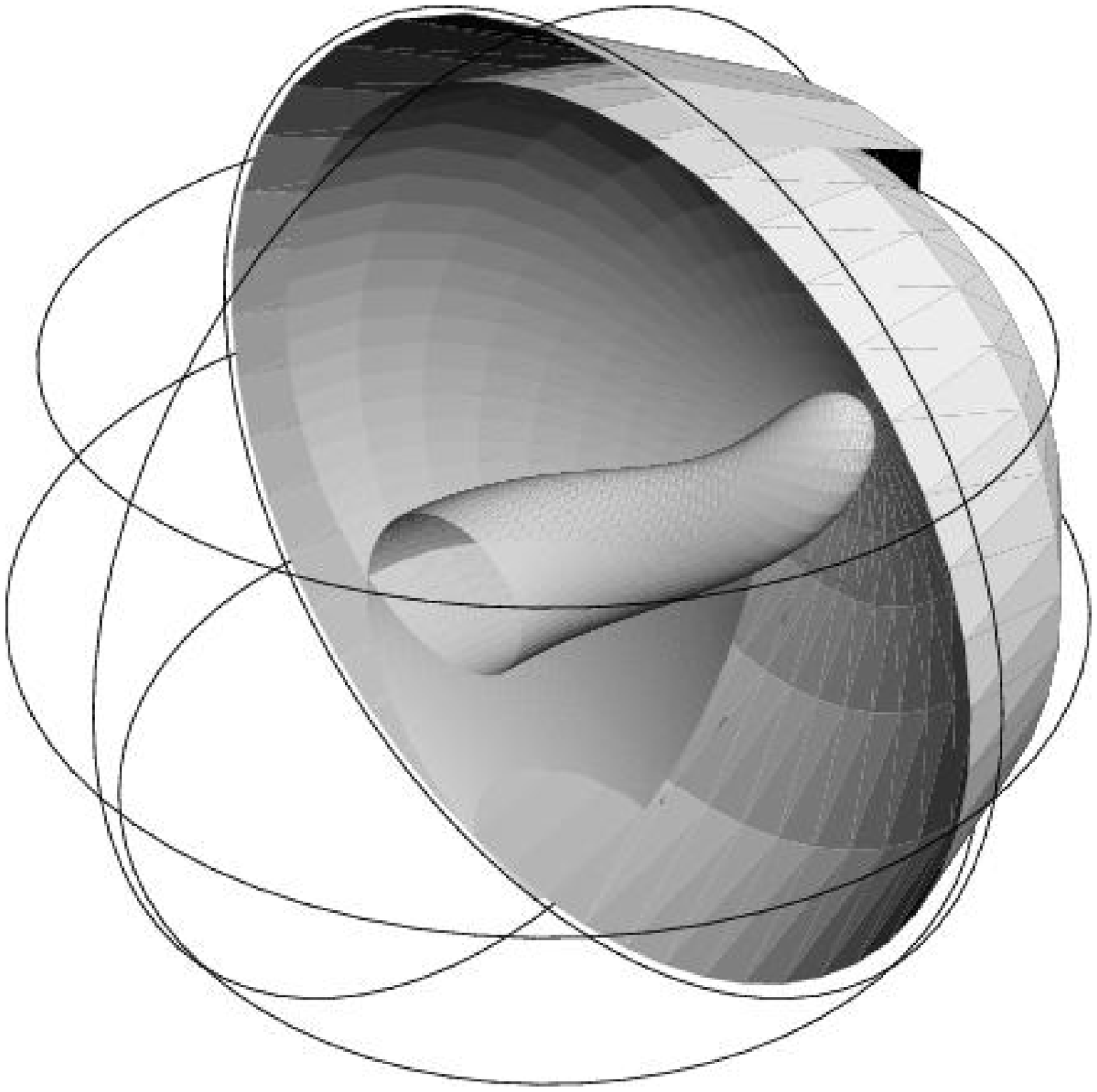} 
\end{tabular}
\end{center}
\caption{
 Example~\ref{ex:uy1-irreg}
}
\label{fig:uy1-irreg}
\end{figure}
\begin{figure}
\begin{center}
\begin{tabular}{c@{\hspace{3em}}c@{\hspace{3em}}c}
       \includegraphics[width=0.8in]{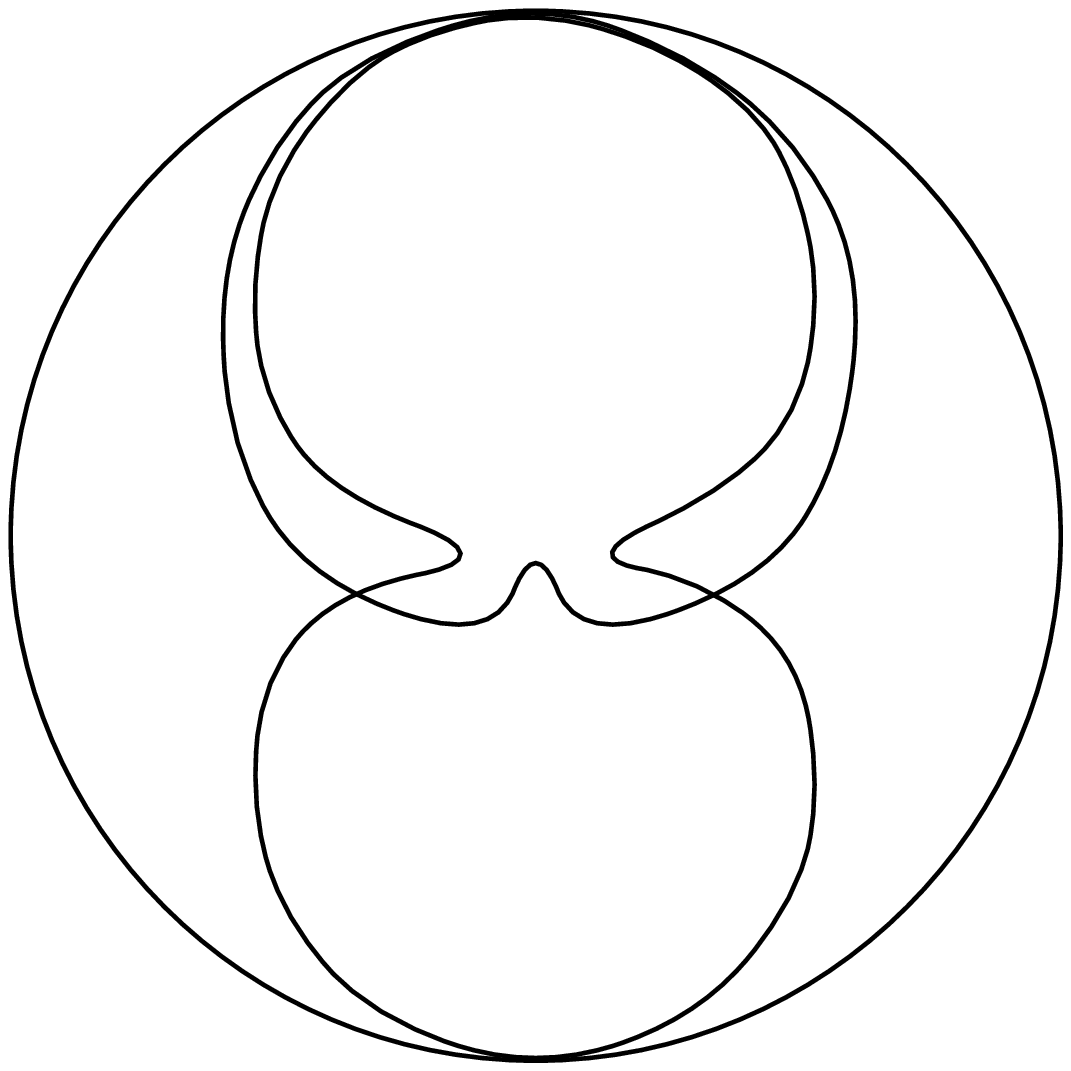} &
       \includegraphics[width=0.8in]{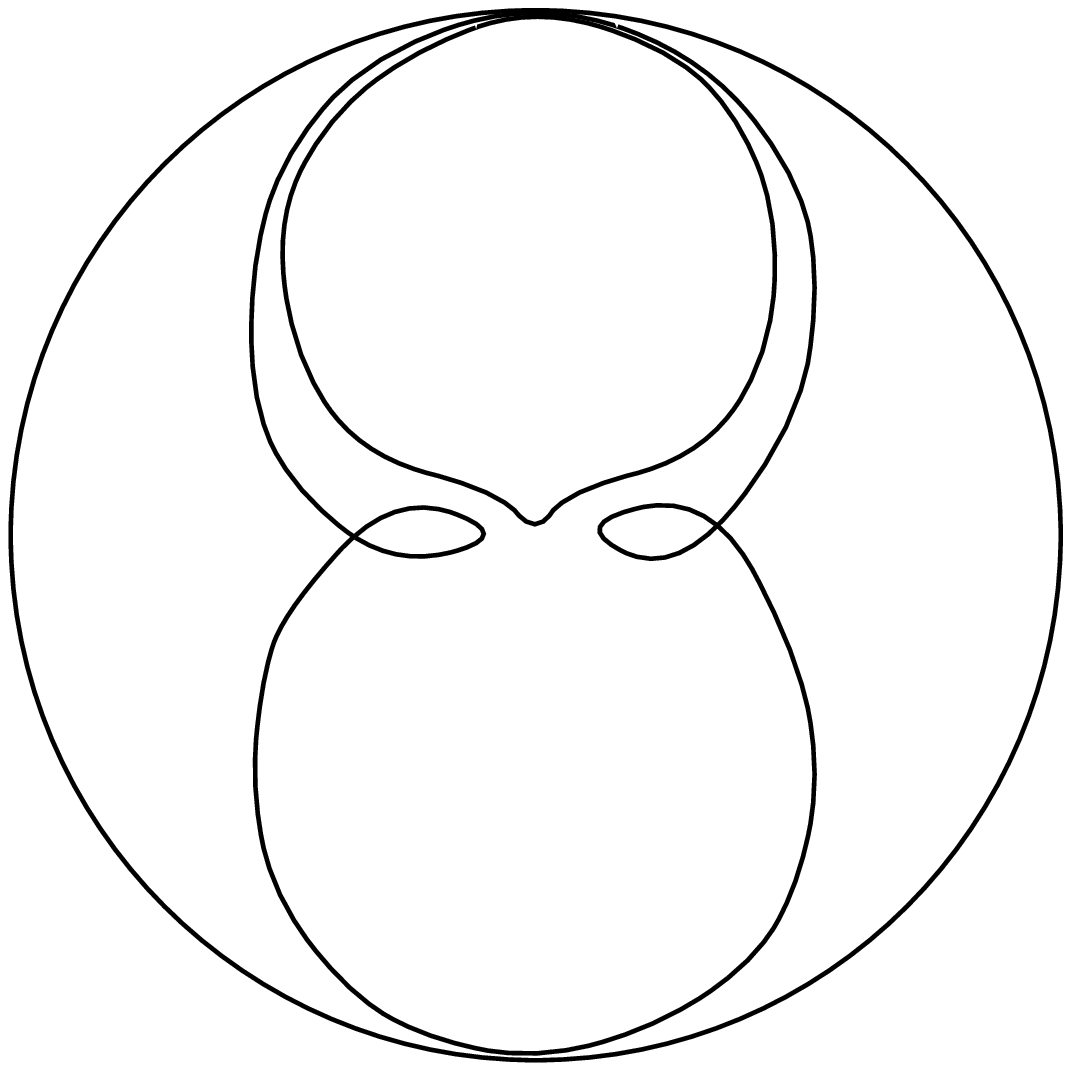} &
       \raisebox{4ex}{
       \includegraphics[width=0.9in]{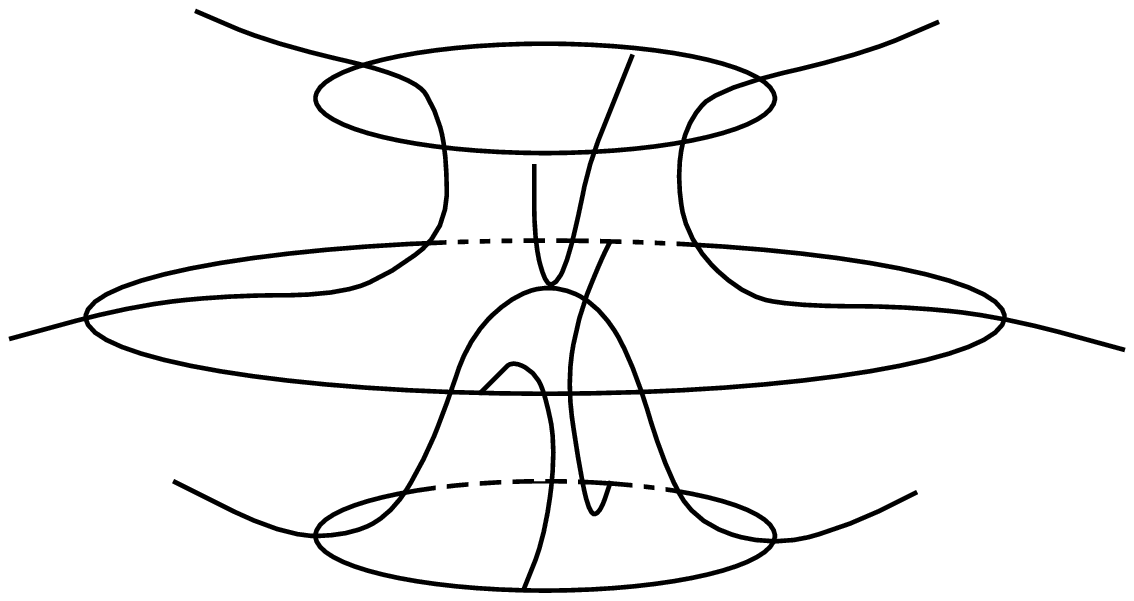}}
\end{tabular}
\end{center}
\caption{
  Example~\ref{ex:costa}
}
\label{fig:costa}
\end{figure}
\begin{example}\label{ex:genus-one-catenoid}
 There exists a genus $1$ catenoid cousin in $H^3$ \cite{rs}
 Figure~\ref{fig:genus-one} (a)).
 No corresponding minimal surface can exist, by Schoen's result \cite{S}.  
 Levitt and Rosenberg \cite{LR} have proved that any complete properly 
 embedded \cmcone{} surface in $H^3$ with asymptotic boundary 
 consisting of at most two points is a surface of revolution, which implies 
 that this example and the last two examples in 
 Figure~\ref{fig:catenoid} cannot  be embedded, and we see that they are
 not.
 \cmcone{} genus $1$ catenoid cousin in $H^3$, proven to exist in 
 \cite{rs}. 
 See Remark~\ref{rem:i-2-2} in Appendix~\ref{app:detailed}.
\end{example}
\begin{example}\label{ex:genus-one-trinoid}
 Figure~\ref{fig:genus-one} shows a genus $1$ trinoid in $H^3$
 proven to exist in \cite{ruy1}.
 See Remark~\ref{rem:i-2-2-2} in Appendix~\ref{app:detailed}.
\end{example}
\begin{example}\label{ex:uy1-5noid}
  Figure~\ref{fig:genus-one} (c) shows $5$ ended \cmcone{} surface in
  $H^3$  found in \cite{uy1}.  
  Here we show only one of six congruent disks that form the surface.  
  The full surface is constructed by reflections across planes 
  containing boundary curves of the disk shown here.
\end{example}
\begin{example}\label{ex:genus-enneper}
 Figure~\ref{fig:enneper2} shows genus $0$ and genus $1$ Enneper
 cousin duals.  
 Each surface has a single end that triply wraps around its 
 limiting point at the south pole of the sphere at infinity.  
 These surfaces are of type $\gO(-4)$ and $\gI(-4)$, and 
 have $\TA(f^\#)=4\pi$ and $\TA(f^\#)=8\pi$.  In both cases 
 only one of four congruent pieces (bounded by planar geodesics) 
 of the surface is shown.
\end{example}
\begin{example}\label{ex:uy1-irreg}
 Figure~\ref{fig:uy1-irreg} shows a \cmcone{} surface in $H^3$, proven
 to exist in \cite{uy1}.  
 This  example is interesting because the hyperbolic Gauss map has an
 essential singularity at one of its two ends, like the end of the
 Enneper cousin.  
 And the geometric behavior of the end here is strikingly similar to that 
 of the Enneper cousin's end (see the middle figure of 
 Figure~\ref{fig:enneper}).  
 Here we show three pictures consecutively 
 including more of this end.
\end{example}
\begin{example}\label{ex:costa}
 \cmcone{} ``Costa cousin'' in $H^3$ was proven to exist by 
 Costa and Sousa Neto \cite{CN}.  
 In Figure~\ref{fig:costa}, 
 rather than showing graphics of this surface, we show two vertical
 cross sections by which the surface is reflectionally symmetric
 (including the ``circles'' at infinity), 
 and a schematic of the central portion of the surface.
\end{example}



\begin{thebibliography}{RUY5}
\bibitem[Bar]{ba}
 E. L. Barbanel,
 {\itshape Complete minimal surface in $\R^3$ of low total curvature},
 Ph.~D. thesis, Univ.~of Massachusetts, 1987.
\bibitem[Bry]{Bryant}
  R.~Bryant,
  {\itshape Surfaces of mean curvature one in hyperbolic space},
  Ast\'erisque {\bfseries 154--155} (1987), 321--347.
\bibitem[CR]{CR}
 C. P. Cos\'\i{}n and A. Ros, 
  {\itshape A plateau problem at infinity for properly
            immersed minimal surfaces with finite total curvature}, 
 Indiana Univ. Math. J. 50 (2001), 847--879.  
\bibitem[CG]{cg}
  C.~C.~Chen, F.~Gackstatter,
  {\itshape Elliptische und hyperelliptische Funktionen und
    vollst\"andige Minimalfl\"achen vom Enneperschen Typ},
    Math.~Ann. {\bfseries 259} (1982), 359--369.
\bibitem[CHR1]{chr}
  P.~Collin, L.~Hauswirth and H.~Rosenberg,
  {\itshape The geometry of finite topology Bryant surfaces},
  Ann. of Math. (2) 153, (2001), 623-659.
\bibitem[CHR2]{chr2}
  \bysame,
  {\itshape The gaussian image of mean curvature one surfaces in 
            $H^3$ of finite total curvature}, 
  In: Minimal Surfaces, Geometric Analysis and Symplectic 
  Geometry, Adv. Stud. Pure Math. 34, Math. Soc. Japan, Tokyo, 
  2002, pp. 9--14.
\bibitem[CN]{CN} 
  C. J. Costa, V. F. de Sousa Neto, 
  {\itshape Mean curvature $1$ surfaces of Costa type in hyperbolic
            3-space}, 
  Tohoku Math. J. (2) 53 (2001), 617--628.  
\bibitem[ET1]{ET1} 
  R. Sa Earp, E. Toubiana, 
  {\itshape On the geometry of constant mean curvature one surfaces in
            hyperbolic space}, 
	Illinois J. Math. 45 (2001), 371--401.  
\bibitem[ET2]{ET2} 
  \bysame, 
  {\itshape Meromorphic data for mean curvature one surfaces in
            hyperbolic space}, 
  Tohoku Math. J. (2) 56 (2004), 27--64.
\bibitem[GGN]{ggn} 
  C. C. G\'oes, M. E. E. L. Galv\~ao, B. Nelli, 
  {\itshape A type Weierstrass representation for surfaces in 
            hyperbolic space with mean curvature one}, 
  An. Acad. Brasil Cienc. 70 (1998), 1--6.
\bibitem[Kar]{K} 
  H. Karcher, 
  {\itshape Hyperbolic constant mean curvature one 
            surfaces with compact fundamental domains}, 
   In: Global Theory of Minimal Surfaces, Clay Math. Proc. 2, Amer. Math. Soc., 
   Providence, RI, 2005, pp. 311--323.
\bibitem[Kat]{Ka} 
  S. Kato, 
  {\itshape Construction of $n$-end catenoids with prescribed flux}, 
  Kodai Math.~J. {\bfseries 18(1)} (1995), 86--98.  
\bibitem[LR]{LR} 
  G. Levitt, H. Rosenberg, 
 {\itshape Symmetry of constant mean curvature hypersurfaces in 
	    hyperbolic space}, 
  Duke Math.~J. {\bfseries 52(1)} (1985), 53--59.  
\bibitem[Lop]{Lopez}
  F.~J.~Lopez,
   {\itshape The classification of complete minimal surfaces with
             total curvature greater than $-12\pi$},
   Trans.~Amer.~Math.~Soc. {\bfseries 334} (1992), 49--74.
\bibitem[MU]{MU} 
  C. McCune, M. Umehara, 
 {\itshape An analogue of the UP-iteration for constant mean curvature
           one surfaces in hyperbolic 3-space}, 
          Diff. Geom. Appl. 20 (2004), 197--207.  
\bibitem[Oss]{O}
   R.~Osserman,
  {\sc A Survey of Minimal Surfaces}, {2nd ed.}, Dover (1986).  
\bibitem[Ros]{Ro} 
  H. Rosenberg, 
  {\itshape Bryant surfaces}, 
  Lecture Notes in Math. 1775, Springer-Verlag, 2002, pp. 67-111.  
\bibitem[RS]{rs}
  W.~Rossman, K.~Sato,
  {\itshape Constant mean curvature surfaces with two ends in hyperbolic
            space}, 
  Experimental Math. {\bfseries 7(2)} (1998), 101--119.
\bibitem[RUY1]{ruy1}
  W.~Rossman, M.~Umehara and K.~Yamada,
  {\itshape Irreducible constant mean curvature 1 surfaces in
    hyperbolic space with positive genus}, 
  T\^ohoku Math.~J. {\bfseries 49}  (1997), 449--484.
\bibitem[RUY2]{ruy2}
  \bysame,
  {\itshape A new flux for mean curvature $1$ surfaces
    in hyperbolic $3$-space, and applications},
    Proc.~Amer.~Math.~Soc. {\bfseries 127} (1999), 2147--2154.
\bibitem[RUY3]{ruy3}
  \bysame,
  {\itshape Mean curvature $1$ surfaces with low total curvature 
    in hyperbolic $3$-space I}, Hiroshima Math. J. 34 (2004), 21--56.  
\bibitem[RUY4]{ruy4}
  \bysame,
  {\itshape Mean curvature $1$ surfaces with low total curvature 
    in hyperbolic $3$-space II}, Tohoku Math. J. (2) 55 (2003), 375--395.  
\bibitem[RUY5]{ruy5}
  \bysame,
  {\itshape Constructing mean curvature $1$ surfaces in $H^3$ with
	     irregular ends }, In: Global theory of minimal surfaces, 
   Clay Math. Proc. 2, Amer Math. Soc., Providence, RI, 2005, pp. 561--584.
\bibitem[Sch]{S} 
  R. Schoen,
  {\itshape Uniqueness, symmetry, and embeddedness of minimal surfaces},
  J.~Diff.~Geom. {\bfseries 18} (1982), 791--809.  
\bibitem[Sm]{Sm} 
  A.~J.~Small,
  {\itshape Surfaces of constant mean curvature $1$ in $H^3$ and
	algebraic curves on a quadric}, 
  Proc.~Amer.~Math.~Soc. 
  {\bfseries 122} (1994), 1211--1220.
\bibitem[Tan]{ta}
  S.~Tanaka,
  {\itshape Minimal surfaces with three catenoidal ends}, 
   Master thesis, Hiroshima Univ., 2001.
\bibitem[Tro1]{Troyanov1}
  M.~Troyanov,
  {\itshape Metric of constant curvature on a sphere with two conical
       singularities}, 
  in ``Differential Geometry'', Lect. Notes in Math. vol.~1410, 
  Springer-Verlag,  (1989), 296--306. 
\bibitem[Tro2]{Troyanov2}
   \bysame,
  {\itshape Prescribing curvature on compact surfaces with conical
          singularities},
   Trans. Amer.~Math.~Soc. {\bfseries 324} (1991), 793--821.
\bibitem[UY1]{uy1}
  M.~Umehara and K.~Yamada,
  {\itshape Complete surfaces of constant mean curvature-$1$
       in the hyperbolic $3$-space},
  {Ann. of Math. {\bfseries 137} (1993), 611--638.}
\bibitem[UY2]{uy2}
  \bysame,
  {\itshape A parameterization of Weierstrass formulae and 
     perturbation of some complete minimal surfaces of
        $\R^3$ into the hyperbolic $3$-space},
   J. reine u.~angew.~Math. {\bfseries 432} (1992), 93--116.
\bibitem[UY3]{uy3}
   \bysame,
   {\itshape Surfaces of constant mean curvature-$c$
        in $H^3(-c^2)$ with prescribed hyperbolic Gauss map},
   Math. Ann. {\bfseries 304} (1996), 203--224.
\bibitem[UY4]{uy4}
   \bysame,
   {\itshape Another construction of a CMC $1$ surface in $H^3$},
    Kyungpook Math. J. {\bfseries 35} (1996), 831--849.
\bibitem[UY5]{uy5}
   \bysame,
   {\itshape A duality on CMC $1$ surfaces in hyperbolic $3$-space
        and a hyperbolic analogue of the Osserman Inequality},
    Tsukuba J. Math. {\bfseries 21} (1997), 229-237.
\bibitem[UY6]{uy6}
   \bysame,
   {\itshape Geometry of surfaces of constant mean curvature $1$ in the 
    hyperbolic $3$-space}, 
     Suugaku Expositions {\bfseries 10(1)} (1997), 41--55.
\bibitem[UY7]{uy7}
   \bysame,
   {\itshape Metrics of constant curvature $1$ with three conical 
   singularities on the $2$-sphere}, 
     Illinois J. Math. {\bfseries 44(1)} (2000), 72--94.
\bibitem[Yu1]{Yu1}
    Z.~Yu,
   {\itshape Value distribution of hyperbolic Gauss maps},
  Proc.~Amer.~Math.~Soc. 
  {\bfseries 125} (1997), 2997--3001.
\bibitem[Yu2]{Yu2}
   \bysame,
   {\itshape The inverse surface and the Osserman Inequality},
   Tsukuba J. Math. {\bfseries 22} (1998), 575--588.
\bibitem[Yu3]{Yu3}
   \bysame,
   {\itshape  Surfaces of constant mean curvature one in the hyperbolic
    three-space with irregular ends},
   T\^ohoku Math. J. {\bfseries 53} (2001) 305--318.
\end{thebibliography}
\end{document}